\documentclass[11pt]{article}

\usepackage{epsfig,epsf,fancybox}
\usepackage{amsmath}
\usepackage{mathrsfs}
\usepackage{amssymb}
\usepackage{graphicx}
\usepackage{color}
\usepackage{multirow}
\usepackage{paralist}
\usepackage{verbatim}
\usepackage{galois}
\usepackage{algorithm}
\usepackage{algorithmic}
\usepackage{boxedminipage}
\usepackage{booktabs}
\usepackage{accents}
\usepackage{stmaryrd}
\usepackage{subfig}

\newtheorem{theorem}{Theorem}[section]

\newtheorem{lemma}[theorem]{Lemma}

\newtheorem{condition}{Condition}[section]
\newtheorem{definition}{Definition}[section]

\newtheorem{remark}[theorem]{Remark}
\newtheorem{assumption}[theorem]{Assumption}

\newcommand{\BI}{\mathbb{I}}

\,
\,

\newcommand{\x}{\mathbf x}
\newcommand{\y}{\mathbf y}
\newcommand{\s}{\mathbf s}
\newcommand{\sa}{\mathbf a}
\newcommand{\g}{\mathbf g}
\newcommand{\e}{\mathbf e}
\newcommand{\z}{\mathbf z}
\newcommand{\w}{\mathbf w}
\newcommand{\argmin}{\mathop{\rm argmin}}

\newcommand{\KCal}{\mathcal{K}}

\newcommand{\QCal}{\mathcal{Q}}
\newcommand{\SCal}{\mathcal{S}}

\newcommand{\br}{\mathbb{R}}

\newcommand{\ba}{\begin{array}}
\newcommand{\ea}{\end{array}}
\newcommand{\ACal}{\mathcal{A}}

\newcommand{\etal}{{\it et al.\ }}

\title{A Unified Scheme to Accelerate Adaptive  Cubic Regularization and Gradient Methods for Convex Optimization }

\author{
	Bo JIANG
	\thanks{Research Center for Management Science and Data Analytics, School of Information Management and Engineering, Shanghai University of Finance and Economics, Shanghai 200433, China. Email: isyebojiang@gmail.com. } \and
	Tianyi LIN
	\thanks{Department of Industrial Engineering and Operations Research, UC Berkeley, Berkeley, CA 94720, USA. Email: darren\_lin@berkeley.edu} \and
	Shuzhong ZHANG
	\thanks{Department of Industrial and Systems Engineering, University of Minnesota, Minneapolis, MN 55455, USA. Email: zhangs@umn.edu.}}

\begin{document}
	\maketitle
	
	\begin{abstract}
		In this paper we propose a unified two-phase scheme for convex optimization to accelerate: (1) the adaptive cubic regularization methods with exact/inexact Hessian matrices, and (2) the adaptive gradient method, 
		without any knowledge of the Lipschitz constants for the gradient or the Hessian. 
		This is achieved by tuning the parameters used in the algorithm \textit{adaptively} in its process of progression, which can be viewed as a 
		relaxation over the existing algorithms in the literature. Under 
		the assumption that the sub-problems can be solved approximately, we establish overall iteration complexity bounds for three newly proposed algorithms to obtain an $\epsilon$-optimal solution. Specifically, 
		we show that the adaptive cubic regularization methods with the exact/inexact Hessian matrix both achieve an iteration complexity in the order of $O\left( 1 / \epsilon^{1/3} \right)$, which matches that of the original accelerated cubic regularization method presented in \cite{Nesterov-2008-Accelerating} assuming the availability of the exact Hessian information and the Lipschitz constants, and the global solution of the sub-problems.
		Under the same two-phase adaptive acceleration framework, the  gradient method achieves an iteration complexity in the order of $O\left( 1 / \epsilon^{1/2} \right)$, which is known to be best possible (cf.~\cite{Nesterov-2013-Introductory}).
		Our numerical experiment results show 
		a clear effect of 
		acceleration displayed in the adaptive Newton's method with cubic regularization
		on a set of regularized logistic regression instances. 
	\end{abstract}
	
	\vspace{0.25cm}
	
	
	\vspace{0.25cm}
	
	
	\section{Introduction}
	\subsection{Motivations}\label{Section1:Motivation}
	We consider the following generic unconstrained optimization model: 
	\begin{equation}\label{Prob:main}
	f^* := \min_{\x\in\br^d} \ f(\x),
	\end{equation}
	where $f: \br^d\rightarrow\br$ is \textit{smooth} and \textit{convex}, and $f^*>-\infty$. During the past 
	decades, various classes of optimization algorithms for solving \eqref{Prob:main} have been developed and carefully analyzed; see \cite{Luenberger-1984-Linear, Nocedal-2006-Numerical, Nesterov-2013-Introductory} for detailed information and references. 
	Two types of concerns often arise in the design of optimization algorithms. First, the high order information (such as the Hessian matrices) maybe expensive to acquire. Second,
	the problem  parameters such as the first and the second order Lipschitz constants are usually hard to estimate. 
	On the other hand, for an optimization algorithms to be effective and practical, they will need to be robust and less dependent on the knowledge of the structure of the problem at hand. In this context, schemes to adaptively adjust the parameters used in the algorithm are desirable, and 
	are likely leading to improve its numerical performances. 
	As an example, researchers in the area of deep learning tend to  
	train their models with adaptive gradient method (see e.g.\ AdaGrad in \cite{Duchi-2011-Adaptive})  due to its robustness and effectiveness (cf.\ \cite{Karparthy-2017-Peak}). In fact, Adam \cite{Kingma-2014-Adam} and RMSProp \cite{Tieleman-2012-Lecture} are recognized as the default solution methods in the deep learning setting.
	
	Another fundamental issue in optimization (as well as in machine learning) is to understand how the classical algorithms (including both the first-order and second-order methods) can be accelerated. Nesterov  \cite{Nesterov-1983-Accelerated} put forward the very first accelerated (optimal in its iteration counts) gradient-based algorithm for convex optimization.
	Recently, a number of adaptively accelerated gradient methods have been proposed; see~\cite{Duchi-2011-Adaptive, Nesterov-2013-Gradient, Lin-2014-Adaptive, Monteiro-2016-Adaptive}. Unfortunately, none of these are {\it fully parameter free}.
	Comparing to their first-order counterpart, 
	investigations on 
	the second-order methods is relatively scarce, 
	as acceleration with the second-order information is much more involved. 
	To the best of our knowledge, 
	\cite{Nesterov-2008-Accelerating, Monteiro-2013-Accelerated} are the only papers that are concerned with accelerating the second-order methods. However, these two algorithms do require the knowledge of some problem (Lipschitz) constants,  

	Indeed, algorithms exhibiting both traits of \textbf{\textit{acceleration}} and \textbf{\textit{adaptation}} have been largely missing in the literature. 
	As a matter of fact, we are unaware of any prior accelerated second-order methods (or even any first-order methods) that are fully independent of the problem constants while maintaining superior 
	theoretical iteration complexity bounds. For instance, the adaptive cubic regularized Newton's method \cite{Cartis-2012-Evaluation} merely achieves an iteration complexity bound of 
	$O\left( 1 /\epsilon^{1/2} \right)$ without acceleration. 
	Thus, a natural question raises:
	\begin{quote}
		\textsf{Can we develop an implementable accelerated cubic regularization method with an iteration complexity lower than $O\left( 1 / \epsilon^{1/2} \right)$?}
	\end{quote}
	
	This paper sets out to present an affirmative answer to the above question. 
	Moreover, the resulting accelerated adaptive cubic regularization algorithm displays an excellent numerical performance in solving a variety of large-scale machine learning models in our experiments.

	\subsection{Related Work}\label{Section1:RelatedWork}
	Nesterov's seminal work \cite{Nesterov-1983-Accelerated} triggered a burst of research on accelerating first-order methods.
	There have been a good deal of recent efforts to understand its nature from other perspectives \cite{Allen-2014-Linear, Bubeck-2015-Geometric, Su-2016-Differential, Wibisono-2016-Variational, Wilson-2016-Lyapunov}, or modify it to account for more general settings \cite{Beck-2009-Fast, Cotter-2011-Better, Lan-2012-Optimal, Drori-2014-Performance, Shalev-2014-Accelerated, Lin-2015-Universal}. Parallel to this, the adaptive gradient methods with the optimal convergence rate have been proposed \cite{Duchi-2011-Adaptive, Nesterov-2013-Gradient, Lin-2014-Adaptive, Monteiro-2016-Adaptive}, and widely used in training the deep neural networks \cite{Kingma-2014-Adam, Tieleman-2012-Lecture}. However, all of these algorithms are not fully parameter-independent.
	Specifically, Duchi \etal \cite{Duchi-2011-Adaptive} needs to tune the step-size $\eta$ and the regularization parameter $\delta$; Lin and Xiao \cite{Lin-2014-Adaptive} and Nesterov \cite{Nesterov-2013-Gradient} require a lower bound on the Lipschitz constant $L_g$ for the gradient; and Monteiro and Svaiter \cite{Monteiro-2016-Adaptive} need an upper bound of $L_g-\mu$, where $\mu$ is a strong convexity parameter.
	

	In terms of the second-order methods (in particular Newton's method), the literature regarding acceleration is quite limited. To the best of our knowledge, Nesterov \cite{Nesterov-2008-Accelerating} 
	is the first along this direction, where the overall iteration complexity for convex optimization was improved from $O\left( 1 / \epsilon^{1/2} \right)$ to $O\left( 1 / \epsilon^{1/3} \right)$ for the cubic regularization for Newton's method \cite{Nesterov-2006-Cubic}. After that, Monteiro and Svaiter \cite{Monteiro-2013-Accelerated} managed to accelerate the Newton proximal extragradient method \cite{Monteiro-2012-Iteration} with an improved iteration complexity of $O\left( 1 / \epsilon^{2/7} \right)$. Moreover, this approach allows a larger stepsize and can even accommodate a non-smooth objective function. Very recently, Shamir and Shiff \cite{Shamir-2017-Oracle} proved that $O\left( 1 / \epsilon^{2/7} \right)$ is actually a lower bound for the oracle complexity of the second-order methods for convex smooth optimization, which implies that the accelerated Newton proximal extragradient method is an optimal second-order method. However, viewed from an implementation perspective, the acceleration second-order scheme in \cite{Nesterov-2008-Accelerating, Monteiro-2013-Accelerated} are not easy to apply in practice. Indeed, Nesterov's method assumes that all the parameters, including the Lipschitz constant for the  Hessian, are known, and the sub-problems with cubic regularization are solved to global optimality; Monteiro and Svaiter's method also assumes the knowledge of the Lipschitz constant of the Hessian. To alleviate this, Cartis \etal incorporated an adaptive strategy into Nesterov's approach \cite{Nesterov-2008-Accelerating}, and further relaxed the criterion for solving each sub-problem while maintaining the convergence properties for both convex \cite{Cartis-2012-Evaluation} and non-convex \cite{Cartis-2011-Adaptive-I, Cartis-2011-Adaptive-II} cases. However, as mentioned earlier, the iteration complexity established in \cite{Cartis-2012-Evaluation} for convex optimization is merely $O\left( 1 / \epsilon^{1/2} \right)$. Furthermore, in \cite{Cartis-2012-Oracle} the same authors  also developed a way to construct an approximation for the Hessian, which significantly reduces the per-iteration cost. There are other recent works on approximate cubic regularization for Newton's method.
	For instance, Carmon and Duchi \cite{Carmon-2016-Gradient} and Agarwal \etal \cite{Agarwal-2016-Finding} proposed some variants, where the sub-problem is approximately solved without resorting to 
	Hessian matrix; Kohler and Lucchi \cite{Kohler-2017-SubSample} proposed a uniform sub-sampling strategy to approximate the Hessian in the cubic regularization for Newton's method. However, the approximative Hessian and gradient are constructed based on a priori unknown step which can only be determined after such approximations are formed. Xu \etal \cite{Xu-2017-Newton, Xu-2017-Second} fixed this issue by proposing appropriate uniform and non-uniform sub-sampling strategies to construct Hessian approximations in the trust region context, as well as the cubic regularization for Newton's method.
	

	\subsection{Contributions}\label{Section1:Contribution}
	The contributions of this paper can be summarized as follows.
	We present a unified adaptive accelerating scheme that can be specialized to several optimization algorithms including cubic regularized Newton's method with \textit{exact/inexact} Hessian and gradient method. This can be considered complementary to the current stream of research in two aspects. First, all the accelerated algorithms developed in this paper are parameter-free due to the new {\it fully adaptive}\/ strategies, while only {\it partially adaptive}\/ strategies are observed from other accelerated 
	first-order methods in the literature \cite{Nesterov-2013-Gradient, Lin-2014-Adaptive, Monteiro-2016-Adaptive}.
	Second, it is worth noting that the research efforts on accelerated algorithms have been rather unequally spread between the first-order and second-order methods, 
	with the former receiving a lot more attention. 
	Our results on the adaptive and accelerated cubic regularization for Newton's method contribute as one step towards balancing the studies on the two methods. 
	
	In terms of the convergence rates of our algorithms, for the cubic regularized Newton's method we show that a global convergence rate
	of $O\left( 1 / \epsilon^{1/3} \right)$ holds (Theorem \ref{Thm:AARC-Main}) without assuming any knowledge of the problem parameters. We further prove that, even without the exact Hessian information, the same $O\left( 1 / \epsilon^{1/3} \right)$ rate of convergence (Theorem \ref{Thm:AARCQ-Main}) is still achievable
	for the cubic regularized approximative Newton's method. For the gradient descent method, our adaptive algorthm achieves a convergence rate of $O\left( 1 / \epsilon^{1/2} \right)$ (Theorem \ref{Thm:AAGD-Main}) which matches the optimal rate for the first order methods \cite{Nesterov-2013-Introductory}. When the objective function is strongly convex, the convergence results are also established for these three algorithms accordingly.
	
	For the subproblem in the cubic regularized Newton's method with
	\textit{exact/inexact} Hessian, we only require an approximative solution satisfying \eqref{Eqn:Approx_Subprob}. Note that our approximity measure does not include the usual condition in the form of \eqref{Eqn:Approx2_Subprob}, 
	and thus is weaker than the one used in \cite{Cartis-2011-Adaptive-I}. This relaxation opens up possibilities for other approximation solution methods to solve the subproblem. For instance, Carmon and Duchi \cite{Carmon-2016-Gradient} proposed to use the gradient descent method, and they proved that it works well even when the cubic regularized subproblem is nonconvex. Moreover, such function in our case is strongly convex, and thus the gradient descent subroutine is expected to have a fast (linear) convergence.

	\subsection{Notations and Organization}
	Throughout the paper, we denote vectors by bold lower case letters, e.g., $\x$, and matrices by regular upper case letters, e.g., $X$. The transpose of a real vector $\x$ is denoted as $\x^\top$. For a vector $\x$, and a matrix $X$, $\left\|\x\right\|$ and $\left\|X\right\|$ denote the $\ell_2$ norm and the matrix spectral norm, respectively. $\nabla f(\x)$ and $\nabla^2 f(\x)$ are respectively the gradient and the Hessian of $f$ at $\x$, and $\BI$ denotes the identity matrix. For two symmetric matrices $A$ and $B$, $A \succeq B$ indicates that $A-B$ is symmetric positive semi-definite. The subscript, e.g., $\x_i$, denotes iteration counter. $\log(x)$ denotes the natural logarithm of $x$. The inexact Hessian is denoted by $H(\x)$, but for notational simplicity, we also use $H_i$ to denote the inexact Hessian evaluated at the iterate $\x_i$ in iteration $i$, i.e., $H_i\triangleq H(\x_i)$.
	
	The rest of the paper is organized as follows. In Section \ref{Section2:Notation}, we introduce notations and assumptions used throughout this paper, and present our general framework in Section \ref{Section2:Framework}. Then the specializations to cubic regularized Newton's method with exact/inexact Hessian matrix and gradient descent method
	are presented in Sections \ref{Section:AARC}, \ref{Section:AARCQ} and \ref{Section:AAGD} respectively. In Section \ref{Section5:Experiment}, we present some preliminary numerical results on solving Regularized Logistic Regression, where acceleration of the method based on the adaptive cubic regularization for Newton's method is clearly observed. The details of all the proofs can be found in the appendix.
	

	\section{A Unified Adaptive Acceleration Framework}
	In this section, we first introduce the main definitions and assumptions used in the paper, and then present our unified adaptive acceleration framework.

	\subsection{Assumptions} \label{Section2:Notation}

	Throughout this paper, we refer to the following definition of  $\epsilon$-optimality.
	\begin{definition}
		($\epsilon$-optimality). Given $\epsilon\in\left(0,1\right)$, $\x\in\br^d$ is said to be an $\epsilon$-optimal solution to problem~\eqref{Prob:main}, if
		\begin{equation}\label{result:optimality}
		f(\x) - f(\x^*) \leq \epsilon,
		\end{equation}
		where $\x^*\in\br^d$ is the global optimal solution to problem~\eqref{Prob:main}.
	\end{definition}
	
	To proceed, 
	we make the following standard assumption regarding the gradient and Hessian of the objective function $f$.
	\begin{assumption}\label{Assumption-Objective-Gradient-Hessian}
		The objective function $f(\x)$ in problem~\eqref{Prob:main} is convex and twice differentiable with the gradient and the Hessian being both Lipschitz continuous, i.e., there are $0<L_g, L_h<\infty$ such that for any $\x, \y\in\br^d$ we have 
		\begin{align}
		\left\| \nabla f(\x) - \nabla f(\y)\right\| \leq &  \ L_g\left\|\x-\y\right\|, \label{Def:Lipschitz-Gradient} \\
		\left\| \nabla^2 f(\x) - \nabla^2 f(\y)\right\| \leq &  \ L_h\left\|\x-\y\right\|. \label{Def:Lipschitz-Hessian}
		\end{align}
	\end{assumption}
	
	We also study the problem with a strongly convex objective defined as follows: 
	\begin{definition}\label{Assumption-Strongly-Convex}
		A function $f$ is said to be strongly convex if there is 
		$\mu>0$, such that for any $\x, \y\in\br^d$ we have 
		\begin{equation}\label{Def:Strongly-Convex}
		f(\y) - f(\x) - \left(\y-\x\right)^\top\nabla f(\x) \geq \frac{\mu}{2}\left\|\y-\x\right\|^2.
		\end{equation}
	\end{definition}
	
	\subsection{Framework} \label{Section2:Framework}
	The adaptive acceleration framework is composed of two separate subroutines. Specifically, the framework starts with a \textsf{Simple Adaptive Subroutine (SAS)}, which terminates as soon as one successful iteration is identified. Then, the output of \textsf{SAS} is used as an initial point to run \textsf{Accelerated Adaptive Subroutine (AAS)} until a sufficient number of successful iterations are recorded. The details of our framework are summarized in Table \ref{Scheme: UAA}.
	
	\begin{table}[!t]
		\begin{tabular}{@{}llr@{}}\toprule
			{\bf Begin Phase I:} {\textsf{Simple Adaptive Subroutine (SAS)}}\\
			\quad \textbf{for} $i = 0,1,\cdots $, \textbf{do}\\
			\quad \qquad Construct certain regularized function $m(\x_i,\s,\sigma_i)$ with a regularization parameter $\sigma_i$;\\
			\quad \qquad Compute $\s_i$ by solving $m(\x_i,\s,\sigma_i)$ approximately or exactly;\\
			\quad \qquad {\bf if} iteration $i$ is successful {\bf then}\\
			\quad \qquad \quad Set $\x_{i+1} = \x_i + \s_i $ and update $\sigma_{i+1}$; \\
			\quad \qquad \quad  Record the total number of iterations for {\bf SAS}: $T_1 = i+1$; \\
			\quad \qquad \quad  $\textbf{break}$;\\
			\quad \qquad {\bf else} \\
			\quad \qquad \quad Set $\x_{i+1} = \x_i$, and update $\sigma_{i+1}$. \\
			\quad    \qquad {\bf end if}\\
			\quad \textbf{end for}\\
			\vspace{0.5cm}
			{\bf End Phase I (SAS)} \\
			{\bf Begin Phase II:} {\textsf{Accelerated Adaptive Subroutine (AAS)}}\\
			Set the count of successful iterations $l=1$ and let $\bar{\x}_1 = \x_{T_1}$;\\
			Construct auxiliary function $\psi_{1}(\z, \varsigma_1)$ with some $\varsigma_1>0$, and let $\z_1= \argmin_{\z\in\br^d} \psi_1(\z, \varsigma_1)$, \\ and choose $\y_1=\alpha_1 \bar{\x}_1 + (1 - \alpha_1 )\z_1$; \\
			\quad \textbf{for} $j = 0,1,\cdots $, \textbf{do}\\
			\quad \qquad Construct regularized function $m(\y_l,\s,\sigma_{T_1+j})$ with regularized parameter $\sigma_{T_1+j}$;\\
			\quad \qquad Compute $\s_{T_1+j}$ by solving $m(\y_l,\s,\sigma_{T_1+j})$ approximately or exactly;\\
			\quad \qquad {\bf if} iteration $T_1+j$ is successful {\bf then} \\
			\quad \qquad \quad Update $\sigma_{T_1+j+1}$ and set $\x_{T_1+j+1} = \x_{T_1+j} + \s_{T_1+j} $; \\
			\quad \qquad \quad  Update the count of successful iterations $l = l +1$; \\
			\quad \qquad \quad Update the auxiliary function $\psi_{l}(\z, \varsigma_l)$ by choosing the regularization parameter $\varsigma_l$ automatically; \\
			\quad \qquad \quad  Solve $\z_l= \argmin_{\z \in \br^d} \psi_l(\z)$, let $\bar{\x}_l = \x_{T_1+j+1}$ and $\y_l=\alpha_l \bar{\x}_l + (1 - \alpha_l )\z_l$;\\
			\quad \qquad {\bf else} \\
			\quad \qquad \quad Set $\x_{T_1+j+1} = \x_{T_1+j}$ and update $\sigma_{T_1+j+1}$; \\
			\quad \qquad {\bf end if}\\
			\quad \textbf{end for}\\
			\quad Record the total number of iterations for {\bf AAS}: $T_2 = j+1$. \\
			{\bf End Phase II (AAS)} \\
			\hline
		\end{tabular}
		\caption{Unified Adaptive Acceleration Framework}\label{Scheme: UAA}
	\end{table}
	
	Note that certain adaptive strategies are adopted to tune the regularization parameters in both $m(\x_i,\s,\sigma_i)$ and $\psi_l(\z, \varsigma_l)$ while the acceleration is only installed in \textsf{AAS}, where the tuple $\left(\bar{\x}_l, \y_l, \z_l\right)$ is updated when a successful iteration is identified. In addition, the criteria for identifying the successful iteration in each subroutine are different. When specialized to cubic regularization for Newton's method, \textsf{SAS} can be interpreted as the initialization step based on a modification of {\it adaptive cubic regularization method} proposed in \cite{Cartis-2011-Adaptive-I, Cartis-2011-Adaptive-II}.
	
	For the three algorithms mentioned above, the specific forms of regularized function $m(\x,\s,\sigma)$ are presented in Table \ref{tab:function}, and the iterative update rule for auxiliary function $\psi_l (\z)$ and the accelerating coefficient $\alpha_l$ are presented in Table \ref{tab:auxiliary}. In the rest of the paper, we shall analyze these three specialized algorithms within the framework just introduced. 

	\begin{table}[t]
		\center
		\begin{tabular}{|c|c|} \hline
			Method & $m(\x_i, \s,\sigma_i)$ \\ \hline
			{Algorithm \ref{Algorithm: AARC}} & $f(\x_i) + \s^\top\nabla f(\x_i) + \frac{1}{2}\s^\top \nabla^2 f(\x_i) \s + \frac{1}{3}\sigma_i\left\|\s\right\|^3$ \\ \hline
			{Algorithm \ref{Algorithm: AARCQ}} & $f(\x_i) + \s^\top\nabla f(\x_i) + \frac{1}{2}\s^\top H(\x_i)\s + \frac{1}{3}\sigma_i\left\|\s\right\|^3$ \\ \hline
			{Algorithm \ref{Algorithm: AAGD}} & $f(\x_i) + \s^\top\nabla f(\x_i) + \frac{1}{2}\sigma_i\left\|\s\right\|^2$ \\  \hline
		\end{tabular}
		\caption{Specific choices of $m(\x_i,\s,\sigma_i)$}\label{tab:function}
	\end{table}
	
	\begin{table}[h]
		\center
		\begin{tabular}{|c|c|c|} \hline
			Method & $\psi_l (\z)$ & $\alpha_l$\\ \hline
			{Algorithm \ref{Algorithm: AARC}} & $\psi_{l-1}(\z) + \frac{l(l+1)}{2}\left(f(\bar{\x}_{l-1})+\left(\z - \bar{\x}_{l-1}\right)^\top\nabla f(\bar{\x}_{l-1})\right) + \frac{1}{6}(\varsigma_{l}-\varsigma_{l-1})
			\|\z - \bar{\x}_1\|^3$ & $\frac{l}{l+3} $ \\ \hline
			{Algorithm \ref{Algorithm: AARCQ}} & $\psi_{l-1}(\z) + \frac{l(l+1)}{2}\left(f(\bar{\x}_{l-1})+\left(\z - \bar{\x}_{l-1}\right)^\top\nabla f(\bar{\x}_{l-1})\right) + \frac{1}{6}(\varsigma_{l}-\varsigma_{l-1}) \|\z - \bar{\x}_1\|^3$ & $\frac{l}{l+3}$ \\ \hline
			{Algorithm \ref{Algorithm: AAGD}}& $\psi_{l-1}(\z) + l \left(f(\bar{\x}_{l-1})+\left(\z - \bar{\x}_{l-1}\right)^\top\nabla f(\bar{\x}_{l-1})\right) + \frac{1}{4}(\varsigma_{l}-\varsigma_{l-1})\| \z - \bar{\x}_1\|^2$ & $\frac{l}{l+2}$ \\  \hline
		\end{tabular}
		\caption{Specific choices of $\psi_l (\z)$ and $\alpha_l$} \label{tab:auxiliary}
	\end{table}
	
	\section{Accelerated Adaptive Cubic Regularization with Exact Hessian}\label{Section:AARC}
	As illustrated in Table \ref{tab:function}, we consider the following approximation of $f$ evaluated at $\x_i$ with cubic regularization \cite{Cartis-2011-Adaptive-I, Cartis-2011-Adaptive-II}:
	\begin{equation}\label{prob:AARC}
	m(\x_i,\s,\sigma) = f(\x_i) + \s^\top\nabla f(\x_i) + \frac{1}{2}\s^\top\nabla^2 f(\x_i)\s + \frac{1}{3}\sigma_i\left\|\s\right\|^3,
	\end{equation}
	where $\sigma_i>0$ is a regularized parameter adjusted by the algorithm in the process of iterating. Now we present
	the accelerated adaptive cubic regularization for Newton's method with exact Hessian in Algorithm \ref{Algorithm: AARC}.
	\begin{algorithm}[t]
		\begin{algorithmic} \scriptsize
			\STATE Given $\gamma_2>\gamma_1>1$, $\gamma_3>1$, $\eta>0$ and $\sigma_{\min}>0$.
			Specify $m(\x_i,\s,\sigma_i)$ as in Table \ref{tab:function}.
			Choose $\x_0\in\br^d$, $\sigma_0 \ge \sigma_{\min}$, and $\varsigma_1>0$.
			\STATE \textbf{Begin Phase I:} \textsf{Simple Adaptive Subroutine (SAS)}
			\FOR{$i=0,1,2,\ldots$}
			\STATE \quad Compute $\s_i\in\br^d$ such that $\s_i\approx\argmin_{\s\in\br^d} \ m(\x_i,\s,\sigma_i)$;
			\STATE \quad Compute $\rho_i=f(\x_i+\s_i)-m(\x_i,\s_i, \sigma_i)$.
			\IF{$\rho_i<0$ [successful iteration] }
			\STATE \quad $\x_{i+1}=\x_i+\s_i$, \; $\sigma_{i+1}\in\left[\sigma_{\min}, \sigma_i\right]$;
			\STATE \quad Record the total number of iterations for \textsf{SAS}: $T_1 = i+1$.
			\STATE \quad $\textbf{break}$.
			\ELSE
			\STATE \quad $\x_{i+1}=\x_i$, \; $\sigma_{i+1}\in\left[\gamma_1\sigma_i, \gamma_2\sigma_i\right]$.
			\ENDIF
			\ENDFOR
			\STATE \textbf{End Phase I (SAS).}
			
			\vspace{0.5cm}
			\STATE \textbf{Begin Phase II:} \textsf{Accelerated Adaptive Subroutine (AAS)}\\
			Set the count of successful iterations $l=1$ and let $\bar{\x}_1 = \x_{T_1}$;\\
			Construct $\psi_1(\z) = f(\bar{\x}_1) + \frac{1}{6}\varsigma_1\|\z-\bar{\x}_1\|^3$, and let $\z_1= \argmin_{\z\in\br^d} \psi_1(\z)$, and choose $\y_1=\frac{1}{4}\bar{\x}_1 + \frac{3}{4}\z_1$;
			\FOR{$j=0,1,2\ldots$}
			\STATE Compute $\s_{T_1+j}\in\br^d$ such that $\s_{T_1+j}\approx\argmin_{\s\in\br^d} \ m(\y_l,\s,\sigma_{T_1+j})$, and $\rho_{T_1+j}=-\frac{\s_{T_1+j}^\top\nabla f(\y_l+\s_{T_1+j})}{\left\|\s_{T_1+j}\right\|^3}$;
			\IF{$\rho_{T_1+j}\geq\eta$ [successful iteration]}
			\STATE $\x_{T_1+j+1}=\y_l+\s_{T_1+j}$,\;  $\sigma_{T_1+j+1}\in\left[\sigma_{\min}, \sigma_{T_1+j}\right]$;
			\STATE Set $l=l+1$ and $\varsigma= \varsigma_{l-1}$;
			\STATE Update $\psi_l(\z)$ as illustrated in Table \ref{tab:auxiliary} by using $\varsigma_l=\varsigma$, and compute $\z_l=\argmin_{\z\in\br^d} \ \psi_l(\z)$;
			\WHILE{$\psi_l(\z_l)\geq\frac{l(l+1)(l+2)}{6} f(\bar{\x}_l)$}
			\STATE Set $\varsigma = \gamma_3 \varsigma$, and $\psi_l(\z)=\psi_{l-1}(\z)+\frac{l(l+1)}{2}\left[f(\x_{T_1+j+1})+\left(\z - \x_{T_1+j+1}\right)^\top\nabla f(\x_{T_1+j+1})\right]+\frac{1}{6}(\varsigma-\varsigma_{l-1})\|\z -\bar{\x}_1 \|^3$;
			\STATE Compute $\z_l = \argmin_{\z\in\br^d} \ \psi_l(\z)$.
			\ENDWHILE
			\STATE Set $\varsigma_l=\varsigma$;
			\STATE Let $\bar{\x}_l = \x_{T_1+j+1} $ and $\y_l=\frac{l}{l+3}\bar{\x}_l + \frac{3}{l+3}\z_l$;
			\ELSE
			\STATE $\x_{T_1+j+1}=\x_{T_1+j}$,\, $\sigma_{T_1+j+1}\in\left[\gamma_1\sigma_{T_1+j}, \gamma_2\sigma_{T_1+j}\right]$;
			\ENDIF
			
			\ENDFOR
			\STATE Record the total number of iterations for {\bf AAS}: $T_2 = j+1$.
			\STATE \textbf{End Phase II (AAS)} 
		\end{algorithmic} \caption{Accelerated Adaptive Cubic Regularization for Newton's Method with Exact Hessian}\label{Algorithm: AARC}
	\end{algorithm}

	Note that in each iteration of Algorithm \ref{Algorithm: AARC}, we approximately solve
	\begin{equation*}
	\s_i\approx\argmin_{\s\in\br^d} \ m(\x_i,\s,\sigma_i),
	\end{equation*}
	where $m(\x_i,\s,\sigma_i)$ is defined in \eqref{prob:AARC}
	and the symbol ``$\approx$'' is quantified as follows:
	\begin{condition}\label{Cond:Approx_Subprob}
		We call $\s_i$ to be an approximative solution -- denoted as $\s_i\approx\argmin_{\s\in\br^d} \ m(\x_i,\s,\sigma_i)$ -- for $\min_{\s\in\br^d} \ m(\x_i,\s,\sigma_i)$, if the following holds 
		\begin{equation}\label{Eqn:Approx_Subprob}
		\left\|\nabla m(\x_i, \s_i, \sigma_i)\right\| \leq \kappa_\theta\min\left(1, \left\|\s_i\right\|\right)
		\min\left(\left\|\s_i\right\|, \left\|\nabla f(\x_i)\right\|\right),
		\end{equation}
		where $\kappa_\theta\in\left(0,1\right)$ is a pre-specified  constant.
	\end{condition}
	Note that \eqref{Eqn:Approx_Subprob} is also used as one of the two stopping criteria for solving the subproblem in the original adaptive cubic regularization for Newton's method in \cite{Cartis-2012-Evaluation}. However, the other criterion
	\begin{equation}\label{Eqn:Approx2_Subprob}
	\s_i^\top \nabla f(\x_i) + \s_i^\top \nabla^2 f(\x_i) \s_i + \sigma_i\left\|\s_i\right\|^3 = 0
	\end{equation}
	is not needed in Algorithm \ref{Algorithm: AARC}.
	Another difference is that both criteria for the successful iterations in \textsf{SAS} and \textsf{AAS} of Algorithm \ref{Algorithm: AARC} are different than these used in \cite{Cartis-2012-Evaluation}.
	
	From the standpoint of acceleration, we shall show that Algorithm \ref{Algorithm: AARC} will retain the same iteration complexity of $O\left( 1 / \epsilon^{1/3} \right)$ as for the nonadaptive version of \cite{Nesterov-2008-Accelerating} even when the subproblem is now only solved approximatively. On the surface, under the new scheme  we need to solve an additional cubic subproblem:
	\begin{equation*}
	\z_l=\argmin_{\z\in\br^d} \ \psi_l(\z).
	\end{equation*}
	Fortunately, this problem admits a closed-form solution. In particular, recall that the objective function is obtained by using the updating rule in Table \ref{tab:auxiliary}, and so
	\begin{equation*}
	\psi_l(\z) = \ell_l(\z) + \frac{1}{6}\varsigma_l \|\z-\bar{\x}_1\|^3, \quad l = 1,2,\ldots,
	\end{equation*}
	where $\ell_{l}(\z)$ is a certain linear function of $\z$. By letting
	\begin{equation*}
	\nabla \psi_l(\z) = \nabla\ell_l(\z) +  \frac{1}{2}\varsigma_l \|\z-\bar{\x}_1\| \cdot(\z-\bar{\x}_1) = 0,
	\end{equation*}
	we have $\|\z-\bar{\x}_1\| = \sqrt{\frac{2}{\varsigma_l}\|\nabla \ell_l(\z) \|}$. Since $\ell_l(\z)$ is linear, $\nabla\ell_l(\z)$ is independent of $\z$. Therefore, we have
	\begin{equation*}
	\z_l=\bar{\x}_1-\sqrt{\frac{2}{\varsigma_l \| \nabla\ell_l(\z)\|}} {\nabla\ell_l(\z)}.
	\end{equation*}
	
	\subsection{The Convex Case}
	
	In this subsection, we aim to analyze the theoretical performance of  Algorithm \ref{Algorithm: AARC} when the objective function is convex.
	
	\subsubsection{Sketch of the Proof}\label{ProofOutlines}

	To give a holistic picture of the proof, we sketch some major steps below.

	\begin{center}
		\shadowbox{\begin{minipage}{6in}
				\textbf{Proof Outline:}
				\begin{enumerate}
					\item We denote $T_1$ to be the total number of iterations in \textsf{SAS}. Note that the criterion for the sucessfuly iteration in \textsf{SAS} will be satisfied when $\sigma_i$ is sufficiently large. Then $T_1$ is bounded above by some constant (Lemma \ref{Lemma:AARC-T1}).
					\item We denote $T_2$ by the total number of iterations in \textsf{AAS}, and
					\begin{equation*}
					\SCal = \left\{j\leq T_2: T_1+j \text{ successful iteration}\right\}
					\end{equation*}
					to be the index set of all successful iterations in \textsf{AAS}.
					Then $T_2$ is bounded above by $|\SCal|$ multiplied by some constant (Lemma \ref{Lemma:AARC-T2}).
					\item We denote $T_3$ by the total number of counts successfully updating $\varsigma>0$, and $T_3$ is upper bound by some constant (Lemma \ref{Lemma:AARC-T3}).
					
					\item We relate the objective function to the count of successful iterations in \textsf{AAS} (Theorem \ref{Theorem:AARC-Main}).
					\item Putting all the pieces together, we obtain an iteration complexity result (Theorem \ref{Thm:AARC-Main}).
				\end{enumerate}
		\end{minipage}}
	\end{center}

	\subsubsection{Bound the Iteration Numbers}
	\begin{lemma}\label{Lemma:AARC-T1}
		Letting $\bar{\sigma}_1= \max\left\{\sigma_0, \frac{ \gamma_2 L_h}{2}\right\} > 0$, we have $T_1\leq 1+\frac{2}{\log\left(\gamma_1\right)}\log\left(\frac{\bar{\sigma}_1}{\sigma_{\min}}\right)$.
	\end{lemma}
	\begin{proof}
		We have
		\begin{eqnarray}
		f(\x_i + \s_i) & = & f(\x_i) + \s_i^\top \nabla f(\x_i) + \frac{1}{2} \s_i^\top \nabla^2 f(\x_i) \s_i + \int_{0}^{1} (1 - \tau) \s_i^\top \left[\nabla^2 f(\x_i + \tau \s_i) -  \nabla^2 f(\x_i)\right] \s_i \ d \tau \nonumber \\
		&\le & f(\x_i) + \s_i^\top \nabla f(\x_i) + \frac{1}{2} \s_i^\top \nabla^2 f(\x_i) \s_i  + \frac{ L_h}{6}  \| \s_i \|^3   \nonumber \\
		& = & m(\x_i, \s_i, \sigma_i) + \left(\frac{L_h}{6}-\frac{\sigma_i}{3}\right)\left\|\s_i\right\|^3, \label{Equality-Second-Order-AARC}
		\end{eqnarray}
		where the inequality holds true due to Assumption \ref{Assumption-Objective-Gradient-Hessian}. Therefore, we conclude that
		\begin{equation*}
		\sigma_i\geq\frac{L_h}{2} \quad \Longrightarrow \quad f(\x_i+\s_i)\leq m(\x_i, \s_i, \sigma_i),
		\end{equation*}
		which further implies that $\sigma_i<\frac{L_h}{2}$ for $i\le T_1-2$. Hence,
		\begin{equation*}
		\sigma_{T_1} \le \sigma_{T_1 -1}\le \gamma_2 \sigma_{T_1 -2} \le \frac{ \gamma_2 L_h}{2}.
		\end{equation*}
		By the definition that $\bar{\sigma}_1=\max\left\{\sigma_0, \frac{\gamma_2 L_h}{2}\right\}$, it follows from the construction of Algorithm \ref{Algorithm: AARC} that $\sigma_{\min}\leq\sigma_i$ for all iterations, and $\gamma_1\sigma_i\leq\sigma_{i+1}$ for all unsuccessful iterations. Consequently, we have
		\begin{equation*}
		\frac{\bar{\sigma}_1}{\sigma_{\min}} \geq \frac{\sigma_{T_1}}{\sigma_0} = \frac{\sigma_{T_1}}{\sigma_{T_1-1}} \cdot \prod_{j=0}^{T_1-2} \frac{\sigma_{j+1}}{\sigma_j} \geq \gamma_1^{T_1-1}\left(\frac{\sigma_{\min}}{\bar{\sigma}_1}\right),
		\end{equation*}
		and hence $T_1\leq 1+\frac{2}{\log\left(\gamma_1\right)}\log\left(\frac{\bar{\sigma}_1}{\sigma_{\min}}\right)$.
	\end{proof}
	\begin{lemma}\label{Lemma:AARC-T2}
		Letting $\bar{\sigma}_2 = \max\left\{\bar{\sigma}_1,\frac{\gamma_2 L_h}{2}+\gamma_2\kappa_\theta+\gamma_2\eta\right\}>0$, we have $T_2 \leq \left(1+ \frac{2}{\log\left(\gamma_1\right)} \log
		\left(\frac{\bar{\sigma}_2}{\sigma_{\min}}\right)\right)|\SCal|$.
	\end{lemma}
	\begin{proof}
		We have
		{\small\begin{eqnarray*}
				& & \s_{T_1+j}^\top\nabla f(\y_l+\s_{T_1+j}) \\
				& = & \s_{T_1+j}^\top\left[\nabla f(\y_l + \s_{T_1+j}) - \nabla f(\y_l) - \nabla^2 f(\y_l) \s_{T_1+j}\right] + \s_{T_1+j}^\top\left[ \nabla f(\y_l) + \nabla^2 f(\y_l) \s_{T_1+j}\right] \\
				& \leq & \left\| \nabla f(\y_l+\s_{T_1+j}) - \nabla f(\y_l) - \nabla^2 f(\y_l) \s_{T_1+j}\right\| \left\| \s_{T_1+j}\right\|   + \s_{T_1+j}^\top \left[ \nabla m\left(\y_l, \s_{T_1+j}, \sigma_{T_1+j}\right) - \sigma_{T_1+j} \left\| \s_{T_1+j}\right\| \s_{T_1+j} \right]  \\
				& {\eqref{Eqn:Approx_Subprob} \above 0pt  \leq }  & \left\| \nabla f(\y_l+\s_{T_1+j}) - \nabla f(\y_l) - \nabla^2 f(\y_l) \s_{T_1+j}\right\| \left\| \s_{T_1+j}\right\| - \sigma_{T_1+j} \left\| \s_{T_1+j}\right\|^3 + \kappa_\theta \left\| \s_{T_1+j}\right\|^3 \\
				& = & \left\| \int_{0}^{1}  \left[ \nabla^2 f(\y_l + \tau \cdot \s_{T_1+j}) - \nabla^2 f(\y_l ) \right] \s_{T_1+j} \ d \tau \right\| \left\| \s_{T_1+j}\right\|  - \sigma_{T_1+j} \left\| \s_{T_1+j}\right\|^3 + \kappa_\theta \left\| \s_{T_1+j}\right\|^3 \\
				&  \leq  & \left(\frac{L_h}{2} + \kappa_\theta - \sigma_{T_1+j}\right)\left\|\s_{T_1+j}\right\|^3,
		\end{eqnarray*}}
		\noindent where the last inequality is due to Assumption \ref{Assumption-Objective-Gradient-Hessian}. Then it follows that
		\begin{equation*}
		-\frac{\s_{T_1+j}^\top\nabla f(\y_l+\s_{T_1+j})}{\left\| \s_{T_1+j}\right\|^3} \geq \sigma_{T_1+j} - \frac{L_h}{2} - \kappa_\theta.
		\end{equation*}
		Therefore, we have
		\begin{equation*}
		\sigma_{T_1+j} \geq \frac{L_h}{2}+\kappa_\theta+\eta \quad \Longrightarrow -\frac{\s_{T_1+j}^\top\nabla f(\y_l+\s_{T_1+j})}{\left\| \s_{T_1+j}\right\|^3}\geq\eta ,
		\end{equation*}
		which further implies that
		\begin{equation*}
		\sigma_{T_1+j+1}\leq \sigma_{T_1+j}\leq \gamma_2 \cdot \sigma_{T_1+j-1}\leq \gamma_2 \left(\frac{L_h}{2}+\kappa_\theta+\eta\right),\; \forall \; j \in \SCal.
		\end{equation*}
		Therefore, we can define $\bar{\sigma}_2=\max\left\{\bar{\sigma}_1,\frac{\gamma_2 L_h}{2}+\gamma_2\kappa_\theta+\gamma_2\eta\right\}$, where the term $\bar{\sigma}_1$ accounts for an upper bound of $\sigma_{T_1}$. In addition, it follows from the construction of Algorithm \ref{Algorithm: AARC} that $\sigma_{\min}\leq\sigma_{T_1+j}$ for all iterations, and $\gamma_1\sigma_{T_1+j}\leq\sigma_{T_1+j+1}$ for all unsuccessful iterations. Therefore, we have
		\begin{equation*}
		\frac{\bar{\sigma}_2}{\sigma_{\min}} \geq \frac{\sigma_{T_1+T_2}}{\sigma_{T_1}} = \prod_{j\in\SCal} \frac{\sigma_{T_1+j+1}}{\sigma_{T_1+j}} \cdot \prod_{j\notin\SCal} \frac{\sigma_{T_1+j+1}}{\sigma_{T_1+j}} \geq \gamma_1^{T_2-|\SCal|}\left(\frac{\sigma_{\min}}{\bar{\sigma}_2}\right)^{|\SCal|},
		\end{equation*}
		and hence
		\begin{equation*}
		|\SCal| \le T_2\leq |\SCal|+\frac{\left(|\SCal|+1\right)}{\log\gamma_1} \log\left(\frac{\bar{\sigma}_2}{\sigma_{\min}}\right) \leq \left(1+ \frac{2}{\log\gamma_1} \log\left(\frac{\bar{\sigma}_2}{\sigma_{\min}}\right)\right)|\SCal|.
		\end{equation*}
	\end{proof}
	Before estimating the upper bound of $T_3$, i.e., the total number of the count of successfully updating $\varsigma>0$, we need the following three technical lemmas.
	\begin{lemma}\label{Lemma:General-Second-Order}
		For any $\s\in\br^d$ and $\g\in\br^d$, it holds that
		\begin{equation*}
		\s^\top\g + \frac{1}{3}\sigma\left\|\s\right\|^3 \geq -\frac{2}{3\sqrt{\sigma}}\left\|\g\right\|^{\frac{3}{2}}.
		\end{equation*}
	\end{lemma}
	\begin{proof}
		Denote $\s^*$ as the minimum of $\s^\top\g + \frac{1}{3}\sigma\left\|\s\right\|^3$. The first-order optimality condition gives that
		\begin{equation*}
		\g+\sigma\left\|\s^*\right\|\s^*=0.
		\end{equation*}
		Therefore, we have $(\s^*)^\top\g = -\sigma\left\|\s^*\right\|^3$ and $\left\|\g\right\|=\sigma\left\|\s^*\right\|^2$, and
		\begin{equation*}
		(\s^*)^\top\g+\frac{1}{3}\sigma\left\|\s^*\right\|^3 = -\frac{2}{3}\sigma\left\|\s^*\right\|^3 = -\frac{2}{3\sqrt{\sigma}}\left\|\g\right\|^{\frac{3}{2}}.
		\end{equation*}
	\end{proof}
	
	\begin{lemma} \label{Lemma:AARC-T3-P1}
		Letting $\z_l=\argmin\limits_{\z\in\br^d} \ \psi_l(\z)$, we have $\psi_l(\z) - \psi_l(\z_l) \geq \frac{1}{12}\varsigma_l\left\|\z-\z_l\right\|^3$.
	\end{lemma}
	\begin{proof}
		It suffices to show that
		\begin{equation*}
		\psi_l(\z) - \psi_l(\z_l) - \nabla\psi_l(\z_l)^\top (\z-\z_l) \geq \frac{1}{12}\varsigma_l\left\|\z-\z_l\right\|^3,
		\end{equation*}
		since $\z_l=\argmin_{\z\in\br^d} \ \psi_l(\z)$ and $\nabla\psi_l(\z_l)=0$. Furthermore, observe that $\psi_l(\z)=\ell_l(\z)+d(\z)$ where $\ell_l$ is a linear function and $d(\z)=\frac{\varsigma_l}{6}\left\|\z-\bar{\z}_1\right\|^3$. Therefore, it suffices to show that
		\begin{equation*}
		d(\z) - d(\z_l) - \nabla d(\z_l)^\top (\z-\z_l) \geq \frac{ \varsigma_l} {12}\left\|\z-\z_l\right\|^3,
		\end{equation*}
		since $\ell_l(\z) - \ell_l(\z_l) - \nabla \ell_l(\z_l)^\top (\z-\z_l) = 0$. The conclusion follows from Lemma 4 in \cite{Nesterov-2008-Accelerating} by letting $p=3$.
	\end{proof}
	
	\begin{lemma}\label{Lemma:AARC-T3-P2}
		For each iteration $j$ in the subroutine \textsf{AAS}, if it is successful, we have
		\begin{equation*}
		(1-\kappa_\theta)\left\|\nabla f(\x_{j+1})\right\| \leq \left(\frac{L_h}{2}+\bar{\sigma}_2 + \kappa_\theta L_g\right)\left\|\s_j\right\|^2,
		\end{equation*}
		where $\kappa_\theta\in(0,1)$ is used in Condition~\ref{Cond:Approx_Subprob}.
	\end{lemma}
	\begin{proof}
		We denote $j$-th iteration to be the $l$-th successful iteration, and note $\nabla_{\s} m(\y_l,\s_j,\sigma_j)=\nabla f(\y_l) + \nabla^2 f(\y_l)\s_j + \sigma_j\|\s_j\|\cdot\s_j$. Then we have
		\begin{eqnarray*}
			\left\|\nabla f(\x_{j+1})\right\| & \leq & \left\| \nabla f(\y_l+\s_j) - \nabla_{\s} m(\y_l, \s_j, \sigma_j)\right\| + \left\|\nabla_{\s} m(\y_l, \s_j, \sigma_{j})\right\| \\
			& \leq & \left\| \nabla f(\y_l+\s_j) - \nabla_{\s} m(\y_l, \s_j, \sigma_j)\right\| + \kappa_\theta\cdot\min\left(1, \left\|\s_j\right\|\right) \cdot \left\| \nabla f(\y_l)\right\| \\
			& \leq & \left\|\int_0^1\left(\nabla^2 f(\y_l+\tau\s_j)-\nabla^2 f(\y_l)\right) \s_j \ d\tau\right\| + \sigma_j\left\| \s_j\right\|^2 + \kappa_\theta\cdot\min\left(1, \left\|\s_j\right\|\right) \cdot \left\|\nabla f(\y_l)\right\| \\
			& \leq & \frac{L_h}{2} \left\|\s_j\right\|^2 + \sigma_j\left\| \s_j\right\|^2 + \kappa_{\theta}\cdot \| \s_j \| \cdot \left\|\nabla f(\y_l) - \nabla f(\y_l+\s_j)\right\| + \kappa_{\theta}\left\| \nabla f(\x_{j+1})\right\| \\
			& \leq & \frac{L_h}{2}\left\|\s_j\right\|^2 + \bar{\sigma}_2\left\|\s_j\right\|^2 + \kappa_\theta L_g\left\|\s_j\right\|^2 + \kappa_\theta\left\|\nabla f(\x_{j+1})\right\|,
		\end{eqnarray*}
		where the second inequality holds true due to Condition~\ref{Cond:Approx_Subprob}, and the last two inequality follow from Assumption \ref{Assumption-Objective-Gradient-Hessian}. Rearranging the terms, the conclusion follows.
	\end{proof}
	
	Now we are ready to estimate an upper bound of $T_3$, i.e., the total number of count of successfully updating $\varsigma>0$.
	\begin{lemma}\label{Lemma:AARC-T3}
		We have
		\begin{equation}\label{Inequality:induction}
		\psi_l(\z_l)\ge\frac{l(l+1)(l+2)}{6} f(\bar{\x}_l)
		\end{equation}
		if $\varsigma_l\ge\left(\frac{ L_h+2\bar{\sigma}_2+2\kappa_\theta L_g}{1-\kappa_\theta}\right)^{3} \frac{1}{\eta^2}$, which further implies that
		\begin{equation*}
		T_3 \le \left\lceil \frac{1}{\log\left(\gamma_3\right)}\log \left[\left(\frac{ L_h+2\bar{\sigma}_2+2\kappa_\theta L_g}{1-\kappa_\theta}\right)^{3} \frac{1}{\eta^2\varsigma_1} \right] \right\rceil.
		\end{equation*}
	\end{lemma}
	\begin{proof}
		When $l=1$, it trivially holds true that $\psi_l(\z_l)\ge\frac{l(l+1)(l+2)}{6} f(\bar{\x}_l)$ since $\psi_1(\z_1)=f(\bar{\x}_1)$. As a result, it suffices to show that $\varsigma_l\ge\left(\frac{ L_h+2\bar{\sigma}_2+2\kappa_\theta L_g}{1-\kappa_\theta}\right)^{3} \frac{1}{\eta^2}$ by mathematical induction. Without loss of generality, we assume \eqref{Inequality:induction} holds true for some $l - 1 \ge 1$. 
		Then, it follows from Lemma \ref{Lemma:AARC-T3-P1}, and the construction of $\psi_l(\z)$ that
		\begin{equation*}
		\psi_{l-1}(\z)\geq\psi_{l-1}(\z_{l-1})+\frac{1}{12}\varsigma_{l-1}\left\|\z-\z_{l-1}\right\|^3 \geq \frac{(l-1)l(l+1)}{6} f(\bar{\x}_{l-1}) + \frac{1}{12}\varsigma_{l-1}\left\|\z-\z_{l-1}\right\|^3.
		\end{equation*}
		As a result, we have
		\begin{eqnarray*}
			& & \psi_l(\z_l) \\
			& = & \min_{\z\in\br^d} \ \left\{\psi_{l-1}(\z) + \frac{l(l+1)}{2}\left[ f(\bar{\x}_l)+\left(\z-\bar{\x}_l\right)^\top\nabla f(\bar{\x}_l)\right] + \frac{1}{6}\left(\varsigma_l - \varsigma_{l-1}\right)\|\z-\bar{\x}_1\|^3  \right\} \\
			& \geq & \min_{\z\in\br^d} \ \left\{\frac{(l-1)l(l+1)}{6} f(\bar{\x}_{l-1}) + \frac{1}{12}\varsigma_l\left\|\z-\z_{l-1}\right\|^3 + \frac{l(l+1)}{2}\left[ f(\bar{\x}_l)+\left(\z-\bar{\x}_l\right)^\top\nabla f(\bar{\x}_l)\right]  \right\} \\
			& \geq & \min_{\z\in\br^d} \ \left\{\frac{(l-1)l(l+1)}{6} \left[ f(\bar{\x}_l)+\left(\bar{\x}_{l-1}-\bar{\x}_l\right)^\top\nabla f(\bar{\x}_l)\right] + \frac{1}{12}\varsigma_l\left\|\z- \z_{l-1}\right\|^3\right. \\
			& & \quad \quad \quad \left. + \frac{l(l+1)}{2}\left[ f(\bar{\x}_l)+\left(\z -\bar{\x}_l\right)^\top\nabla f(\bar{\x}_l)\right]  \right\} \\
			& = & \frac{l(l+1)(l+2)}{6} f(\bar{\x}_l) + \min_{\z\in\br^d} \ \left\{ \frac{(l-1)l(l+1)}{6} \left(\bar{\x}_{l-1}-\bar{\x}_l\right)^\top\nabla f(\bar{\x}_l) + \frac{1}{12}\varsigma_l\left\|\z-\z_{l-1}\right\|^3\right. \\
			& & \quad \quad \quad \left. + \frac{l(l+1)}{2}\left(\z-\bar{\x}_l\right)^\top\nabla f(\bar{\x}_l)\right\},
		\end{eqnarray*}
		where the first equality holds since $ \varsigma_l\ge \varsigma_{l-1}$. By the construction of $\y_{l-1}$, we have
		\begin{eqnarray*}
			\frac{(l-1)l(l+1)}{6}\bar{\x}_{l-1} & = & \frac{l(l+1)(l+2)}{6}\cdot\frac{l-1}{l+2}\bar{\x}_{l-1} \\
			& = & \frac{l(l+1)(l+2)}{6}\left(\y_{l-1} - \frac{3}{l+2}\z_{l-1}\right) \\
			& = & \frac{l(l+1)(l+2)}{6}\y_{l-1} - \frac{l(l+1)}{2}\z_{l-1}.
		\end{eqnarray*}
		Combining the above two formulas yields
		\begin{eqnarray*}
			& & \psi_l(\z_l) \\
			& \geq & \frac{l(l+1)(l+2)}{6} f(\bar{\x}_l) + \min_{\z\in\br^d} \ \left\{\frac{l(l+1)(l+2)}{6}\left(\y_{l-1}-\bar{\x}_l\right)^\top\nabla f(\bar{\x}_l) + \frac{1}{12}\varsigma_l\left\|\z-\z_{l-1}\right\|^3\right. \\
			& & \left. + \frac{l(l+1)}{2}\left(\z-\z_{l-1}\right)^\top \nabla f(\bar{\x}_l) \right\}.
		\end{eqnarray*}
		Then, by the criterion of successful iteration in \textsf{AAS} and Lemma \ref{Lemma:AARC-T3-P2}, we have
		\begin{eqnarray*}
			\left(\y_{l-1}-\bar{\x}_l\right)^\top\nabla f(\bar{\x}_l)  & = & -\s_{T_1+j}^\top\nabla f(\y_{l-1}+\s_{T_1+j}) \geq \eta\left\| \s_{T_1+j} \right\|^3 \\
			& \geq & \eta \left(\frac{1-\kappa_\theta}{\frac{L_h}{2}+\bar{\sigma}_2+\kappa_\theta L_g}\right)^{\frac{3}{2}}\left\| \nabla f(\bar{\x}_l)\right\|^{\frac{3}{2}},
		\end{eqnarray*}
		where the $l$-th successful iteration count refers to the $(j-1)$-th iteration count in \textsf{AAS}. Hence, it suffices to establish
		\begin{equation*}
		\frac{l(l+1)(l+2)\eta}{6}\left(\frac{1-\kappa_\theta}{\frac{ L_h}{2}+\bar{\sigma}_2+\kappa_\theta L_g}\right)^{\frac{3}{2}}\left\| \nabla f(\bar{\x}_l)\right\|^{\frac{3}{2}} + \frac{1}{12}\varsigma_l\left\|\z-\z_{l-1}\right\|^3 + \frac{l(l+1)}{2}\left(\z-\z_{l-1}\right)^\top\nabla f(\bar{\x}_l) \geq 0.
		\end{equation*}
		Using Lemma \ref{Lemma:General-Second-Order} and setting $\g=\frac{l(l+1)}{2}\nabla f(\bar{\x}_l)$, $\s=\z-\z_l$, and $\sigma=\frac{1}{4}\varsigma_l$, the above is implied by
		\begin{equation}\label{Cubic-Proof-last-Inequality}
		\frac{l(l+1)(l+2)\eta}{6}\left(\frac{1-\kappa_\theta}{\frac{ L_h}{2}+\bar{\sigma}_2+\kappa_\theta L_g}\right)^{\frac{3}{2}}\geq\frac{4}{3\sqrt{\varsigma_l}}
		\left(\frac{l(l+1)}{2}\right)^{\frac{3}{2}}.
		\end{equation}
		Therefore, the conclusion follows if
		\begin{equation*}
		\varsigma_l \ge \left(\frac{ L_h+2\bar{\sigma}_2+2\kappa_\theta L_g}{1-\kappa_\theta}\right)^{3} \frac{1}{\eta^2}.
		\end{equation*}
	\end{proof}
	
	\subsubsection{Iteration Complexity}
	Recall that $l= 1,2,\ldots$ is the count of successful iterations, and the sequence $\{ \bar{\x}_l, \ l=1,2,\ldots \}$ is updated when a successful iteration is identified. The iteration complexity result is presented in Theorem \ref{Theorem:AARC-Main} and Theorem \ref{Thm:AARC-Main}.
	\begin{theorem}\label{Theorem:AARC-Main}
		The sequence $\{\bar{\x}_l, \ l=1,2,\ldots\}$ generated by Algorithm \ref{Algorithm: AARC} satisfies
		\begin{eqnarray*}
			& & \frac{l(l+1)(l+2)}{6}f(\bar{\x}_l) \leq \psi_l(\z_l) \leq \psi_l(\z) \\
			&\leq& \frac{l(l+1)(l+2)}{6}f(\z) + \frac{L_h+\bar{\sigma}_1}{3}\left\|\z-\x_0\right\|^3 + \frac{2\kappa_\theta (1+\kappa_\theta) L_g^2}{\sigma_{\min}}\|\x_0 - \x^*\|^2+ \frac{1}{6}\varsigma_l\left\|\z-\bar{\x}_1\right\|^3 ,
		\end{eqnarray*}
		where
		\begin{equation*}
		\bar{\sigma}_1= \max\left\{\sigma_0, \frac{ \gamma_2 L_h}{2}\right\} > 0.
		\end{equation*}
	\end{theorem}
	\begin{proof}
		The proof is based on mathematical induction. We postpone the base case of $l=1$ to Theorem \ref{Theorem:AARC-Main-Prep}. Suppose that the theorem is true for some $l\geq 1$. Let us consider the case of $l+1$:
		\begin{eqnarray*}
			\psi_{l+1}(\z_{l+1})  & \le & \psi_{l+1}(\z) \\
			& \leq & \frac{l(l+1)(l+2)}{6}f(\z) + \frac{L_h+\bar{\sigma}_1 }{3}\left\|\z-\x_0\right\|^3+ \frac{1}{6}\varsigma_l\left\|\z-\bar{\x}_1\right\|^3 + \frac{2\kappa_\theta(1+\kappa_\theta)  L_g^2}{\sigma_{\min}}\|\x_0-\x^*\|^2 \\
			& & + \frac{(l+1)(l+2)}{2}\left[ f(\bar{\x}_l)+\left(\z-\bar{\x}_l\right)^\top\nabla f(\bar{\x}_l)\right]+ \frac{1}{6}\left(\varsigma_{l+1}-\varsigma_l\right) \left\|\z-\bar{\x}_1\right\|^3 \\
			& \le & \frac{(l+1)(l+2)(l+3)}{6}f(\z) + \frac{L_h + \bar{\sigma}_1 }{3}\left\|\z-\x_0\right\|^3 + \frac{2\kappa_\theta (1+\kappa_\theta) L_g^2}{\sigma_{\min}}\|\x_0 - \x^*\|^2 + \frac{1}{6}\varsigma_{l+1}\left\|\z-\bar{\x}_1\right\|^3 ,
		\end{eqnarray*}
		where the last inequality is due to convexity of $f(\z)$.
		On the other hand, it follows from the way that $\psi_{l+1}(\z)$ is updated that $\frac{(l+1)(l+2)(l+3)}{6}f(\bar{\x}_{l+1}) \leq \psi_{l+1}(\z_{l+1})$, and thus Theorem~\ref{Theorem:AARC-Main} is proven.
	\end{proof}
	After establishing Theorem~\ref{Theorem:AARC-Main}, the iteration complexity of Algorithm \ref{Algorithm: AARC} readily follows.
	\begin{theorem}\label{Thm:AARC-Main}
		The sequence $\{\bar{\x}_l, \ l=1,2,\ldots\}$ generated by Algorithm \ref{Algorithm: AARC} satisfies that
		\begin{equation*}
		f(\bar{\x}_l)-f(\x^*) \leq \frac{C_1}{l(l+1)(l+2)} \leq \frac{C_1}{l^3},
		\end{equation*}
		where
		\begin{equation*}
		C_1 = \left(2 L_h + 2\bar{\sigma}_1 \right)\|\x_0-\x^* \|^3 + \left(\frac{ L_h+2\bar{\sigma}_2+2\kappa_\theta L_g}{1-\kappa_\theta}\right)^{3} \frac{1}{\eta^2}\|\bar{\x}_1-\x^* \|^3 + \frac{12\kappa_\theta(1+\kappa_\theta) L_g^2}{\sigma_{\min}}\|\x_0-\x^*\|^2.
		\end{equation*}
		The total number of iterations required to find $\bar{\x}_k$ such that $f(\bar{\x}_k) -f(\x^*) \le \epsilon$ is
		{\small\begin{equation*}
			k \leq 1 + \frac{2}{\log(\gamma_1)}\log \left(\frac{\bar{\sigma}_1}{\sigma_{\min}}\right) + \left(1 + \frac{2}{\log(\gamma_1)}\log\left(\frac{\bar{\sigma}_2}{\sigma_{\min}}\right)\right)\left[\left(\frac{C_1}{\epsilon}\right)^{\frac{1}{3}}+1\right] + \left\lceil \frac{1}{\log(\gamma_3)}\log\left[\left(\frac{ L_h+2\bar{\sigma}_2+2\kappa_\theta L_g}{1-\kappa_\theta}\right)^{3} \frac{1}{\eta^2\varsigma_1} \right] \right\rceil,
			\end{equation*}}
		\noindent where
		\begin{equation*}
		\bar{\sigma}_2 = \max\left\{\bar{\sigma}_1,\frac{\gamma_2 L_h}{2}+\gamma_2\kappa_\theta+\gamma_2\eta\right\}>0.
		\end{equation*}
	\end{theorem}
	\begin{proof}
		By Theorem~\ref{Theorem:AARC-Main} and taking $\z=\x^*$ we have
		\begin{equation*}
		\frac{l(l+1)(l+2)}{6}f(\bar{\x}_l) \leq \frac{l(l+1)(l+2)}{6}f(\x^*) + \frac{L_h+\bar{\sigma}_1}{3}\left\|\x^*-\x_0\right\|^3 + \frac{2\kappa_\theta (1+\kappa_\theta) L_g^2}{\sigma_{\min}}\|\x_0-\x^*\|^2 + \frac{\varsigma_l}{6}\left\| \x^*-\bar{\x}_1\right\|^3.
		\end{equation*}
		Rearranging the terms, and combining with Lemmas \ref{Lemma:AARC-T1}, \ref{Lemma:AARC-T2} and \ref{Lemma:AARC-T3} lead to the conclusions.
	\end{proof}
	
	Finally let us go back to prove the base case ($l =1$) of Theorem \ref{Theorem:AARC-Main}.
	\begin{theorem}\label{Theorem:AARC-Main-Prep}
		It holds that
		\begin{equation*}
		f(\bar{\x}_1) \leq \psi_1(\z_1) \leq \psi_1(\z) \leq f(\z) + \frac{L_h + \bar{\sigma}_1 }{3}\left\| \z-\x_0\right\|^3 + \frac{2\kappa_\theta (1+\kappa_\theta) L_g^2}{\sigma_{\min}}\|\x_0 - \x^*\|^2 + \frac{1}{6}\varsigma_1\left\|\z-\bar{\x}_1\right\|^3 .
		\end{equation*}
	\end{theorem}
	\begin{proof}
		By the definition of $\psi_1(\z)$ and the fact that $\bar{\x}_1=\x_{T_1}$, we have
		\begin{equation*}
		f(\bar{\x}_1) = f(\x_{T_1}) = \psi_1(\z_1).
		\end{equation*}
		Furthermore, by the criterion of successful iteration in \textsf{SAS},
		\begin{eqnarray*}
			f(\bar{\x}_1) & = & f(\x_{T_1}) \\
			& \leq & m(\x_{T_1-1}, \s_{T_1-1}, \sigma_{T_1-1}) \\
			& = & \left[ m(\x_{T_1-1}, \s_{T_1-1}, \sigma_{T_1-1}) - m(\x_{T_1-1}, \s^m_{T_1-1}, \sigma_{T_1-1}) \right] + m(\x_{T_1-1}, \s^m_{T_1-1}, \sigma_{T_1-1}),
		\end{eqnarray*}
		where $\s^m_{T_1-1}$ denotes the global minimizer of $m(\x_{T_1-1}, \s, \sigma_{T_1-1})$ over $\br^d$. Since $f$ is convex, so is $m(\x_{T_1-1},\s, \sigma_{T_1-1})$. Therefore, we have
		\begin{eqnarray*}
			& & m(\x_{T_1-1}, \s_{T_1-1}, \sigma_{T_1-1}) - m(\x_{T_1-1}, \s^m_{T_1-1}, \sigma_{T_1-1}) \\
			& \leq & \nabla_{\s} m(\x_{T_1-1}, \s_{T_1-1}, \sigma_{T_1-1})^\top \left(\s_{T_1-1} - \s^m_{T_1-1}\right) \\
			& \leq & \left\| \nabla_{\s} m(\x_{T_1-1}, \s_{T_1-1}, \sigma_{T_1-1}) \right\| \left\|\s_{T_1-1} - \s^m_{T_1-1}  \right\| \\
			& { \eqref{Eqn:Approx_Subprob} \above 0pt  \leq }& \kappa_\theta \left\| \nabla f(\x_{T_1-1}) \right\| \left\| \s_{T_1-1} \right\| \left\| \s_{T_1-1} - \s^m_{T_1-1}  \right\|.
		\end{eqnarray*}
		To bound $\left\|\s_{T_1-1} - \s^m_{T_1-1}\right\|$, we observe that
		\begin{eqnarray*}
			\sigma_{\min}\left\|\s\right\|^3  \leq  \sigma_{T_1-1}\left\|\s\right\|^3 &=& \s^\top \left[ \nabla m \left(\x_{T_1-1}, \s, \sigma_{T_1-1}\right) - \nabla f(\x_{T_1-1}) - \nabla^2 f(\x_{T_1-1}) \s \right] \\
			&\leq & \left\|\s\right\| \left[ \left\| \nabla f(\x_{T_1-1})\right\| +   \left\| \nabla m\left(\x_{T_1-1}, \s, \sigma_{T_1-1}\right)\right\|  \right] \\
			&{ \eqref{Eqn:Approx_Subprob} \above 0pt  \leq }& (1+\kappa_\theta)\left\|\s\right\| \left\| \nabla f(\x_{T_1-1})\right\|,
		\end{eqnarray*}
		where $\s=\s_{T_1-1}$ or $\s=\s^m_{T_1-1}$. Thus, we conclude that
		\begin{equation*}
		\left\|\s_{T_1-1}-\s^m_{T_1-1}  \right\| \leq \left\|\s_{T_1-1}\right\| + \left\|\s^m_{T_1-1} \right\| \leq 2\sqrt{\frac{(1+\kappa_\theta)\left\|\nabla f(\x_{T_1-1})\right\|}{\sigma_{\min}}},
		\end{equation*}
		which combines with Assumption \ref{Assumption-Objective-Gradient-Hessian} implies that
		\begin{eqnarray*}
			m(\x_{T_1-1}, \s_{T_1-1}, \sigma_{T_1-1}) - m(\x_{T_1-1}, \s^m_{T_1-1}, \sigma_{T_1-1}) & \leq & \frac{2\kappa_\theta(1+\kappa_\theta)}{\sigma_{\min}} \left\|\nabla f(\x_{T_1-1})\right\|^2 \\
			& = &  \frac{2\kappa_\theta(1+\kappa_\theta)}{\sigma_{\min}} \left\|\nabla f(\x_{T_1-1}) - \nabla f(\x^*)\right\|^2 \\
			& \leq & \frac{2\kappa_\theta (1+\kappa_\theta) L_g^2}{\sigma_{\min}}\left\| \x_{T_1-1} - \x_*\right\|^2 \\
			& = & \frac{2\kappa_\theta (1+\kappa_\theta) L_g^2}{\sigma_{\min}}\left\| \x_0 - \x_*\right\|^2.
		\end{eqnarray*}
		On the other hand, we have
		\begin{eqnarray*}
			& & m(\x_{T_1-1}, \s^m_{T_1-1}, \sigma_{T_1-1}) \\
			& = & f(\x_{T_1-1}) + (\s^m_{T_1-1})^\top \nabla f(\x_{T_1-1}) + \frac{1}{2} (\s^m_{T_1-1})^\top \nabla^2 f(\x_{T_1-1}) \s^m_{T_1-1} + \frac{1}{3}\sigma_{T_1-1}\left\| \s^m_{T_1-1}\right\|^3 \\
			& \leq & f(\x_{T_1-1}) + (\z-\x_{T_1-1})^\top \nabla f(\x_{T_1-1}) + \frac{1}{2} (\z-\x_{T_1-1})^\top \nabla^2 f(\x_{T_1-1})(\z-\x_{T_1-1}) + \frac{1}{3}\sigma_{T_1-1}\left\|\z-\x_{T_1-1}\right\|^3 \\
			& \leq & f(\z) + \frac{L_h}{6}  \left\|\z-\x_{T_1-1}\right\|^3  + \frac{1}{3}\sigma_{T_1-1} \left\|\z-\x_{T_1-1}\right\|^3  \\
			& \leq & f(\z) + \frac{L_h+\bar{\sigma}_1}{3}\left\|\z-\x_{T_1-1}\right\|^3 \\
			& = & f(\z) + \frac{L_h+\bar{\sigma}_1}{3}\left\|\z-\x_0\right\|^3,
		\end{eqnarray*}
		where the second inequality is due to \eqref{Equality-Second-Order-AARC} and Assumption \ref{Assumption-Objective-Gradient-Hessian}. Therefore, we conclude that
		\begin{eqnarray*}
			\psi_1(\z) & = & f(\bar{\x}_1)  + \frac{1}{6}\varsigma_1\left\| \z-\bar{\x}_1\right\|^3 \\
			&\leq & f(\z) + \frac{L_h+\bar{\sigma}_1}{3}\left\|\z-\x_0\right\|^3 + \frac{2\kappa_\theta (1+\kappa_\theta) L_g^2}{\sigma_{\min}}\|\x_0-\x^*\|^2+ \frac{1}{6}\varsigma_1\left\|\z-\bar{\x}_1\right\|^3 .
		\end{eqnarray*}
	\end{proof}

	\subsection{Strongly Convex Case}
	Next we extend the analysis to the case where the objective function is strongly convex (cf.\ Definition \ref{Assumption-Strongly-Convex}). We further assume the level set of $f(\x)$, $\{\x\in\br^d:\, f(\x)\leq f(\x_0)\}$,
	is bounded and is contained in $\|\x-\x_*\| \le D$. Then according to Lemma 3 in \cite{Nesterov-2008-Accelerating}, we have
	\begin{equation}\label{Strongly-Convex-1}
	\nabla^2 f(\x) \succeq \mu\BI,
	\end{equation}
	and
	\begin{equation}\label{Strongly-Convex-2}
	f(\y) - f(\x) - \left(\y-\x\right)^\top\nabla f(\x) \leq \frac{1}{2 \mu} \left\|\nabla f(\y) - \nabla f(\x) \right\|^2.
	\end{equation}
	We shall prove the improvement of the adaptive acceleration scheme in terms of the constant underlying the linear rate of convergence. To this end, denote $\ACal_m^1(\x)$ ($m\geq 1$) to be the point generated by running $m$ iterations of Algorithm \ref{Algorithm: AARC} with starting point $\x$. Then, generate sequence $\{ \hat{\x}_k,\; k=0, 1,2,\ldots \}$ through the following procedure
	\begin{center}
		\begin{enumerate}
			\item Define
			\begin{eqnarray*}
				m & = & 1+\frac{2}{\log(\gamma_1)}\log\left(\frac{\bar{\sigma}_1}{\sigma_{\min}}\right) + \left(1+\frac{2}{\log(\gamma_1)}\log\left(\frac{\bar{\sigma}_2}{\sigma_{\min}}\right) \right)\left[2\left(\frac{\tau_1 D + \tau_2 }{\mu}\right)^{\frac{1}{3}}+1\right] \\
				& & + \left\lceil \frac{1}{\log(\gamma_3)}\log\left[\left(\frac{ L_h+2\bar{\sigma}_2+2\kappa_\theta L_g}{1-\kappa_\theta}\right)^{3} \frac{1}{\eta^2\varsigma_1} \right] \right\rceil,
			\end{eqnarray*}
			with
			\begin{equation*}
			\tau_1 = 2 L_h + 2\bar{\sigma}_1 + \left(\frac{ L_h+2\bar{\sigma}_2+2\kappa_\theta L_g}{1-\kappa_\theta}\right)^{3} \frac{1}{\eta^2}\quad \mbox{and} \quad \tau_2 = \frac{12\kappa_\theta (1+\kappa_\theta) L_g^2}{\sigma_{\min}}.
			\end{equation*}
			\item Set $\hat{\x}_0\in\br^d$.
			\item For $k\geq 0$, iterate $\hat{\x}_k=\ACal_m^1(\hat{\x}_{k-1})$.
		\end{enumerate}
	\end{center}
	The linear convergence of $\{\hat{\x}_k,\; k=0, 1,2,\ldots \}$ is presented in the following theorem.
	\begin{theorem}\label{Theorem:AARC-Linear-Main}
		Suppose the sequence $\{\hat{\x}_k,\; k=0,1,2,\ldots \}$ is generated by the procedure above. For $k \ge O\left(\log\left(\frac{1}{\epsilon}\right)\right)$ we have $f(\hat{\x}_k) -f(\x^*) \le \epsilon$. Specifically, the total number of iterations
		required to find such solution is $O\left(\sqrt[3]{\max\left\{\frac{L_g}{\mu}, \frac{L_h}{\mu}\right\}}\, \log\left(\frac{1}{\epsilon}\right)\right)$.
	\end{theorem}
	\begin{proof}
		By Theorem \ref{Theorem:AARC-Main}, we have
		\begin{eqnarray*}
			& & f(\hat{\x}_{k+1}) - f(\x^*) \\
			& \leq & \frac{1}{m^3} \left[ \left(2 L_h + 2\bar{\sigma}_1 + \left(\frac{ L_h+2\bar{\sigma}_2+2\kappa_\theta L_g}{1-\kappa_\theta}\right)^{3} \frac{1}{\eta^2} \right)\| \hat{\x}_k - \x^* \|^3 + \frac{12\kappa_\theta (1+\kappa_\theta) L_g^2}{\sigma_{\min}}\| \hat{\x}_k - \x^*\|^2 \right]
		\end{eqnarray*}
		where the number of successful iteration $m =  \left( (\tau_1 D +\frac{ \tau_2 }{\sigma_{\min}})/ {\frac{\mu}{8}} \right)^{1/3} \ge \left( (\tau_1 \left\|\x_k - \x^*\right\|  +\frac{ \tau_2 }{\sigma_{\min}})/ {\frac{\mu}{8}} \right)^{1/3}$. Combining this with \eqref{Def:Strongly-Convex} implies that
		\begin{equation*}
		f(\hat{\x}_{k+1}) - f(\x^*) \leq  \frac{\mu}{8}\left\| \hat{\x}_k - \x^*\right\|^2 \leq \frac{1}{4}\left(f(\hat{\x}_k) - f(\x^*)\right),
		\end{equation*}
		which proves the first part of the conclusion. Then the total iteration is
		$$m \cdot \log\left(\frac{1}{\epsilon}\right) = O\left(\sqrt[3]{\max\left\{\frac{L_g}{\mu}, \frac{L_h}{\mu}\right\}}\log\left(\frac{1}{\epsilon}\right)\right),$$
		where we want to explore how the iteration complexity dependent on the conditional number $\frac{ L_g}{\mu}$ and $\frac{L_h}{\mu}$, and the Lipschitz parameters $L_g$ and $L_h$ that are not coupling with $\mu$ are treated as constants.
	\end{proof}
	\begin{remark}
		Remark that comparing to \cite{Cartis-2011-Adaptive-I}, the accelerated scheme has improved the dependence of the conditional number from $O\left( \sqrt{\cdot} \ \right)$ to $O \left( \sqrt[3]{\cdot} \ \right)$.
	\end{remark}
	Furthermore, when the objective function is strongly convex, the local quadratic convergence is retained by our adaptive scheme even without solving the cubic sub-problem exactly if we set $0<\kappa_\theta\leq\frac{\mu}{2}$. Indeed, we can construct sequence $\{\w_l ,\; l=1,2,\ldots \}$ such that $\w_{l+1} = \w_l + \bar{\s}_l$ and $\bar{\s}_l$ is obtained by running the subroutine \textsf{SAS} of Algorithm \ref{Algorithm: AARC} with starting point $\w_l$. Then it holds that
	\begin{eqnarray}
	f(\w_l) - f(\w_{l+1}) & \geq & f(\w_l)  - m(\w_l, \bar{\s}_l, \bar{\sigma}_l) \nonumber \\
	& = & -\bar{\s}_l^\top \nabla f(\w_l) - \frac{1}{2} \bar{\s}_l \top \nabla^2 f(\w_l) \bar{\s}_l - \frac{\bar{\sigma}_l}{3}\left\| \bar{\s}_l \right\|^3  \nonumber \\
	& = & -\bar{\s}_l^\top \nabla  m(\w_l, \bar{\s}_l, \bar{\sigma}_l) + \frac{1}{2} \bar{\s}_l^\top \nabla^2 f(\w_l) \bar{\s}_l + \frac{2\bar{\sigma}_l}{3}\left\| \bar{\s}_l \right\|^3  \nonumber \\
	& { \eqref{Eqn:Approx_Subprob} \above 0pt  \geq }  & \frac{1}{2} \bar{\s}_l^\top \nabla^2 f(\w_l) \bar{\s}_l - \kappa_\theta \left\|\bar{\s}_l\right\|^2 \nonumber \\
	& { \eqref{Strongly-Convex-1} \above 0pt  \geq }& \frac{\mu-2\kappa_\theta}{2}\left\| \bar{\s}_l\right\|^2 \nonumber  \\
	& {\text{Lemma }\ref{Lemma:AARC-T3-P2} \above 0pt  \geq } & \frac{(\mu-2\kappa_\theta)(1-\kappa_\theta)}{2(\frac{ L_h}{2}+\bar{\sigma}_2+\kappa_\theta L_g)}\left\| \nabla f(\w_{l+1}) \right\| \nonumber \\
	& { \eqref{Strongly-Convex-2} \above 0pt  \geq } & \frac{(\mu-2\kappa_\theta)(1-\kappa_\theta)}{2(\frac{L_h}{2}+\bar{\sigma}_2+\kappa_\theta L_g)} \sqrt{2 \mu( f(\w_{l+1})-f(\x^*))}. \nonumber
	\end{eqnarray}
	Hence,
	\begin{equation*}
	f(\w_{l+1}) - f(\x^*) \leq \frac{2(\frac{L_h}{2}+\bar{\sigma}_2+\kappa_\theta L_g)^2}{\mu(\mu-2\kappa_\theta)^2(1-\kappa_\theta)^2} \left(f(\w_l) - f(\w_{l+1})\right)^2 \leq \frac{2(\frac{L_h}{2}+\bar{\sigma}_2 +\kappa_\theta L_g)^2}{\mu(\mu-2\kappa_\theta)^2(1-\kappa_\theta)^2} \left(f(\w_l) - f(\x^*)\right)^2,
	\end{equation*}
	and the region of quadratic convergence is given by
	\begin{equation*}
	\QCal = \left\{\w\in\br^d: f(\w)-f(\x^*)\leq\frac{\mu(\mu-2\kappa_\theta)^2(1-\kappa_\theta)^2}{2(\frac{L_h}{2}+\bar{\sigma}_2+\kappa_\theta L_g)^2}\right\} .
	\end{equation*}
	The above discussion suggests that we can first run Algorithm \ref{Algorithm: AARC} until the generated sequence fall into the local quadratic convergence region $\QCal$, and then switch back and stick to \textsf{SAS} by allowing performing multiple successful iterations. This way, one would still benefit from the accelerated global convergence rate before local quadratic convergence becomes effective.
	
	\section{Accelerated Adaptive Cubic Regularization with Inexact Hessian} \label{Section:AARCQ}
	In this section, we study the scenario where the Hessian information is not available; instead, an approximation is used, based on the gradient information. Indeed, as illustrated in Table \ref{tab:function}, we consider the following approximation of $f$ evaluated at $\x_i$ with cubic regularization:
	\begin{equation}\label{prob:AARCQ}
	m(\x_i,\s,\sigma) = f(\x_i) + \s^\top\nabla f(\x_i) + \frac{1}{2}\s^\top H(\x_i) \s + \frac{1}{3}\sigma_i\left\|\s\right\|^3,
	\end{equation}
	where $\sigma_i>0$ is a regularized parameter, and $H(\x_i)$ is an approximation of the Hessian $\nabla^2 f(\x_i)$, i.e., the inexact Hessian. In particular, the inexact Hessian $H(\x_i)$ can be computed by first computing $d$ forward gradient differences at $\x_i$ with stepsize $h_i\in\br$,
	\begin{equation*}
	A_i = \left[\frac{\nabla f(\x_i+h_i \e_1)-\nabla f(\x_i)}{h_i}, \ldots, \frac{\nabla f(\x_i+h_i \e_d)-\nabla f(\x_i)}{h_i}\right],
	\end{equation*}
	symmetrizing the resulting matrix: $\hat H(\x_i) = \frac{1}{2}\left(A_i+A_i^\top\right)$ and then further adding an constant multiple of identity matrix to $\hat H(\x_i)$: $H(\x_i) = \hat H(\x_i) + \kappa_c h_i \BI$,
	where $\e_j$ is the $j$-th vector of the canonical basis.
	 It is well known in \cite{Nocedal-2006-Numerical} that, for some constant $\kappa_{e}>0$, we have
	$\left\| \hat H(\x_i) - \nabla^2 f(\x_i)\right\|\leq \kappa_{e} h_i$. Consequently, it holds that	
	\begin{equation}\label{condition:finite_approx}
	\left\| H(\x_i) - \nabla^2 f(\x_i)\right\|\leq (\kappa_{e} + \kappa_c) h_i.
	\end{equation}
	That is to say, the gap between exact and inexact Hessian can be bounded by a multiple of the stepsize $h_i$. This together with Algorithm 4.1 in \cite{Cartis-2012-Oracle} inspires us to design a procedure to search a pair of $\left(h_i, \s_i\right)$ such that, for some $\kappa_{hs}>0$,
	\begin{equation}\label{condition:stepsize}
	h_i\leq \kappa_{hs}\left\|\s_i\right\|.
	\end{equation}
	Combining \eqref{condition:finite_approx} and \eqref{condition:stepsize} yields that
	\begin{equation}\label{Hessian-Approximation}
	\left\| H(\x_i) - \nabla^2 f(\x_i)\right\|\leq  (\kappa_{e} + \kappa_c) \kappa_{hs}\left\|\s_i\right\|.
	\end{equation}
	Moreover, since $f$ is convex, we set $\kappa_c \ge \kappa_e$ such that
	\begin{equation}\label{Hessian-Convex}
	H(\x_i) = \hat H(\x_i) + \kappa_c h_i \BI \succeq \nabla^2 f(\x_i) - \kappa_e h_i \BI + \kappa_c h_i \BI \succeq0.
	\end{equation}
	Now we propose the accelerated adaptive cubic regularization of Newton's method with inexact Hessian in Algorithm \ref{Algorithm: AARCQ}. In each iteration we instead approximately solve
	\begin{equation*}
	\s_i\approx\argmin_{\s\in\br^d} \ m(\x_i,\s,\sigma_i),
	\end{equation*}
	where $m(\x_i,\s,\sigma_i)$ is defined in \eqref{prob:AARCQ} and the symbol ``$\approx$'' is quantified in Condition \ref{Cond:Approx_Subprob}, and \eqref{Hessian-Approximation} is a key property that will be used in the iteration complexity analysis for Algorithm \ref{Algorithm: AARCQ}.

	\begin{algorithm}[!p]
		\begin{algorithmic} \scriptsize
			\STATE Given $\gamma_2>\gamma_1>1$, $\gamma_3>1$, $\gamma_4\in (0,1)$,  and $\sigma_{\min} > 0$.  Specify $m(\x_i,\s,\sigma_i)$ as in Table \ref{tab:function}.\\  Choose $\x_0\in\br^d$, $\sigma_0 \geq \sigma_{\min}$,  $h_{0,0}\in\left(0,1\right]$, and $\varsigma_1>0$.
			\STATE \textbf{Begin Phase I:} \textsf{Simple Adaptive Subroutine (SAS)}
			\FOR{$i=0,1,2,\ldots$}
			\FOR{$k=0,1,2,\ldots$}
			\STATE Compute $H_k(\x_i)$ using the finite difference with stepsize $h_{i,k}$ and the iterate $\x_i$;
			\STATE Compute $\s_{i,k}\in\br^d$ such that $\s_{i,k}\approx\argmin_{\s\in\br^d} \ m(\x_i,\s,\sigma_i)$ with the inexact Hessian $H_k(\x_i)$.
			\IF{$h_{i,k} > \kappa_{hs}\left\| \s_{i,k}\right\|$}
			\STATE $h_{i,k+1}=\gamma_4 h_{i,k}$;
			\ELSE
			\STATE $\s_i=\s_{i,k}$ and $h_i=h_{i,k}$;
			\STATE $\textbf{break}$.
			\ENDIF
			\ENDFOR
			\STATE Let $h_{i+1, 0} = h_i$ and compute $\rho_i=f(\x_i+\s_i)-m(\x_i, \s_i, \sigma_i)$.
			\IF{$\rho_i<0$ [successful iteration] }
			\STATE Set $\x_{i+1}=\x_i+\s_i$ and choose $\sigma_{i+1}\in\left[\sigma_{\min}, \sigma_i\right]$;
			\STATE Record the total number of iterations of \textsf{SAS}: $T_1 = i+1$;
			\STATE $\textbf{break}$
			\STATE Set $\x_{i+1}=\x_i$, and choose $\sigma_{i+1}\in\left[\gamma_1\sigma_i, \gamma_2\sigma_i\right]$.
			\ENDIF
			\IF {$\| \nabla f(\x_{i+1}) \| < \epsilon$}
			\STATE $\textbf{break and skip Phase II}$.
			\ENDIF
			\ENDFOR
			\STATE \textbf{End Phase I:} \textsf{Simple Adaptive Subroutine}
			\vspace{0.3 cm}
			
			\STATE \textbf{Begin Phase II:} \textsf{Accelerated Adaptive Subroutine (AAS)}\\
			\STATE Set the count of successful iterations $l=1$ and let $\bar{\x}_1 = \x_{T_1}$.
			\STATE Construct $\psi_1(\z) = f(\bar{\x}_1) + \frac{1}{6}\varsigma_1\|\z -\bar{\x}_1 \|^3$, and let $\z_1= \argmin_{\z\in \br^d} \psi_1(\z)$, and choose $\y_1=\frac{1}{4}\bar{\x}_1 + \frac{3}{4}\z_1$. \\
			\FOR{$j=0,1,2\ldots$}
			\FOR{$k=0,1,2,\ldots$}
			\STATE Compute $H_k(\y_l)$ using the finite difference with stepsize $h_{T_1+j,k}$ and the iterate $\y_l$.
			\STATE Compute $\s_{T_1+j, k}\in\br^d$ such that $\s_{T_1+j, k}\approx\argmin_{\s\in\br^d} \ m(\y_l,\s,\sigma_{T_1+j})$ with the inexact Hessian $H_k(\y_l)$.
			\IF{$h_{T_1+j,k} > \kappa_{hs}\left\| \s_{T_1+j,k}\right\|$}
			\STATE $h_{T_1+j,k+1}=\gamma_4 h_{T_1+j,k}$;
			\ELSE
			\STATE $\s_{T_1+j}=\s_{T_1+j,k}$ and $h_{T_1+j}=h_{T_1+j,k}$;
			\STATE $\textbf{break}$.
			\ENDIF
			\ENDFOR
			\STATE Set $h_{T_1+j+1, 0} = h_{T_1+j}$, and compute 		$\rho_{T_1+j}=-\frac{\s_{T_1+j}^\top\nabla f(\y_l+\s_{T_1+j})}{\left\|\s_{T_1+j}\right\|^3}$;
			\IF{$\rho_{T_1+j}\geq\eta$ [successful iteration]}
			\STATE Let $\x_{T_1+j+1}=\y_l+\s_{T_1+j}$ and choose $\sigma_{T_1+j+1}\in\left[\sigma_{\min}, \sigma_{T_1+j}\right]$;
			\STATE Set $l=l+1$ and $\varsigma= \varsigma_{l-1}$;
			\STATE Update $\psi_l(\z)$ as illustrated in Table \ref{tab:auxiliary} by using $\varsigma_l=\varsigma$, and compute $\z_l=\argmin_{\z\in\br^d} \ \psi_l(\z)$.
			\WHILE{$\psi_l(\z_l) \geq \frac{l(l+1)(l+2)}{6} f(\bar{\x}_l) $}
			\STATE Set $\varsigma=\gamma_3\varsigma$, and $\psi_l(\z)=\psi_{l-1}(\z) + \frac{l(l+1)}{2}\left[f(\x_{T_1+j+1})+\left(\z-\x_{T_1+j+1}\right)^\top\nabla f(\x_{T_1+j+1})\right] + \frac{1}{6}\left(\varsigma-\varsigma_{l-1}\right)\|\z-\bar{\x}_1\|^3$;
			\STATE Compute $\z_l=\argmin_{\z\in\br^d} \ \psi_l(\z)$.
			\ENDWHILE
			\STATE Let $\varsigma_l=\varsigma$, $\bar{\x}_l = \x_{T_1+j+1} $ and $\y_l=\frac{l}{l+3}\bar{\x}_l + \frac{3}{l+3}\z_l$.
			\ELSE
			\STATE Let $\x_{T_1+j+1}=\x_{T_1+j}$, $\sigma_{T_1+j+1}\in\left[\gamma_1\sigma_{T_1+j}, \gamma_2\sigma_{T_1+j}\right]$;
			\ENDIF
			\IF {$\| \nabla f(\x_{T_1+j+1}) \| < \epsilon$}
			\STATE $\textbf{break}$.
			\ENDIF
			\ENDFOR
			\STATE Record the total number of iterations of \textsf{AAS}: $T_2 = j+1$.
			\STATE \textbf{End Phase II:} \textsf{Accelerated Adaptive Subroutine}
		\end{algorithmic} \caption{Accelerated Adaptive Cubic Regularization for Newton's Method with Inexact Hessian}\label{Algorithm: AARCQ}
	\end{algorithm}
	\subsection{The Convex Case}
	In this subsection, we aim to analyze the theoretical performance of  Algorithm \ref{Algorithm: AARCQ}.
	The main difference between Algorithm \ref{Algorithm: AARCQ} and Algorithm \ref{Algorithm: AARC} is an extra inner loop to update $\left\{h_{i,k}, i, k=0,1,2,\ldots \right\}$. We denote $T_4$ by the total number of the successful count of updating the sequence $\left\{h_{i,k},i, k=0,1,2,\ldots \right\}$ in the inner loop. Thus the road map for proving the iteration complexity of Algorithm \ref{Algorithm: AARCQ} is similar to that of Algorithm \ref{Algorithm: AARC} presented in Section \ref{ProofOutlines} except for the bounding of $T_4$. Therefore, we only estalish the bound for $T_4$ and postpone the rest of the proofs to the appendix. Since $\left\{h_{i,k},i, k=0,1,2,\ldots \right\}$ is monotonically decreasing and $h_{i+1,0}=h_{i}$ where $h_i$ is the final output in the last inner loop, it suffices to estimate the lower bound of the sequence $\left\{h_{i,k},i, k=0,1,2,\ldots \right\}$.
	\begin{lemma}\label{Lemma:AARCQ-T4}
		When $\epsilon$ is sufficiently small, the total number of iterations $T_4$ in the inner loop can not exceed
		\begin{equation*}
		\left\lceil -\frac{1}{\log(\gamma_4)} \log\left[\frac{({L_g + (\kappa_{e}+ \kappa_{c} )\kappa_{hs} + \bar{\sigma}_2 })h_{0,0}}{(1-\kappa_\theta) \kappa_{hs}}\cdot\frac{1}{\epsilon}\right]  \right\rceil
		\end{equation*}
	\end{lemma}
	\begin{proof}
		Note that before Algorithm \ref{Algorithm: AARCQ} terminates, we always have $\|\nabla f(\x_{r})\| \ge \epsilon$ for 
		any iterate $\x_r$ in the process except the last one and $\left\{h_{i,k},i, k=0,1,2,\ldots \right\}$ is not updated for the last iterate.
		We let $\x_j$ be the second last iterate before termination with $\s_j=\s_{j,k}$ and stepsize $h_{j,k}$ for Hessian approximation.
		According to Condition~\ref{Cond:Approx_Subprob}, we have
		$$
		\left\| \nabla f(\x_j) +H_j \s_{j,k} + \sigma_j \| \s_{j,k}\|\cdot\s_{j,k} \right\| \le \kappa_{\theta} \min(1,\|\s_{j,k}\|\cdot \|f(\x_j)\| ),
		$$
		which impies that
		$$
		\kappa_{\theta} \| \nabla f (\x_i) \| \ge \left\| \nabla f(\x_j) \| - \| H_j \s_{j,k} + \sigma_j \| \s_{j,k}\|\cdot\s_{j,k} \right\|.
		$$
		As a result,
		\begin{eqnarray*}
			(1 - \kappa_{\theta}) \| \nabla f (\x_j) \| &\le& \left\| H_j \s_{j,k} + \sigma_j \| \s_{j,k}\|\cdot\s_{j,k} \right\| \nonumber \\
			&\le& \| H_j - \nabla^2 f(\x_j)\|\cdot\| \s_{j,k}\| + \| \nabla^2 f(\x_j)\| \cdot \| \s_{j,k}\|+ \sigma_j \| \s_{j,k}\|^2 \nonumber \\
			& \le& (\kappa_{e} + \kappa_{c}) \kappa_{hs}\left\|\s_{j,k}\right\|^2 + L_{g}\left\|\s_{j,k}\right\| + \sigma_j \| \s_{j,k}\|^2.
		\end{eqnarray*}
		where the third inequality is due to mean value theorem and \eqref{Hessian-Approximation}. Consequently,
		\begin{equation*}
		\epsilon \le \| \nabla f (\x_j) \| \le \frac{L_g + (\kappa_{e}+ \kappa_{c} )\kappa_{hs}\| \s_{j,k}\| + \bar{\sigma}_2  \|\s_{j,k}\| }{1 - \kappa_{\theta} } \|\s_{j,k}\|
		\end{equation*}
		with $\bar{\sigma}_2$ is defined in Lemma~\ref{Lemma:AARCQ-T2}, and thus we have
		\begin{equation}\label{h-inner-inq}
		\min\left\{1, \; \frac{\epsilon(1 - \kappa_{\theta})}{L_g + (\kappa_{e}+ \kappa_{c} )\kappa_{hs} + \bar{\sigma}_2 } \right\} \le \|\s_{j,k}\|
		\end{equation}	
		That is $\|\s_{j,k}\|$ has a constant lower bound. Since $\left\{h_{i,k},i, k=0,1,2,\ldots \right\}$ is a monotonically decreasing sequence, $h_{i,k}$ will not be updated as long as $h_{0,0}\gamma_4^{T_4} \le \kappa_{hs}\left\|\s_{j,k}\right\|$, and 
		according to \eqref{h-inner-inq} with a sufficiently small $\epsilon$ this can be achieved by letting
		$$
		T_4 =  \left\lceil -\frac{1}{\log(\gamma_4)} \log\left[\frac{({L_g + (\kappa_{e}+ \kappa_{c} )\kappa_{hs} + \bar{\sigma}_2 })h_{0,0}}{(1-\kappa_\theta) \kappa_{hs}}\cdot\frac{1}{\epsilon}\right]  \right\rceil.
		$$

		%
	\end{proof}
	
	Recall that $l= 1,2,\ldots$ is the count of successful iterations, and the sequence $\{ \bar{\x}_l, \ l=1,2,\ldots \}$ is updated when a successful iteration is identified. The iteration complexity result is presented in Theorem \ref{Theorem:AARCQ-Main} and Theorem \ref{Thm:AARCQ-Main}.
	\begin{theorem}\label{Theorem:AARCQ-Main}
		The sequence $\{ \bar{\x}_l, \ l=1,2,\ldots\}$ generated by Algorithm \ref{Algorithm: AARCQ} satisfies
		\begin{eqnarray*}
			& & \frac{l(l+1)(l+2)}{6}f(\bar{\x}_l) \leq \psi_l(\z_l) \leq \psi_l(\z) \\
			& \leq & \frac{l(l+1)(l+2)}{6}f(\z) + \frac{L_h + \bar{\sigma}_1 + (\kappa_{e}+ \kappa_{c} )\kappa_{hs} } {2}\left\| \z-\x_0\right\|^3 
			+ \frac{2\kappa_\theta (1+\kappa_\theta) L_g^2}{\sigma_{\min} - (\kappa_{e}+ \kappa_{c} )\kappa_{hs}}\left\| \x_0-\x^*\right\|^2 + \frac{1}{6}\varsigma_l\left\|\z-\bar{\x}_1\right\|^3 ,
		\end{eqnarray*}
		where
		\begin{equation*}
		\bar{\sigma}_1= \max\left\{\sigma_0, \frac{3\gamma_2 L_h + \gamma_2 (\kappa_{e}+ \kappa_{c} )\kappa_{hs}}{2}\right\} > 0.
		\end{equation*}
	\end{theorem}
	\begin{proof}
		The proof is based on mathematical induction. The base case of $l=1$ can be found in Theorem \ref{Theorem:AARCQ-Main-Prep}. Suppose that the theorem is true for some $l\geq 1$. Let us consider the case of $l+1$:
		\begin{eqnarray*}
			\psi_{l+1}(\z_{l+1})  & \le & \psi_{l+1}(\z) \\
			& \leq & \frac{l(l+1)(l+2)}{6}f(\z) + \frac{L_h+\bar{\sigma}_1 +(\kappa_{e}+ \kappa_{c} )\kappa_{hs}}{2}\left\|\z-\x_0\right\|^3 + \frac{1}{6}\varsigma_l\left\|\z-\bar{\x}_1\right\|^3 + \frac{2\kappa_\theta(1+\kappa_\theta)  L_g^2}{\sigma_{\min}}\|\x_0-\x^*\|^2 \\
			& & + \frac{(l+1)(l+2)}{2}\left[ f(\bar{\x}_l)+\left(\z-\bar{\x}_l\right)^\top\nabla f(\bar{\x}_l)\right]+ \frac{1}{6}\left(\varsigma_{l+1}-\varsigma_l\right) \left\|\z-\bar{\x}_1\right\|^3 \\
			& \leq & \frac{(l+1)(l+2)(l+3)}{6}f(\z) + \frac{L_h+\bar{\sigma}_1 +(\kappa_{e}+ \kappa_{c} )\kappa_{hs}}{2}\left\|\z-\x_0\right\|^3 \\
			& & + \frac{2\kappa_\theta (1+\kappa_\theta) L_g^2}{\sigma_{\min}}\|\x_0 - \x^*\|^2 + \frac{1}{6}\varsigma_{l+1}\left\|\z-\bar{\x}_1\right\|^3 ,
		\end{eqnarray*}
		where the last inequality is due to convexity of $f(\z)$.
		On the other hand, it follows from the way that $\psi_{l+1}(\z)$ is updated that $\frac{(l+1)(l+2)(l+3)}{6}f(\bar{\x}_{l+1}) \leq \psi_{l+1}(\z_{l+1})$, and thus Theorem~\ref{Theorem:AARCQ-Main} is proven.
	\end{proof}
	The established Theorem~\ref{Theorem:AARCQ-Main} implies the following main result on iteration complexity of Algorithm \ref{Algorithm: AARCQ}.
	\begin{theorem}\label{Thm:AARCQ-Main}
		The sequence $\{\bar{\x}_l, \ l=1,2,\ldots\}$ generated by Algorithm \ref{Algorithm: AARCQ} satisfies that
		\begin{equation*}
		f(\bar{\x}_l)-f(\x^*) \leq \frac{C_2}{l(l+1)(l+2)} \leq \frac{C_2}{l^3},
		\end{equation*}
		where
		\begin{eqnarray*}
			C_2 & = & \left(3 L_h + 3\bar{\sigma}_1 + 3(\kappa_{e}+ \kappa_{c} )\kappa_{hs} \right)\| \x_0 - \x^* \|^3 +  \left(\frac{ L_h+2\bar{\sigma}_2+2(\kappa_{e}+ \kappa_{c} )\kappa_{hs}+2\kappa_\theta L_g}{1-\kappa_\theta}\right)^{3}\frac{1}{\eta^2}\left\|\bar{\x}_1 - \x^*\right\|^3 \\
			& & + \frac{12\kappa_\theta (1+\kappa_\theta) L_g^2}{\sigma_{\min} - (\kappa_{e}+ \kappa_{c} )\kappa_{hs}}\left\|\x_0-\x^*\right\|^2.
		\end{eqnarray*}
		When $\epsilon$ is sufficiently small, the total number of iterations required to find $\bar{\x}_k$ such that $f(x_k)-f(x^*)\leq \max\{\epsilon, \epsilon D\}$
		is
		\begin{eqnarray*}
			k & \leq & 1+\frac{2}{\log(\gamma_1)}\log\left(\frac{\bar{\sigma}_1}{\sigma_{\min}}\right) + \left(1+ \frac{2}{\log(\gamma_1)}\log\left(\frac{\bar{\sigma}_2}{\sigma_{\min}}\right)\right)\left[\left(\frac{C_2}{\epsilon}\right)^{\frac{1}{3}}+1\right] \\
			& & + \left\lceil \frac{1}{\log(\gamma_3)}\log\left[\left(\frac{L_h+2\bar{\sigma}_2+2(\kappa_{e}+ \kappa_{c} )\kappa_{hs}+2\kappa_\theta L_g}{1-\kappa_\theta}\right)^{3} \frac{1}{\eta^2\varsigma_1} \right] \right\rceil \\
			& & + \left\lceil -\frac{1}{\log(\gamma_4)} \log\left[\frac{({L_g + (\kappa_{e}+ \kappa_{c} )\kappa_{hs} + \bar{\sigma}_2 })h_{0,0}}{(1-\kappa_\theta) \kappa_{hs}}\cdot\frac{1}{\epsilon}\right]  \right\rceil,
		\end{eqnarray*}
		where
		\begin{equation*}
		\bar{\sigma}_2 = \max\left\{\bar{\sigma}_1,\frac{\gamma_2 L_h}{2}+\gamma_2\kappa_{\theta}+\gamma_2(\kappa_{e}+ \kappa_{c} )\kappa_{hs}+\gamma_2\eta\right\}>0.
		\end{equation*}
	\end{theorem}
	\begin{proof}
		By Theorem~\ref{Theorem:AARCQ-Main} and taking $\z=\x^*$ we have
		{\small \begin{equation*}
			\frac{l(l+1)(l+2)}{6}f(\bar{\x}_l) \leq \frac{l(l+1)(l+2)}{6}f(\x^*) + \frac{L_h+\bar{\sigma}_1+(\kappa_{e}+ \kappa_{c} )\kappa_{hs}}{2}\left\|\x^*-\x_0\right\|^3 + \frac{2\kappa_\theta (1+\kappa_\theta) L_g^2}{\sigma_{\min}}\|\x_0-\x^*\|^2 + \frac{1}{6}\varsigma_l\left\| \x^*-\bar{\x}_1\right\|^3.
			\end{equation*}}
		\noindent Rearranging the terms, and combining with Lemmas \ref{Lemma:AARCQ-T1}, \ref{Lemma:AARCQ-T2} and \ref{Lemma:AARCQ-T3} yields the conclusions.
	\end{proof}
	
	\subsection{Strongly Convex Case}
	Next we extend the analysis to the case where the objective function is strongly convex. We further assume the level set of $f(\x)$, $\{\x\in\br^d:\, f(\x)\leq f(\x_0)\}$,
	is bounded and is contained in $\|\x-\x_*\| \le D$. We denote $\ACal_m^2(x), m\geq 1$, as the point generated by running $m$ outer loop iterations of Algorithm \ref{Algorithm: AARCQ}. Assume that $0<\kappa_\theta<\mu$, we show that the accelerated adaptive cubic regularization for Newton's method has a linear convergence rate. In particular, we can generate sequence $\{\hat{\x}_k,\; k=0, 1,2,\ldots \}$ through the following procedure:
	\begin{center}
		\begin{enumerate}
			\item  Define
			\begin{eqnarray*}
				m & = & 1+\frac{2}{\log(\gamma_1)}\log\left(\frac{\bar{\sigma}_1}{\sigma_{\min}}\right) + \left(1 + \frac{2}{\log(\gamma_1)}\log\left(\frac{\bar{\sigma}_2}{\sigma_{\min}}\right)\right)\left[2\left(\frac{\tau_1 D + \tau_2 }{\mu}\right)^{\frac{1}{3}}+1\right] \\
				& & + \left\lceil \frac{1}{\log(\gamma_3)}\log\left[\left(\frac{L_h+2\bar{\sigma}_2+2(\kappa_{e}+ \kappa_{c} )\kappa_{hs}+2\kappa_\theta L_g}{1-\kappa_\theta}\right)^{3} \frac{1}{\eta^2\varsigma_1} \right] \right\rceil \\
				& & + \left\lceil -\frac{1}{\log(\gamma_4)} \log\left[\frac{({L_g + (\kappa_{e}+ \kappa_{c} )\kappa_{hs} + \bar{\sigma}_2 })h_{0,0}}{(1-\kappa_\theta) \kappa_{hs}}\cdot\frac{1}{\epsilon}\right]  \right\rceil,
			\end{eqnarray*}
			with
			$$
			\tau_1 = 3 L_2 + 3\bar{\sigma}_1 + 3(\kappa_{e}+ \kappa_{c} )\kappa_{hs} + \left(\frac{ L_2+2\bar{\sigma}_2+2\kappa_\theta L_1}{1-\kappa_\theta}\right)^{3} \frac{1}{\eta^2}\quad \mbox{and} \quad \tau_2 = \frac{12\kappa_\theta (1+\kappa_\theta) L_1^2}{\sigma_{\min}}.
			$$
			\item Set $\hat \x_0\in\br^d$.
			\item For $k\geq 0$, iterate $\hat\x_k=\ACal_m^2(\hat \x_{k-1})$.
		\end{enumerate}
	\end{center}
	The theoretical guarantee of the above procedure can be described by the following theorem, whose proof is identical to that of Theorem \ref{Theorem:AARC-Linear-Main} and thus omitted.
	\begin{theorem} \label{Theorem:AARCQ-Linear-Main}
		Suppose the sequence $\{ \hat{\x}_k,\; k=0,1,2,\ldots \}$ is generated by the procedure above. For $k \ge O(\log(\frac{1}{\epsilon})) $ we have $f(\hat \x_k) -f(\x^*) \le \epsilon$. Specifically, the total number of iterations required to find such solution is $O\left(\sqrt[3]{\max\left\{\frac{ L_g}{\mu}, \frac{L_h}{\mu}\right\}}\log (\frac{1}{\epsilon})\right)$.
	\end{theorem}
	\begin{remark}
		Theorem \ref{Theorem:AARCQ-Linear-Main} implies a surprising result that, in view of the order of iteration complexity, the accelerated adaptive cubic regularization method for Newton's method remains even with inexact Hessian estimated from the gradients. Specifically, it still has an $O\left( \sqrt[3]{\cdot} \ \right)$ dependence on the conditional numbers $\frac{L_g}{\mu}$ and $\frac{L_h}{\mu}$. However, we need to set $0<\kappa_\theta<\mu$ where $\mu$ is unknown in practice.
	\end{remark}
	
	Furthermore, we can construct sequence $\{\z_l ,\; l=1,2,\ldots \}$ such that $\z_{l+1} = \z_l + \bar{\s}_l$ and $\bar{\s}_l$ is obtained by running \textsf{SAS} of Algorithm \ref{Algorithm: AARCQ} with initial point $\z_l$. Recall that $\sigma_{\min} \ge \kappa_e \kappa_{hs}$, and so
	\begin{eqnarray}
	f(\z_l) - f(\z_{l+1}) & \geq & f(\z_l)  - m(\z_l, \bar{\s}_l, \bar{\sigma}_l) \nonumber \\
	& = & -\bar{\s}_l^\top \nabla f(\z_l) - \frac{1}{2} \bar{\s}_l^\top H(\z_l) \bar{\s}_l - \frac{\bar{\sigma}_l}{3}\left\| \bar{\s}_l\right\|^3  \nonumber \\
	& = & -\bar{\s}_l^\top\nabla m(\z_l, \bar{\s}_l, \bar{\sigma}_l) + \frac{1}{2} \bar{\s}_l^\top H(\z_l)\bar{s}_l + \frac{2 \bar{\sigma}_l}{3}\left\| \bar{\s}_l \right\|^3  \nonumber \\
	& {\eqref{Eqn:Approx_Subprob} \above 0pt \geq} & \frac{1}{2} \bar{\s}_l^\top H(\z_l) \bar{\s}_l - \kappa_\theta \left\|\bar{\s}_l \right\|^2 + \frac{2{\sigma}_{\min}}{3}\left\| \bar{\s}_l \right\|^3  \nonumber \\
	& \geq & \frac{1}{2} \bar{\s}_l^\top \nabla^2 f(\z_l) \bar{\s}_l - \kappa_\theta \left\|\bar{\s}_l \right\|^2 + \frac{1}{2} \bar{\s}_l^\top (H(\z_l) - \nabla^2 f(\z_l)) \bar{\s}_l + \frac{2{\sigma}_{\min}}{3}\left\| \bar{\s}_l\right\|^3  \nonumber \\
	& {\eqref{Strongly-Convex-1}, \eqref{Hessian-Approximation} \above 0pt  \geq}  & \frac{\mu - 2\kappa_\theta}{2}\left\| \bar{\s}_l\right\|^2 \nonumber + \left(\frac{2\sigma_{\min}}{3}-\frac{\kappa_e \kappa_{hs}}{2}\right)\left\|\bar{\s}_l\right\|^3 \\
	& {\text{Lemma}\; \ref{Lemma:AARCQ-T3-P} \above 0pt  \geq} & \frac{(\mu - 2\kappa_\theta)(1-\kappa_\theta)}{2(\frac{ L_h}{2}+\bar{\sigma}_2+(\kappa_{e}+ \kappa_{c} )\kappa_{hs}+\kappa_\theta L_g)}\left\| \nabla f(\z_{l+1}) \right\| \nonumber \\
	& { \eqref{Strongly-Convex-2} \above 0pt  \geq} & \frac{(\mu - 2\kappa_\theta)(1-\kappa_\theta)}{2(\frac{L_h}{2}+\bar{\sigma}_2+(\kappa_{e}+ \kappa_{c} )\kappa_{hs}+\kappa_\theta L_g)} \sqrt{2 \mu( f(\z_{l+1})-f(\x^*))},  \nonumber
	\end{eqnarray}
	where $\kappa_\theta\in(0,1)$ is defined in Condition \ref{Cond:Approx_Subprob}. Hence, we have
	\begin{equation*}
	f(\z_{l+1}) - f(\x^*) \leq \frac{2(\frac{L_h}{2}+\bar{\sigma}_2+(\kappa_{e}+ \kappa_{c} )\kappa_{hs}+\kappa_\theta L_g)^2}{\mu(\mu - 2\kappa_\theta)^2(1-\kappa_\theta)^2} \left(f(\z_l) - f(\z_{l+1})\right)^2 \leq \frac{2(\frac{L_h}{2}+(\kappa_{e}+ \kappa_{c} )\kappa_{hs}+\bar{\sigma}_2+\kappa_\theta L_g)^2}{\mu(\mu - 2\kappa_\theta)^2(1-\kappa_\theta)^2} \left(f(\z_l) - f(\x^*)\right)^2,
	\end{equation*}
	and the region of quadratic convergence is given by
	\begin{equation*}
	\QCal = \left\{\z\in\br^d: f(\z)-f(\x^*)\leq\frac{\mu(\mu - 2\kappa_\theta)^2(1-\kappa_\theta)^2}{2(\frac{L_h}{2}+(\kappa_{e}+ \kappa_{c} )\kappa_{hs}+\bar{\sigma}_2+\kappa_\theta L_g)^2}\right\} .
	\end{equation*}
	
	\section{Accelerated Adaptive Gradient Method} \label{Section:AAGD}
	In this section, we present an accelerated adaptive gradient method that is \textit{fully} Lipschitz-constant-free. In particular, we consider the following standard approximation of $f$ evaluated at $\x_i$ with quadratic regularization:
	\begin{equation}\label{prob:AAGD}
	m(\x_i, \s,\sigma_i) = f(\x_i) + s^\top \nabla f(\x_i) + \frac{1}{2}\sigma_i\left\|\s\right\|^2,
	\end{equation}
	where $\sigma_i>0$ is a regularized parameter. Then our algorithms are described in Algorithm \ref{Algorithm: AAGD}.
	\begin{algorithm}[!h]
		\begin{algorithmic} \scriptsize
			\STATE Given $\gamma_2>\gamma_1>1$, $\gamma_3 >1$, $\eta>0$,  and $\sigma_{\min}>0$. Choose $x_0\in\br^d$, $\sigma_0 \ge \sigma_{\min}$, and $\varsigma_1>0$.
			\STATE \textbf{Begin Phase I:} \textsf{Simple Adaptive Subroutine (SAS)}
			\FOR{$i=0,1,2,\ldots$}
			\STATE Compute $\s_i\in\br^d$ such that $\s_i=\argmin_{\s\in\br^d} m(\x_i, \s, \sigma_i)$;
			\STATE Compute $\rho_i=f(\x_i+\s_i)-m(\x_i, \s_i, \sigma_i)$;
			\IF{$\rho_i<0$ [successful iteration] }
			\STATE $\x_{i+1}=\x_i+\s_i$ and $\sigma_{i+1}\in\left[\sigma_{\min}, \sigma_i\right]$;
			\STATE $\textbf{break}$.
			\STATE Record the total number of iterations in \textsf{SAA}: $T_1 = i+1$;
			\ELSE
			\STATE $\x_{i+1}=\x_i$ and $\sigma_{i+1}\in\left[\gamma_1\sigma_i, \gamma_2\sigma_i\right]$;
			\ENDIF
			\ENDFOR
			\STATE \textbf{End Phase I:} \textsf{Simple Adaptive Subroutine}
			
			\vspace{0.3 cm}
			
			\STATE \textbf{Begin Phase II:} \textsf{Accelerated Adaptive Subroutine (AAS)}\\
			Set the count of successful iterations $l=1$ and let $\bar{\x}_1 = \x_{T_1}$;\\
			Construct $\psi_1(\z) = f(\bar{\x}_1) + \frac{1}{4}\varsigma_1\|\z-\bar{\x}_1\|^2$, and let $\z_1= \argmin_{\z\in \br^d} \psi_1(\z)$, and choose $\y_1=\frac{1}{3}\bar{\x}_1 + \frac{2}{3}\z_1$; \\
			\FOR{$j=0,1,2\ldots$}
			\STATE Compute $\s_{T_1+j} = \argmin_{\s\in\br^d} m(\y_l, \s, \sigma_{T_1+j})$, and $\rho_{T_1+j}=-\frac{\s_{T_1+j}^\top\nabla f(\y_l+\s_{T_1+j})}{\left\|\s_{T_1+j}\right\|^2}$;
			\IF{$\rho_{T_1+j}\geq\eta$ [successful iteration]}
			\STATE $\x_{T_1+j+1}=\y_l+\s_{T_1+j}$, $\sigma_{T_1+j+1}\in\left[\sigma_{\min}, \sigma_{T_1+j}\right]$;
			\STATE Set $l=l+1$ and $\varsigma=\varsigma_{l-1}$;
			\STATE Update $\psi_l(\z)$ as illustrated in Table \ref{tab:auxiliary} by using $\varsigma_l=\varsigma$, and compute $\z_l = \argmin_{\z\in\br^d} \ \psi_l(\z)$;
			\WHILE{$\psi_l(\z_l) \geq \frac{l(l+1)}{2} f(\bar{\x}_l) $}
			\STATE Set $\varsigma = \gamma_3\varsigma$, and $\psi_l(\z)=\psi_{l-1}(\z) + \frac{l(l+1)}{2}\left[f(\x_{T_1+j+1})+\left(\z - \x_{T_1+j+1}\right)^\top\nabla f(\x_{T_1+j+1})\right]+ \frac{1}{4}\left(\varsigma-\varsigma_{l-1}\right)\|\z -\bar{\x}_1\|^2$;
			\STATE Compute $\z_l=\argmin_{\z\in\br^d} \ \psi_l(\z)$;
			\ENDWHILE
			\STATE $\varsigma_l=\varsigma$;
			\STATE Let $\bar{\x}_l = \x_{T_1+j+1} $, $\y_l=\frac{l}{l+2}\bar{\x}_l + \frac{2}{l+2}\z_l$.
			\ELSE
			\STATE $\x_{T_1+j+1}=\x_{T_1+j}$, $\sigma_{T_1+j+1}\in\left[\gamma_1\sigma_{T_1+j}, \gamma_2\sigma_{T_1+j}\right]$;
			\ENDIF
			\ENDFOR
			\STATE Record the total number of iterations of \textsf{AAS}: $T_2 = j+1$.
			\STATE \textbf{End Phase II:} \textsf{Accelerated Adaptive Subroutine}
		\end{algorithmic} \caption{Accelerated Gradient Method with Adaptive Quadratic Regularization}\label{Algorithm: AAGD}
	\end{algorithm}
	
	
	Different from the accelerated adaptive cubic regularization for Newton's method with exact/inexact Hessian, the subproblem in each iteration of Algorithm \ref{Algorithm: AAGD}:
	\begin{equation*}
	\s_i=\argmin_{\s\in\br^d} \ m(\x_i,\s,\sigma_i)=-\frac{1}{\sigma_i}\nabla f(\x_i) 
	\end{equation*}
	where $m(\x_i,\s,\sigma_i)$ is defined in \eqref{prob:AAGD}. Similarly, accoring to Table \ref{tab:auxiliary}, the subproblem
	\begin{equation*}
	\z_l=\argmin_{\z\in\br^d} \ \psi_l(\z) = \ell_l(\z) + \frac{1}{4}\varsigma_l \|\z-\bar{\x}_1\|^2, \quad l = 1,2,\ldots,,
	\end{equation*}
	for the acceleration admits a closed-form solution as well, where
	$\ell_{l}(\z)$ is a certain linear function of $\z$.
	Inparticular, by letting
	\begin{equation*}
	\nabla \psi_l(\z) = \nabla\ell_l(\z) +  \frac{1}{2}\varsigma_l(\z-\bar{\x}_1) = 0,
	\end{equation*}
	and using the fact that $\nabla\ell_l(\z)$ is independent of $\z$, we have
	\begin{equation*}
	\z_l=\bar{\x}_1-\frac{2}{\varsigma_l}\nabla\ell_l(\z).
	\end{equation*}
	
	\subsection{The Convex Case}
	In this subsection, we aim to analyze the theoretical performance of  Algorithm \ref{Algorithm: AAGD}. The proof sketch is similar to that of Algorithm \ref{Algorithm: AARC}. Thus, we shall move the details to the appendix, and only present two main results here.
	Recall that $l= 1,2,\ldots$ is the count of successful iterations, and the sequence $\{ \bar{\x}_l, \ l=1,2,\ldots \}$ is updated when a successful iteration is identified. The iteration complexity result is presented in Theorem \ref{Theorem:AAGD-Main} and Theorem \ref{Thm:AAGD-Main}.
	\begin{theorem}\label{Theorem:AAGD-Main}
		The sequence $\{\bar{x}_l, \ l=1,2,\ldots\}$ generated by Algorithm \ref{Algorithm: AAGD} satisfies
		\begin{equation*}
		\frac{l(l+1)}{2}f(\bar{\x}_l) \leq \psi_l(\z_l) \leq \psi_l(\z) \leq \frac{l(l+1)}{2}f(\z) + \frac{L_g+\bar{\sigma}_1} {2}\left\|\z-\x_0\right\|^2 + \frac{1}{4}\varsigma_l\left\|\z-\bar{\x}_1\right\|^2,
		\end{equation*}
		where
		\begin{equation*}
		\bar{\sigma}_1= \max\left\{\sigma_0, \gamma_2 L_g\right\} > 0.
		\end{equation*}
	\end{theorem}
	\begin{proof}
		As before, the proof is based on mathematical induction. The base case of $l=1$ is precisely the result of Theorem \ref{Theorem:AAGD-Main-Prep}. Suppose that the theorem is true for some $l\geq 1$. Let us consider the case of $l+1$:
		\begin{eqnarray*}
			\psi_{l+1}(\z_{l+1})  & \le & \psi_{l+1}(\z) \\
			& \leq & \frac{l(l+1)}{2}f(\z) + \frac{L_g+\bar{\sigma}_1 } {2}\left\|\z-\x_0\right\|^2 + \frac{1}{4}\varsigma_l\left\|\z-\bar{\x}_1\right\|^2 \\
			& & + (l+1)\left[ f(\bar{\x}_l) + \left(\z-\bar{\x}_l\right)^\top\nabla f(\bar{\x}_l)\right]  + \frac{1}{4}\left(\varsigma_{l+1}-\varsigma_{l}\right) \left\|\z-\bar{\x}_1\right\|^2 \\
			& \le & \frac{(l+1)(l+2)}{2}f(\z) + \frac{L_g + \bar{\sigma}_1} {2}\left\|\z-\x_0\right\|^2 + \frac{1}{4}\varsigma_{l+1}\left\|\z-\bar{\x}_1\right\|^2 ,
		\end{eqnarray*}
		where the last inequality is due to the convexity of $f(\z)$. On the other hand, it follows from the way that $\psi_{l+1}(\z)$ is updated that $\frac{(l+1)(l+2)}{2}f(\bar{\x}_{l+1}) \leq \psi_{l+1}(\z_{l+1})$, and thus Theorem~\ref{Theorem:AAGD-Main} is proven.
	\end{proof}
	Now Theorem~\ref{Theorem:AAGD-Main} leads to the following main result on iteration complexity of Algorithm \ref{Algorithm: AAGD}.
	\begin{theorem}\label{Thm:AAGD-Main}
		The sequence $\{\bar{\x}_l, \ l=1,2,\ldots\}$ generated by Algorithm \ref{Algorithm: AAGD} satisfies that
		\begin{equation*}
		f(\bar{\x}_l)-f(\x^*) \leq \frac{C_3}{l(l+1)} \leq \frac{C_3}{l^2},
		\end{equation*}
		where
		\begin{equation*}
		C_3 = \left( L_g+\bar{\sigma}_1\right)\|\x_0-\x^*\|^2 +  2\left(L_g+\bar{\sigma}_2\right)^{2} \|\x_1 - \x^*\|^2.
		\end{equation*}
		The total iteration number required to reach $\bar{\x}_k$ satisfying $f(\bar{\x}_k)-f(\x^*)\leq\epsilon$ is bounded as follows: 
		\begin{equation*}
		k\leq 1+\frac{2}{\log(\gamma_1)}\log\left(\frac{\bar{\sigma}_1}{\sigma_{\min}}\right) + \left(1+\frac{2}{\log(\gamma_1)}\log\left(\frac{\bar{\sigma}_2}{\sigma_{\min}}\right)\right)\left[\left(\frac{C_3}{\epsilon}\right)^{\frac{1}{2}}+1\right] +\left\lceil \frac{1}{\log(\gamma_3)}\log\left[ \left(L_g+\bar{\sigma}_2\right)^{2}\frac{4}{\eta \, \varsigma_1} \right] \right\rceil,
		\end{equation*}
		\noindent where
		\begin{equation*}
		\bar{\sigma}_2 = \max\left\{\bar{\sigma}_1, \gamma_2 L_g+\gamma_2\eta\right\}>0.
		\end{equation*}
	\end{theorem}
	\begin{proof}
		By Theorem~\ref{Theorem:AAGD-Main} and taking $\z=\x^*$ we have
		\begin{equation*}
		\frac{l(l+1)}{2}f(\bar{\x}_l) \leq \frac{l(l+1)}{2}f(\x^*) + \frac{L_g+\bar{\sigma}_1}{2}\left\|\x^*-\x_0\right\|^2 + \frac{1}{4}\varsigma_l\left\| \x^*-\bar{\x}_1\right\|^2.
		\end{equation*}
		Rearranging the terms, and combining with Lemmas \ref{Lemma:AAGD-T1}, \ref{Lemma:AAGD-T2} and \ref{Lemma:AAGD-T3} yields the conclusions.
	\end{proof}
	
	\subsection{Strongly Convex Case}
	Next we extend the analysis to the case where the objective function is strongly convex. We denote $\ACal_m^3(\x), m\geq 1$, as the point generated by running $m$ outer loop iterations of Algorithm \ref{Algorithm: AAGD}. In particular, we can generate sequence $\{\hat{\x}_k,\; k=0, 1,2,\ldots \}$ through the following procedure:
	
	\begin{center}
		\begin{enumerate}
			\item Define
			\begin{eqnarray*}
				m & = & 1+\frac{2}{\log(\gamma_1)}\log\left(\frac{\bar{\sigma}_1}{\sigma_{\min}}\right) + \left(1+ \frac{2}{\log(\gamma_1)}\log\left(\frac{\bar{\sigma}_2}{\sigma_{\min}}\right)\right)\left[2\left(\frac{ L_1 + \bar{\sigma}_1 + 2\left( L_g+\bar{\sigma}_2\right)^{2}  }{\mu}\right)^{\frac{1}{2}}+1\right] \\
				& & +\left\lceil \frac{1}{\log(\gamma_3)}\log\left[ \left(L_g+\bar{\sigma}_2\right)^{2}\frac{4}{\eta \, \varsigma_1} \right] \right\rceil.
			\end{eqnarray*}
			\item Set $\hat \x_0\in\br^d$.
			\item For $k\geq 0$, iterate $\hat \x_k=\ACal_m^3(\hat \x_{k-1})$.
		\end{enumerate}
	\end{center}
	The linear convergence of the sequence $\{ \hat{\x}_k,\; k=0,1,2,\ldots \}$ is presented in the following theorem.
	\begin{theorem} \label{Theorem:AAGD-Linear-Main}
		Suppose the sequence $\{ \hat{\x}_k,\; k=0,1,2,\ldots \}$ is generated by the procedure above. For $k \ge O(\log(\frac{1}{\epsilon})) $ we have $f(\hat \x_k) -f(\x^*) \le \epsilon$. Specifically, the total number of iterations required to find such solution is $O\left(\sqrt{\frac{L_g}{\mu}}\log(\frac{1}{\epsilon})\right)$.
	\end{theorem}
	\begin{proof}
		Because
		\begin{equation*}
		f(\x_{k+1}) - f(\x^*) \leq \frac{\mu}{4}\left\| \x_k - \x^*\right\|^2 \leq \frac{1}{2}\left(f(\x_k) - f(\x^*)\right),
		\end{equation*}
		the total number of iterations to find an $\epsilon$-solution is $O\left(\sqrt{\frac{L_g}{\mu}}\log(\frac{1}{\epsilon})\right)$.
	\end{proof}

	\section{Numerical Experiments}\label{Section5:Experiment}
	
	In this section, we implement a variant of Algorithm \ref{Algorithm: AARC}, referred to as {\it Adaptively Accelerated \& Cubic Regularized}\/ (AARC) Newton's method. In this variant we first run Algorithm \ref{Algorithm: AARC}. After $10$ successful iterations of {\it Accelerated Adaptive Subroutine}\/ are performed, we check the progress made by each iteration. In particular, when $\frac{\left| f(x^{k+1})-f(x^k) \right|}{\left| f(x^k)\right|} \leq 0.1$, which indicates that it is getting close to the global optimum, we switch to the adaptive cubic regularization phase of Newton's method (ARC) in \cite{Cartis-2011-Adaptive-I, Cartis-2011-Adaptive-II} with stopping criterion $\left\|\nabla f(x)\right\|\leq 10^{-9}$.
	In the implementation, we apply the so-called Lanczos process to approximately solve the subproblem $\min_{\s \in \br^d} m(\x_i,\s,\sigma_i)$. In addition to \eqref{Eqn:Approx_Subprob}, the approximate solution $\s$ is also made to satisfy
	\begin{equation}\label{AM:1}
	\s^\top \nabla f(\x_i) + \s^\top \nabla^2 f(\x_i) \s + \sigma\left\|\s\right\|^3 = 0
	\end{equation}
	for given $x$ and $\sigma$.
	Note that \eqref{AM:1} is a consequence of the first order necessary condition, and as shown in Lemma 3.2 \cite{Cartis-2011-Adaptive-I}, the global minimizer of $m(\x_i, \s, \sigma_i)$ when restricted to a Krylov subspace $$\KCal := \text{span}\{\nabla f(\x_i), \nabla^2 f(\x_i) \nabla f(\x_i), \left( \nabla^2f(\x_i) \right)^2 \nabla f(\x_i), \ldots\}$$ satisfies \eqref{AM:1} independent of the subspace dimension. Moreover, minimizing $m(\x_i,\s,\sigma_i)$ in the Krylov subspace only involve factorizing a tri-diagonal matrix, which can be done at the cost of $O(d)$.
	Thus, the associated approximate solution can be found through the so-called Lanczos process, where the dimension of $\KCal$ is gradually increased and an orthogonal basis of each subspace $\KCal$ is built up which typically involves one matrix-vector product.  Condition \eqref{Eqn:Approx_Subprob} can be used as the termination criterion for the Lanczos process in the hope to find a suitable trial step before the dimension of $\KCal$ approaches $d$.
	
	We test the performance of the algorithms by evaluating the following regularized logistic regression problem
	\begin{equation}\label{Prob:GGLR}
	\min_{\x\in \br^d} \ f(\x) = \frac{1}{n}\sum\limits_{i=1}^n \ln \left( 1+\exp\left(-b_i \cdot {\sa}_i^\top \x\right) \right) + \frac{\lambda}{2}\| \x\|^2
	\end{equation}
	where $(\sa_i,b_i)_{i=1}^n$ is the samples in the data set, and the regularization parameter is set as $\lambda=10^{-5}$. To observe the acceleration,  the starting point is randomly generated from a Gaussian random variable with zero mean and a large variance (say $5000$). In this way, initial solutions are likely to be far away from the global solution.
	
	We compare the new AARC method with 5 other methods, including the adaptive cubic regularization of Newton's method (ARC), the trust region method (TR), the limited memory Broyden-Fletcher-Goldfarb-Shanno method (L-BFGS) that is implemented in \textsf{SCIPY Solvers} \footnote{https://docs.scipy.org/doc/scipy/reference/optimize.html\#module-scipy.optimize}, Algorithm \ref{Algorithm: AAGD} referred to as adaptive accelerated gradient descent (AAGD) and the standard Nesterov's accelerated gradient descent (AGD).  The experiments are conducted on 6 \textsf{LIBSVM Sets} \footnote{https://www.csie.ntu.edu.tw/\~{}cjlin/libsvm/} for binary classification,  and the summary of those datasets are shown in Table \ref{Table: dataset}.

	\begin{table}[h]
		\caption{Statistics of datasets.}\vspace{-1em}
		\begin{center}
			\begin{tabular}{c|c|c} \hline
				Dataset & Number of Samples & Dimension \\ \hline
				\textit{sonar} & 208 & 60 \\
				\textit{splice} & 1,000 & 60 \\
				\textit{svmguide1} & 3,089 & 4 \\
				\textit{svmguide3} & 1,243 & 22 \\
				\textit{w8a} & 49,749 & 300 \\
				\textit{SUSY} & 5,000,000 & 18 \\
				\hline
			\end{tabular}\label{Table: dataset}
		\end{center}
	\end{table}
	
	The results in Figure $1$ and Figure $2$ confirm that AARC indeed accelerates ARC, especially when the current iterates has not entered the local region of quadratic convergence yet. Moreover, AARC outperforms other methods in both computational time and iterations numbers in most cases.

	\begin{figure}
		\begin{tabular}{p{0.03\textwidth}|ccc}
			
			& & & \\ \hline
			%
			
			\rotatebox{90}{}   &
			\subfloat
			{\includegraphics[width=0.3\textwidth, height=4.5cm]
				{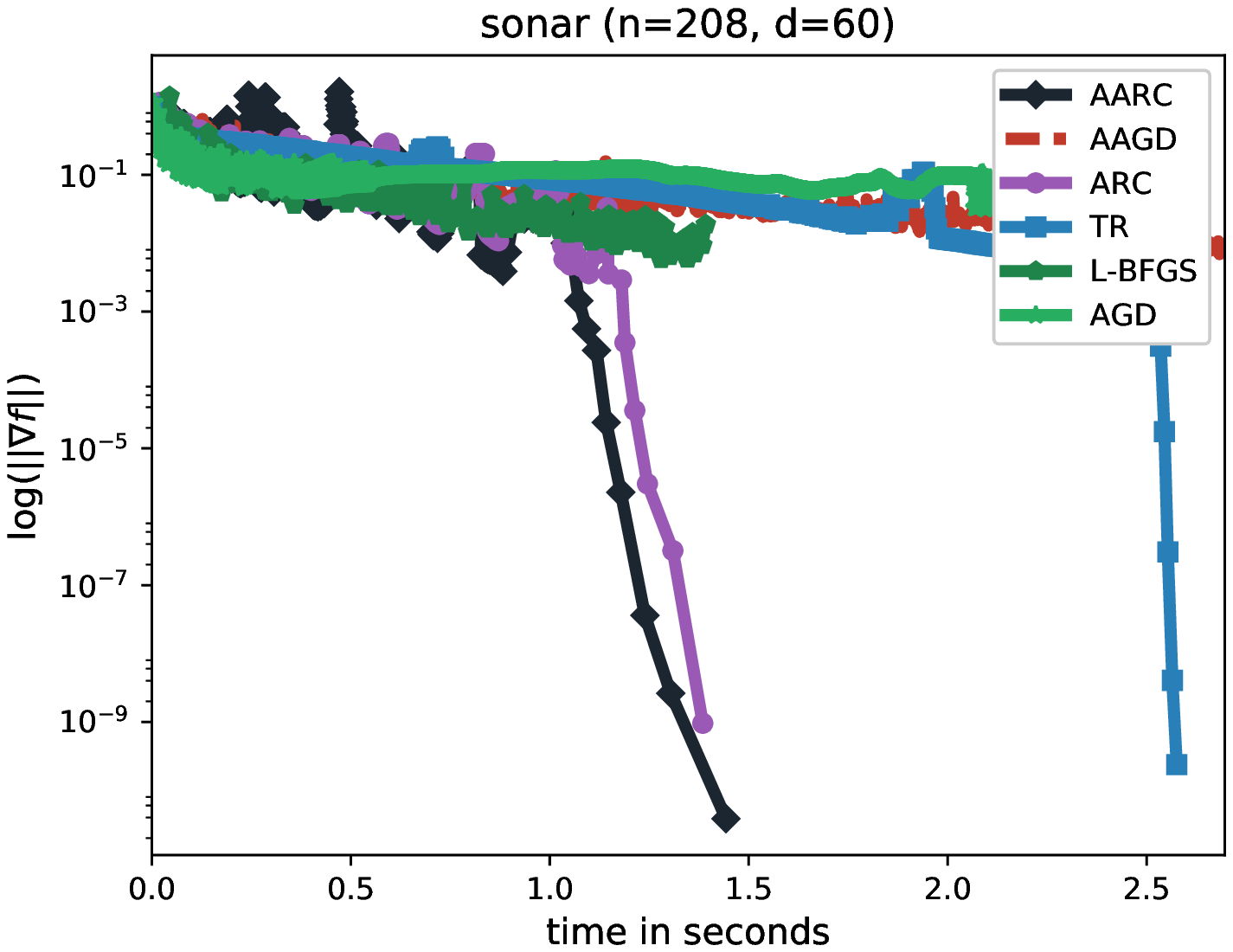}} &
			\subfloat
			{\includegraphics[width=0.3\textwidth, height=4.5cm]
				{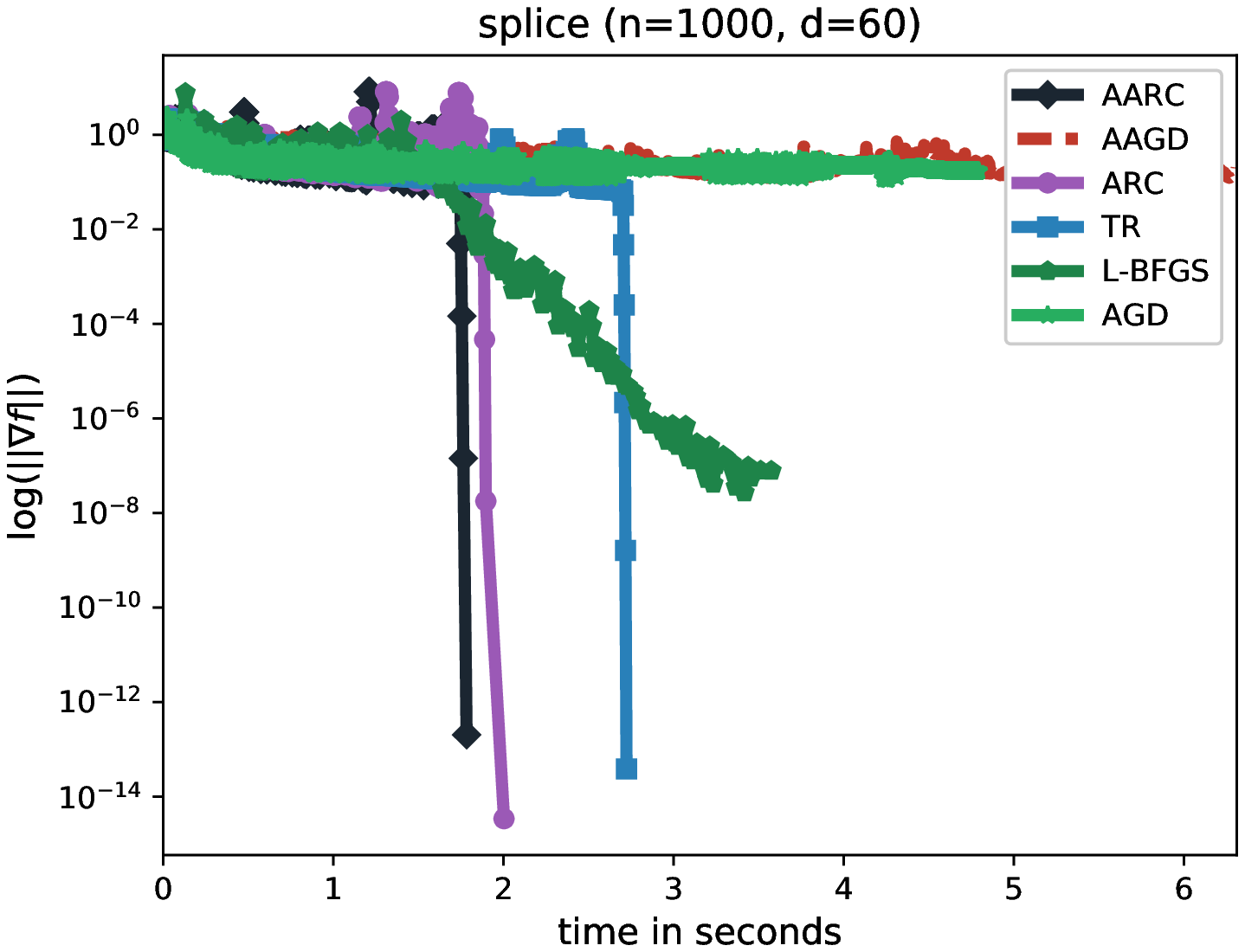}} &
			\subfloat
			{\includegraphics[width=0.3\textwidth, height=4.5cm]
				{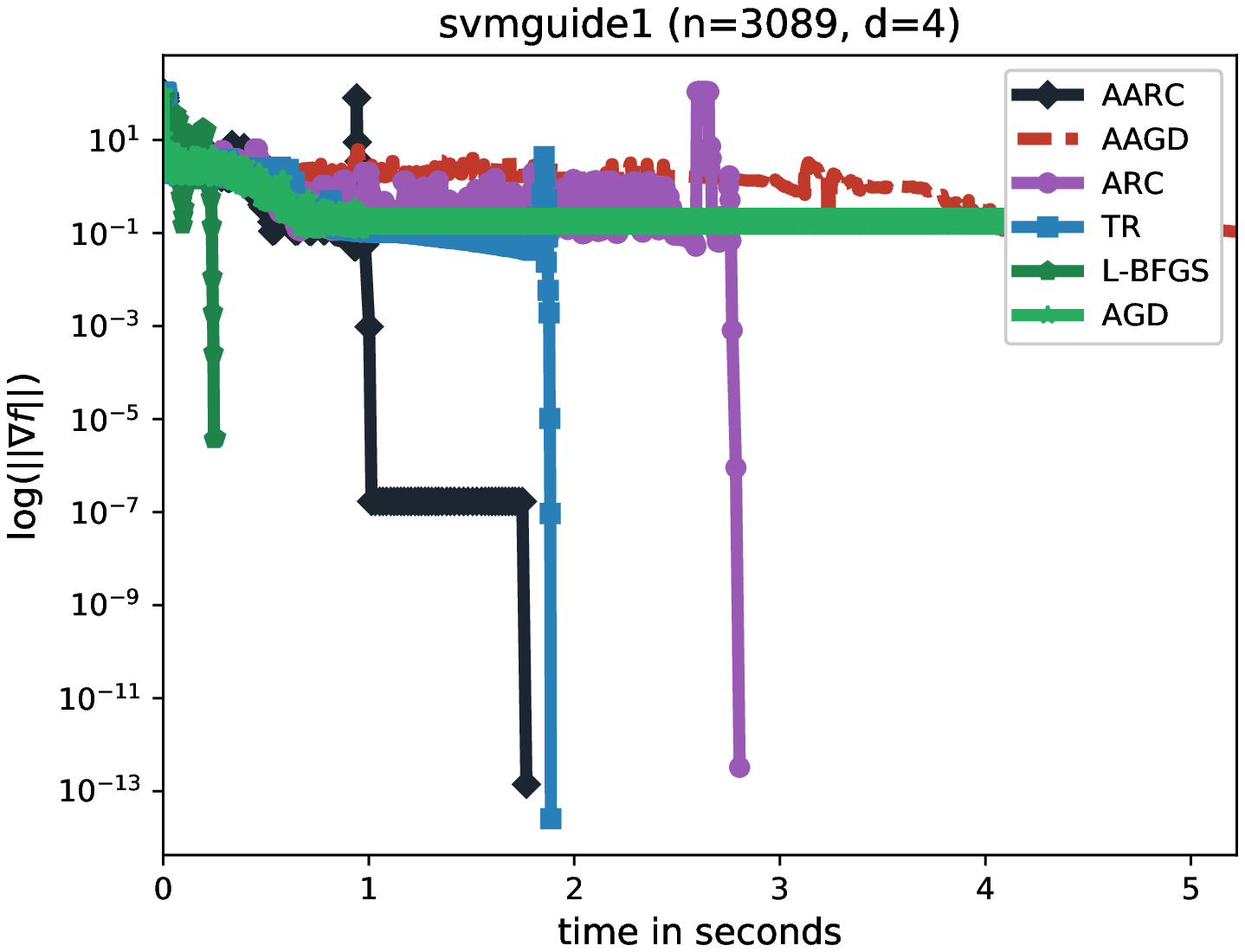}} \\
			
			\rotatebox{90}{}   &
			\subfloat
			{\includegraphics[width=0.3\textwidth, height=4.5cm]
				{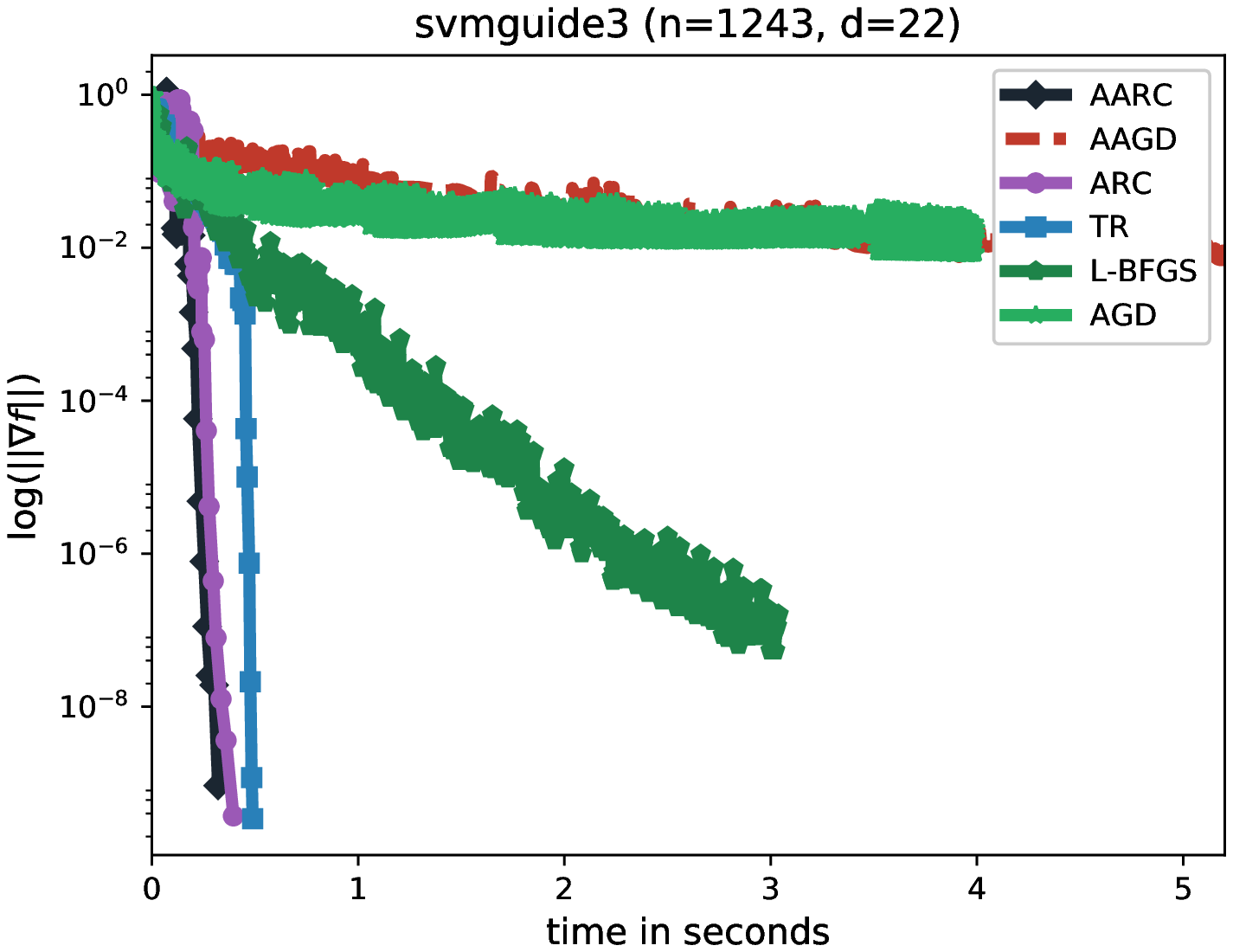}} &
			\subfloat
			{\includegraphics[width=0.3\textwidth, height=4.5cm]
				{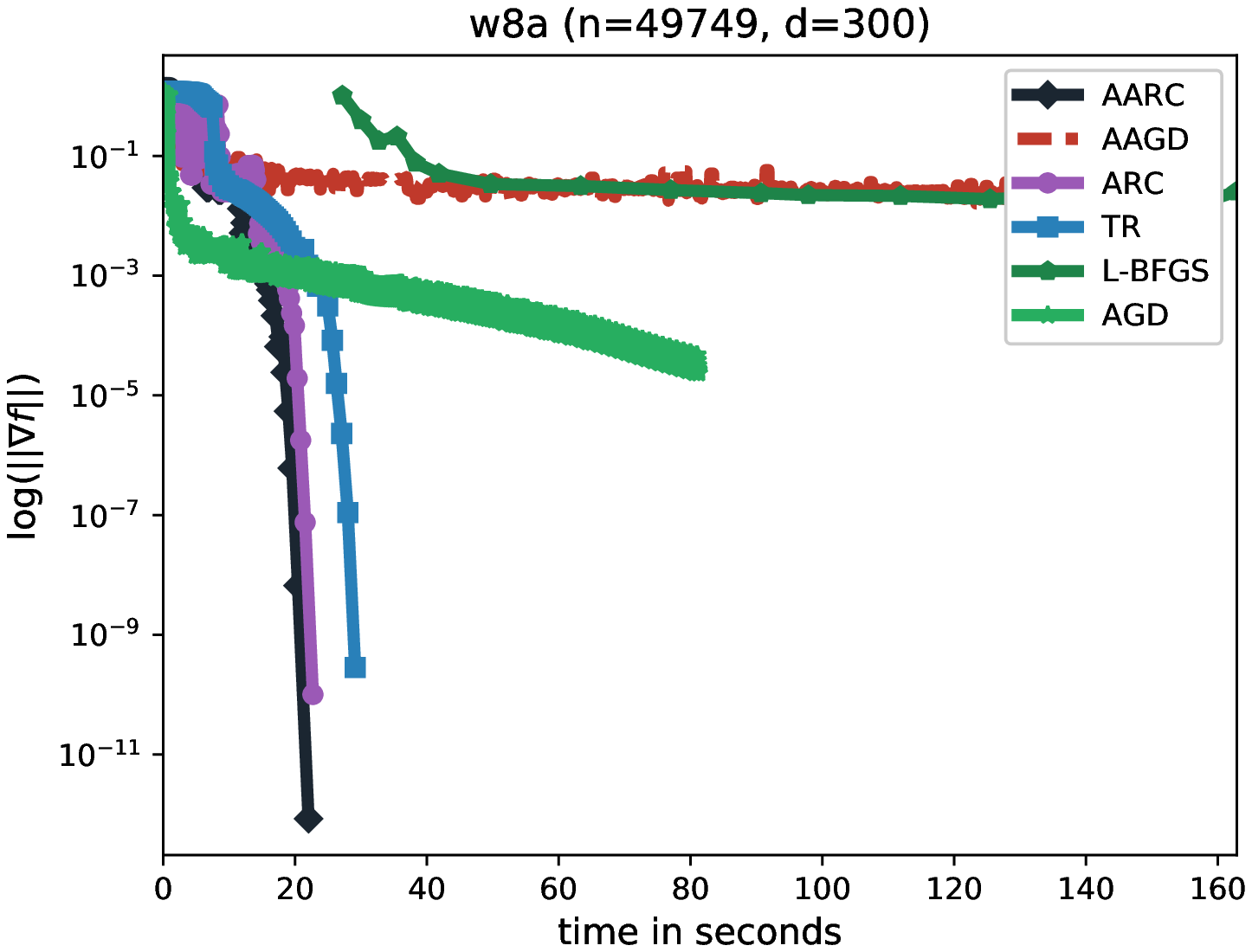}} &
			\subfloat
			{\includegraphics[width=0.3\textwidth, height=4.5cm]
				{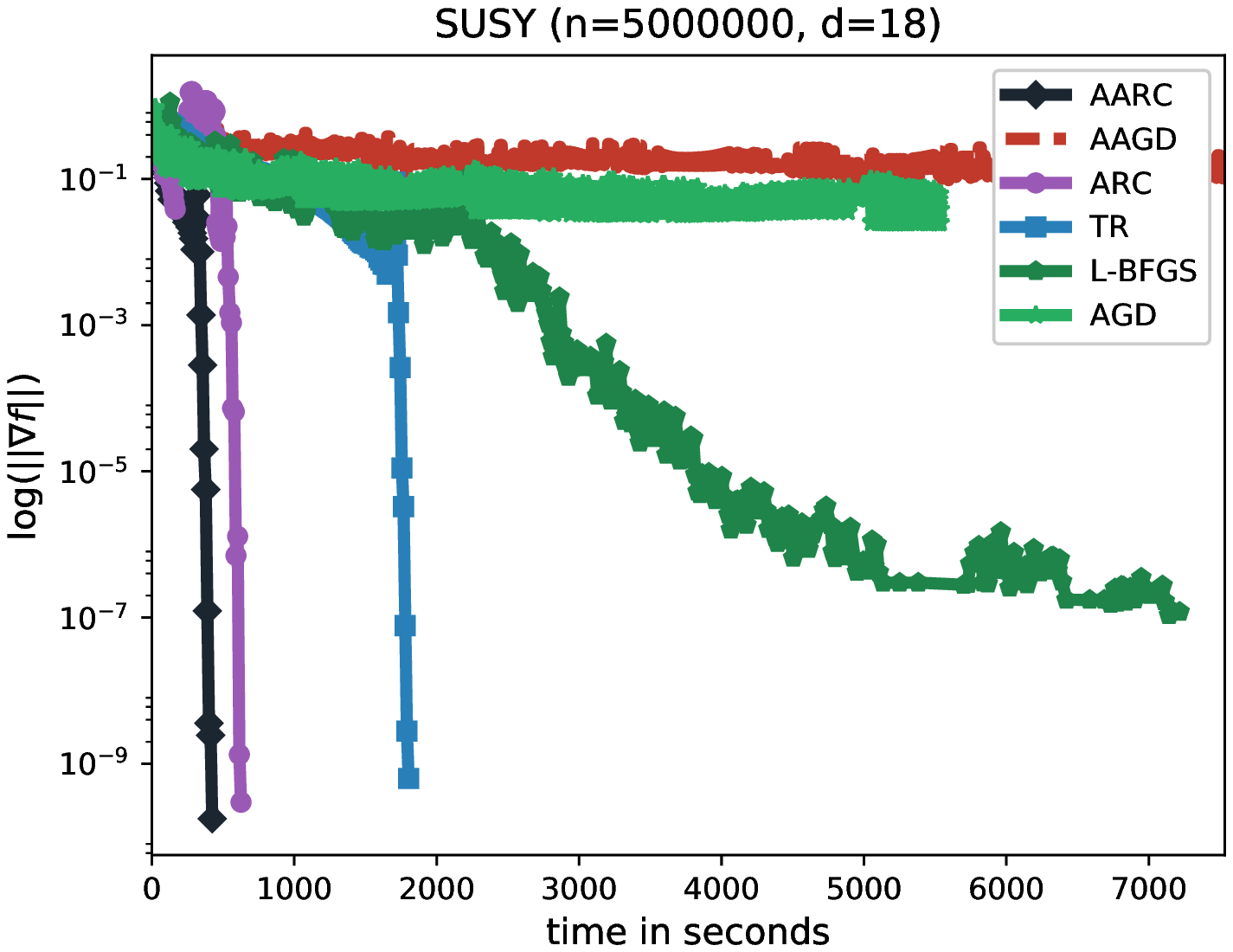}}   \\
			
		\end{tabular}\\
		\caption{Performance of AARC and all benchmark methods on the task of regularized logistic regression (loss vs.\ time)}
	\end{figure}
	
	\begin{figure}
		\begin{tabular}{p{0.03\textwidth}|ccc}
			
			& & & \\ \hline
			%
			%
			\rotatebox{90}{}   &
			\subfloat
			{\includegraphics[width=0.3\textwidth, height=4.5cm]
				{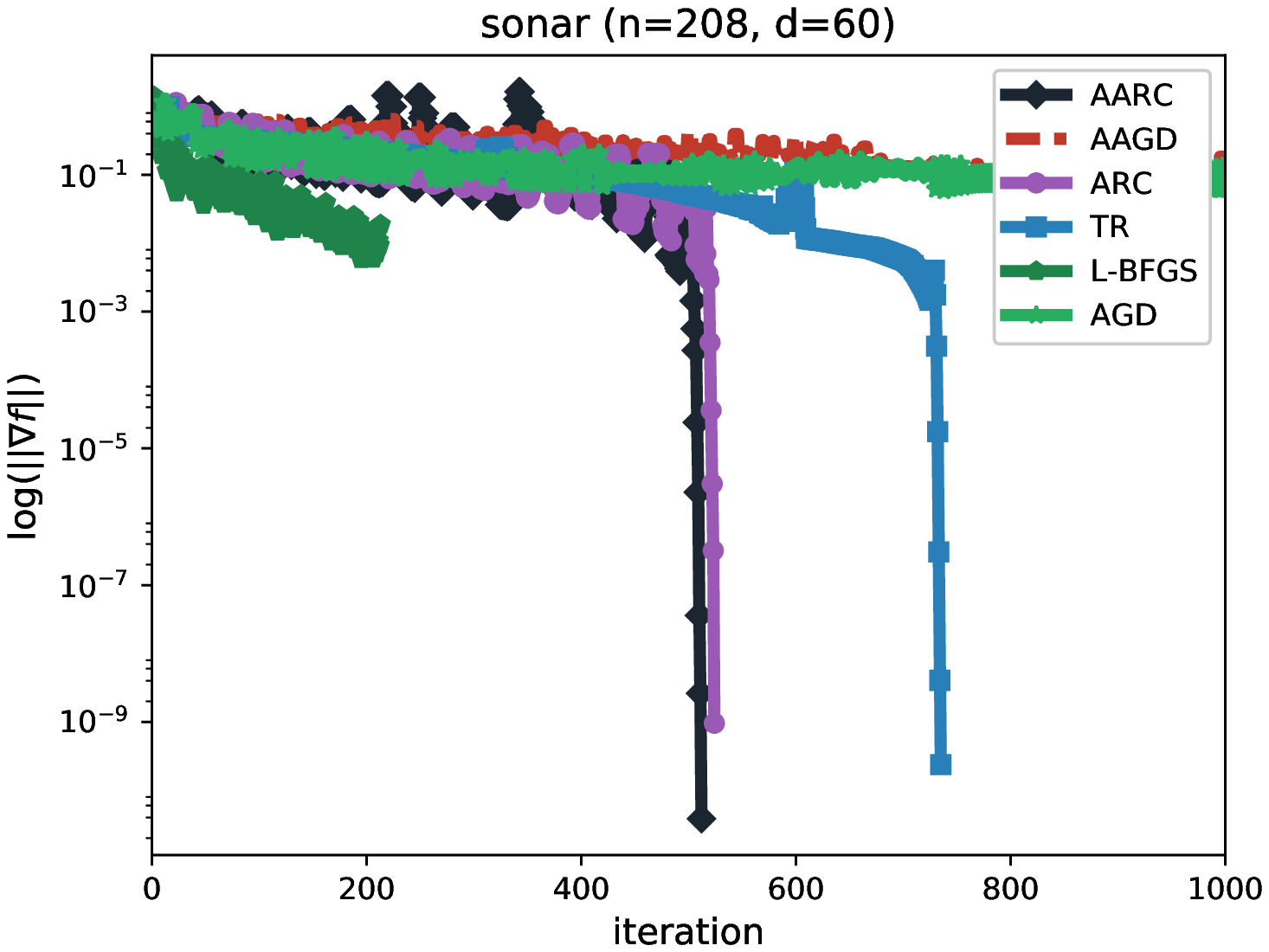}} &
			\subfloat
			{\includegraphics[width=0.3\textwidth, height=4.5cm]
				{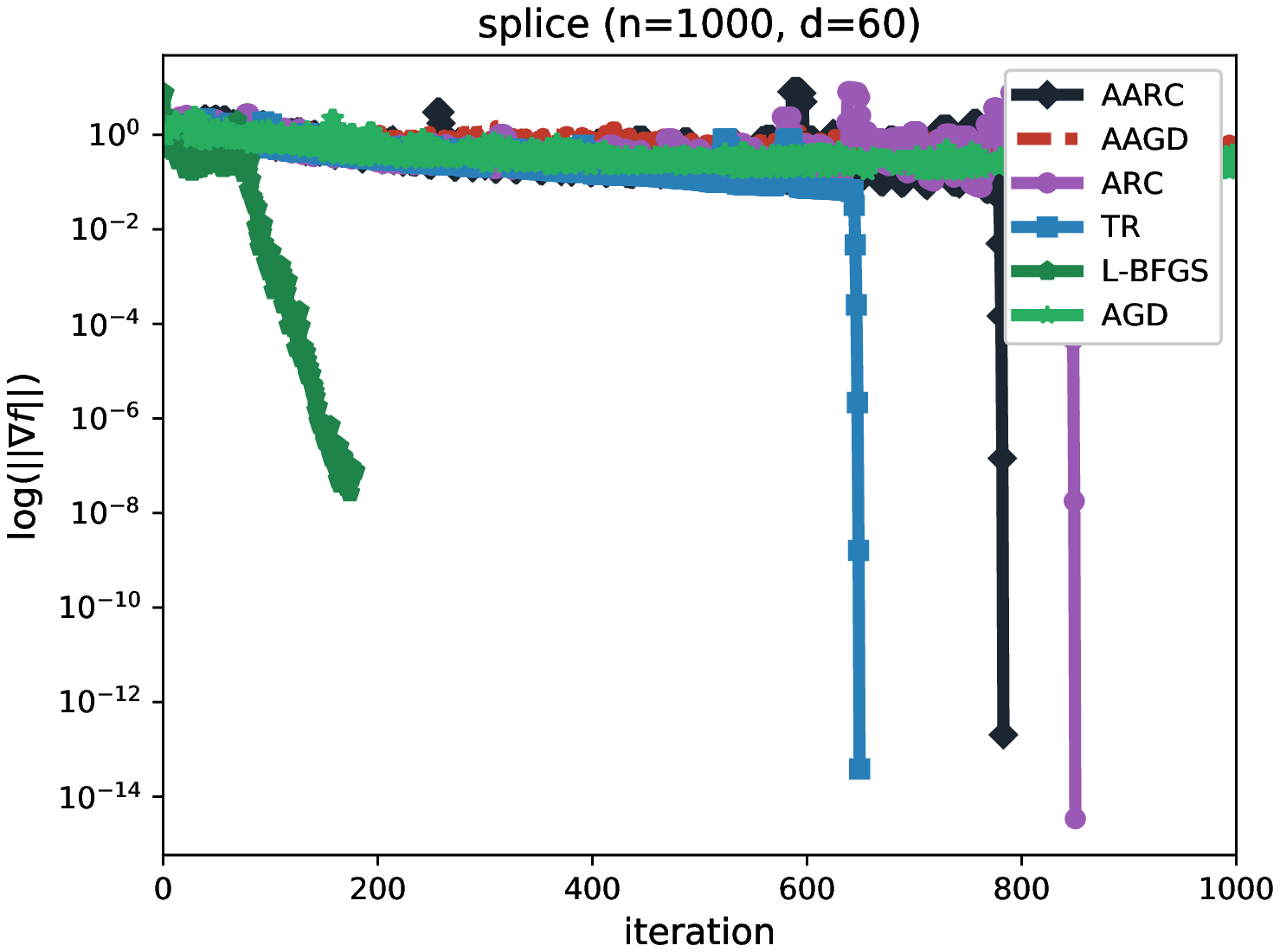}} &
			\subfloat
			{\includegraphics[width=0.3\textwidth, height=4.5cm]
				{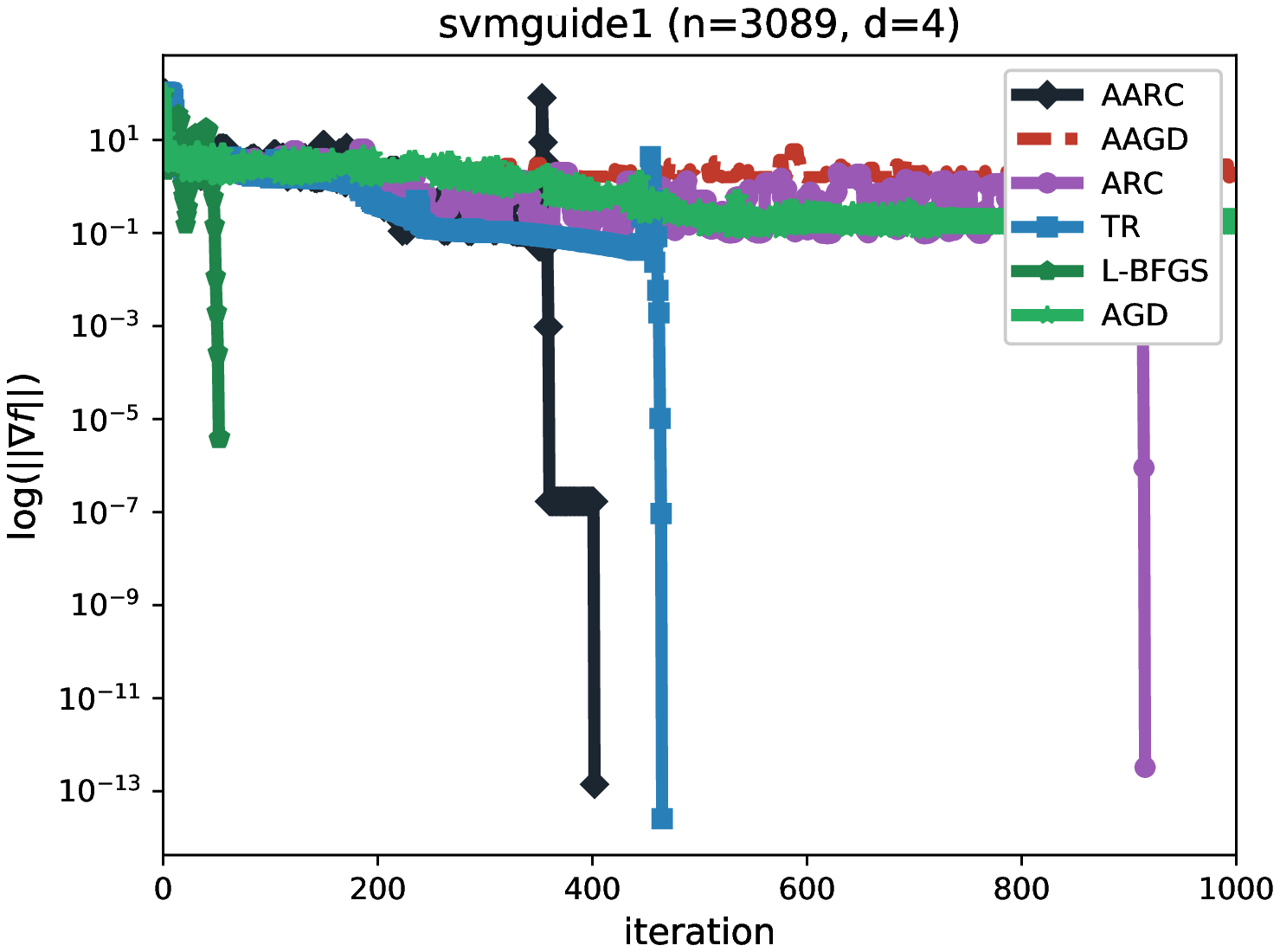}} \\
			
			\rotatebox{90}{}   &
			\subfloat
			{\includegraphics[width=0.3\textwidth, height=4.5cm]
				{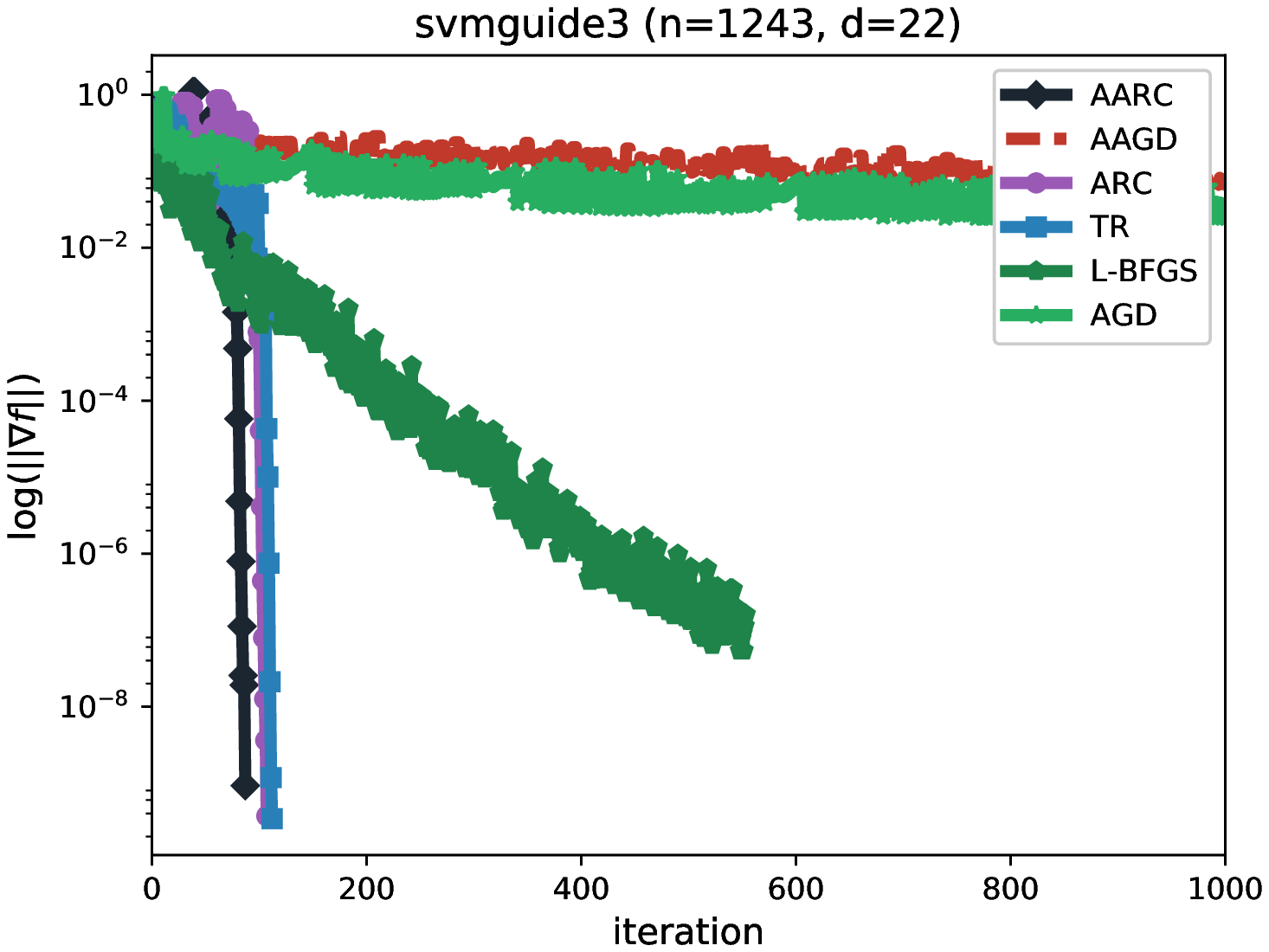}} &
			\subfloat
			{\includegraphics[width=0.3\textwidth, height=4.5cm]
				{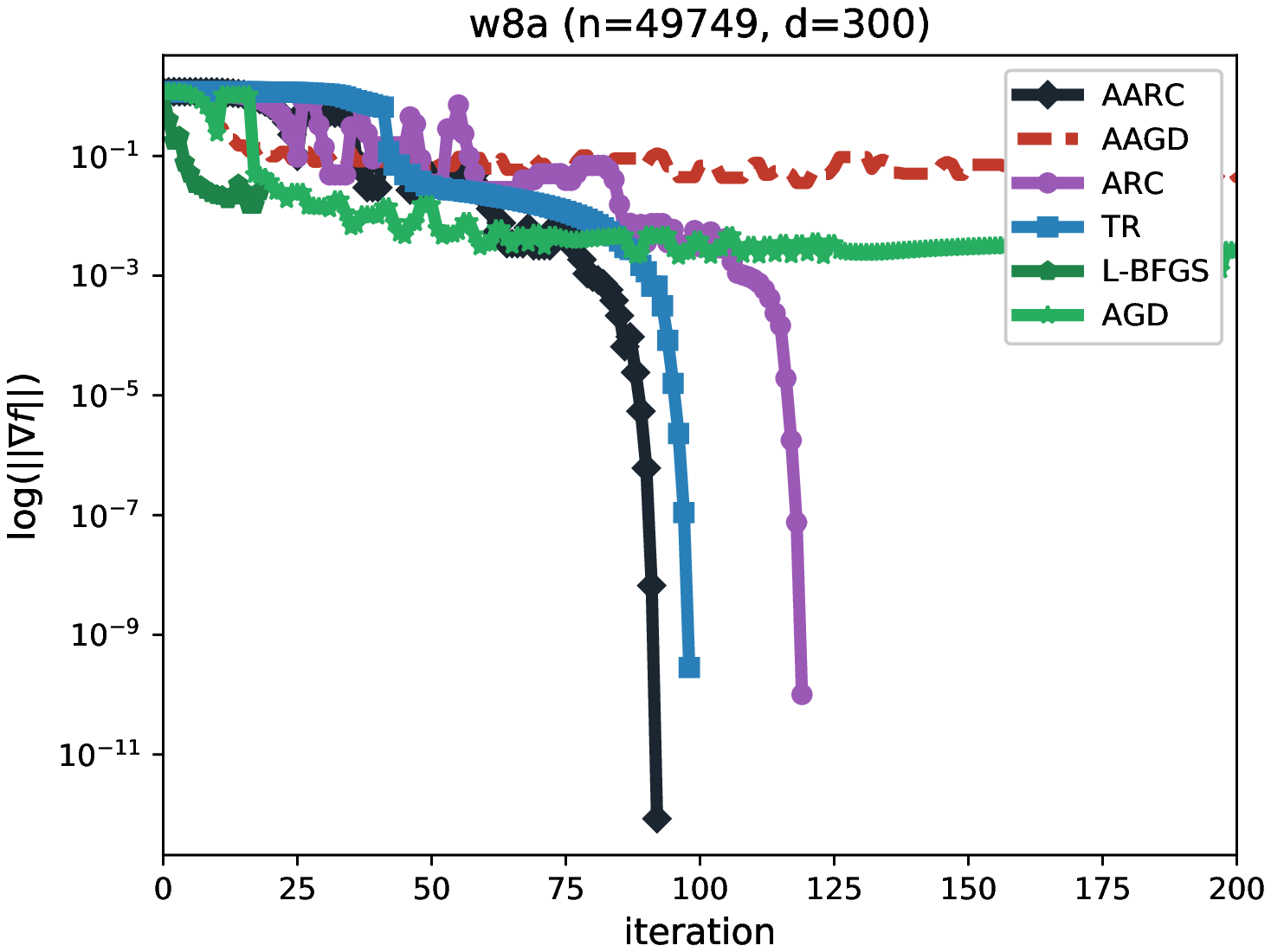}} &
			\subfloat
			{\includegraphics[width=0.3\textwidth, height=4.5cm]
				{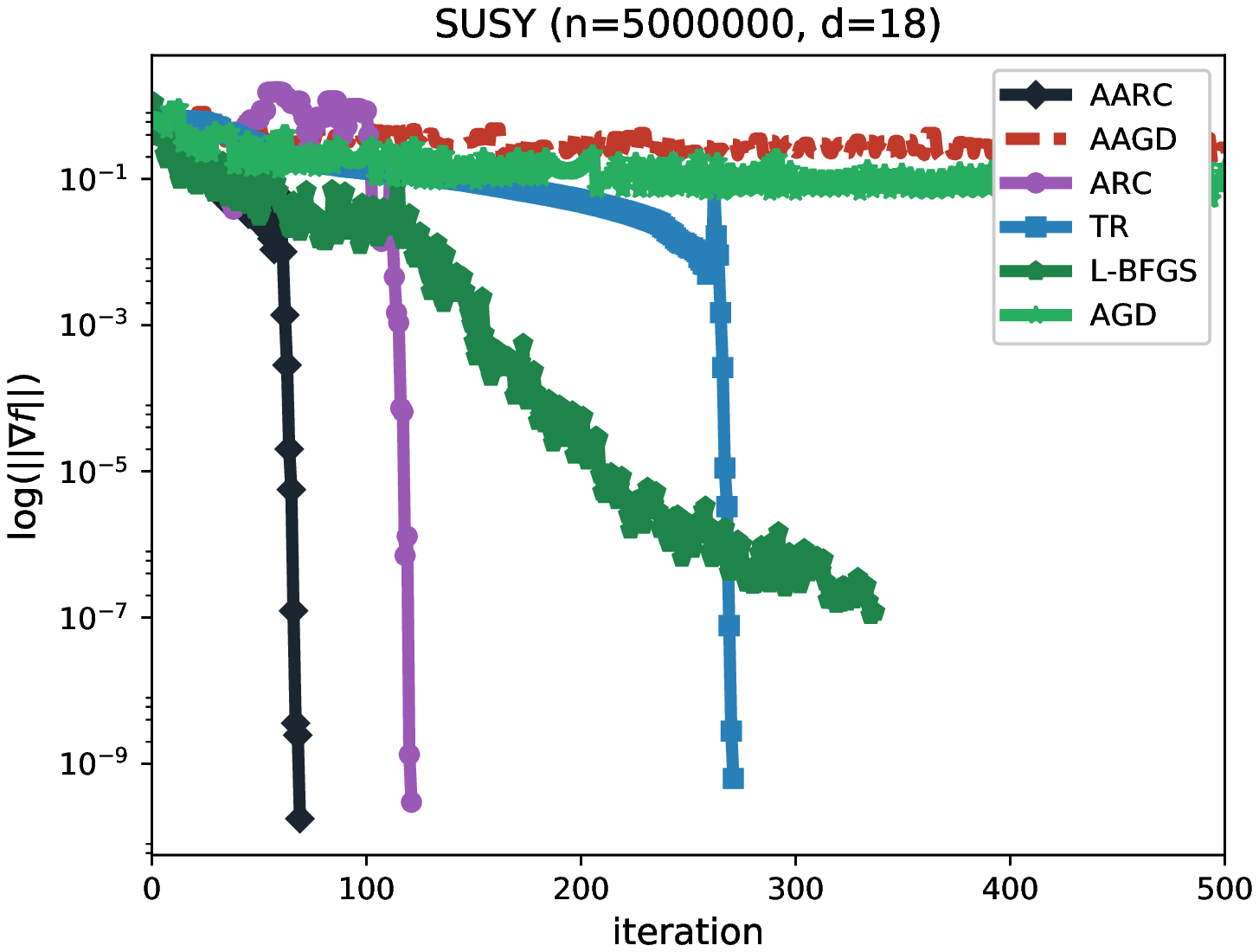}}   \\
			
		\end{tabular}
		\caption{Performance of AARC and all benchmark methods on the task of regularized logistic regression (loss vs.\ iterations)}
	\end{figure}

	
	\section*{Acknowledgement} We would like to express our deep gratitude toward Professor Xi Chen of Stern School of Business at New York University for the fruitful discussions at various stages of this project.

	\bibliographystyle{plain}

	\appendix
	\section{Proofs in Section \ref{Section:AARCQ} and Section \ref{Section:AAGD}}\label{Appendix: Proof}

	\subsection{Proofs in Section \ref{Section:AARCQ}}
	\begin{lemma}\label{Lemma:AARCQ-T1}
		Letting $\bar{\sigma}_1= \max\left\{\sigma_0, \frac{\gamma_2 L_h + 3\gamma_2 (\kappa_{e}+ \kappa_{c} )\kappa_{hs}}{2}\right\} > 0$, we have $T_1\leq 1+\frac{2}{\log\left(\gamma_1\right)}\log\left(\frac{\bar{\sigma}_1}{\sigma_{\min}}\right)$.
	\end{lemma}
	\begin{proof}
		We have
		\begin{eqnarray}
		f(\x_i + \s_i) & = & f(\x_i) + \s_i^\top \nabla f(\x_i) + \frac{1}{2} \s_i^\top \nabla^2 f(\x_i) \s_i + \int_{0}^{1} (1 - \tau) \s_i^\top \left[\nabla^2 f(\x_i + \tau \s_i) -  \nabla^2 f(\x_i)\right] \s_i \ d \tau \nonumber \\
		&\le & f(\x_i) + \s_i^\top \nabla f(\x_i) + \frac{1}{2} \s_i^\top \nabla^2 f(\x_i) \s_i  + \frac{ L_h}{6}  \| \s_i \|^3   \nonumber \\
		& = & m(\x_i, \s_i, \sigma_i) +  \frac{1}{2} \s_i^\top\left(\nabla^2 f(\x_i) - H(\x_i)\right) \s_i + \left(\frac{L_h}{6}-\frac{\sigma_i}{3}\right)\left\|\s_i\right\|^3, \nonumber \\
		& \leq & m(\x_i, \s_i, \sigma_i) + \left(\frac{L_h}{6} + \frac{(\kappa_{e}+ \kappa_{c} )\kappa_{hs}}{2} - \frac{\sigma_i}{3}\right)\left\|\s_i\right\|^3, \label{Equality-Second-Order-AARCQ}
		\end{eqnarray}
		where the inequalities hold true due to Assumption \ref{Assumption-Objective-Gradient-Hessian} and \eqref{Hessian-Approximation}. Therefore, we conclude that
		\begin{equation*}
		\sigma_i\geq\frac{L_h + 3 (\kappa_{e}+ \kappa_{c} )\kappa_{hs}}{2} \quad \Longrightarrow \quad f(\x_i+\s_i)\leq m(\x_i, \s_i, \sigma_i),
		\end{equation*}
		which further implies that $\sigma_i<\frac{L_h + 3(\kappa_{e}+ \kappa_{c} )\kappa_{hs}}{2}$ for $i\le T_1-2$. Hence,
		\begin{equation*}
		\sigma_{T_1} \le \sigma_{T_1 -1}\le \sigma_{T_1 -2} \le \frac{ \gamma_2 L_h + 3\gamma_2 (\kappa_{e}+ \kappa_{c} )\kappa_{hs}}{2}.
		\end{equation*}
		Because $\bar{\sigma}_1=\max\left\{\sigma_0, \frac{\gamma_2 L_h + 3 \gamma_2 (\kappa_{e}+ \kappa_{c} )\kappa_{hs}}{2}\right\}$, it follows from the construction of Algorithm \ref{Algorithm: AARCQ} that $\sigma_{\min}\leq\sigma_i$ for all iterations, and $\gamma_1\sigma_i\leq\sigma_{i+1}$ for all unsuccessful iterations. Consequently, we have
		\begin{equation*}
		\frac{\bar{\sigma}_1}{\sigma_{\min}} \geq \frac{\sigma_{T_1}}{\sigma_0} = \frac{\sigma_{T_1}}{\sigma_{T_1-1}} \cdot \prod_{j=0}^{T_1-2} \frac{\sigma_{j+1}}{\sigma_j} \geq \gamma_1^{T_1-1}\left(\frac{\sigma_{\min}}{\bar{\sigma}_1}\right),
		\end{equation*}
		and hence $T_1\leq 1+\frac{2}{\log\left(\gamma_1\right)}\log\left(\frac{\bar{\sigma}_1}{\sigma_{\min}}\right)$.
	\end{proof}
	
	\begin{lemma}\label{Lemma:AARCQ-T2}
		Letting $\bar{\sigma}_2 = \max\left\{\bar{\sigma}_1,\frac{\gamma_2 L_h}{2}+\gamma_2\kappa_{\theta}+\gamma_2(\kappa_{e}+ \kappa_{c} )\kappa_{hs}+\gamma_2\eta\right\}>0$, we have 
		$$T_2\leq \left(1+\frac{2}{\log(\gamma_1)}\log\left(\frac{\bar{\sigma}_2}{\sigma_{\min}}\right)\right)|\SCal|.$$
	\end{lemma}
	\begin{proof}
		We have
		{\small\begin{eqnarray*}
				& & \s_{T_1+j}^\top\nabla f(\y_l+\s_{T_1+j}) \\
				& = & \s_{T_1+j}^\top\left[\nabla f(\y_l + \s_{T_1+j}) - \nabla f(\y_l) - \nabla^2 f(\y_l) \s_{T_1+j}\right] + \s_{T_1+j}^\top\left[ \nabla f(\y_l) + \nabla^2 f(\y_l) \s_{T_1+j}\right] \\
				& \leq & \left\| \nabla f(\y_l+\s_{T_1+j}) - \nabla f(\y_l) - \nabla^2 f(\y_l) \s_{T_1+j}\right\| \left\| \s_{T_1+j}\right\|   + \s_{T_1+j}^\top\nabla m\left(\y_l, \s_{T_1+j}, \sigma_{T_1+j}\right) \\
				& & + \s_{T_1+j}^\top\left(\nabla^2 f(\y_l) - H(\y_l)\right)\s_{T_1+j} - \sigma_{T_1+j} \left\| \s_{T_1+j}\right\|^2 \\
				& {\text{Condition \ref{Cond:Approx_Subprob}} \above 0pt  \leq }  & \left\| \nabla f(\y_l+\s_{T_1+j}) - \nabla f(\y_l) - \nabla^2 f(\y_l) \s_{T_1+j}\right\| \left\| \s_{T_1+j}\right\| - \sigma_{T_1+j} \left\| \s_{T_1+j}\right\|^3 + \left(\kappa_\theta + (\kappa_{e}+ \kappa_{c} )\kappa_{hs}\right) \left\| \s_{T_1+j}\right\|^3 \\
				& = & \left\| \int_{0}^{1}  \left[ \nabla^2 f(\y_l + \tau \cdot \s_{T_1+j}) - \nabla^2 f(\y_l ) \right] \s_{T_1+j} \ d \tau \right\| \left\| \s_{T_1+j}\right\|  - \sigma_{T_1+j} \left\| \s_{T_1+j}\right\|^3 + \left(\kappa_\theta + (\kappa_{e}+ \kappa_{c} )\kappa_{hs}\right) \left\| \s_{T_1+j}\right\|^3 \\
				&  \leq  & \left(\frac{L_h}{2} + \kappa_\theta + (\kappa_{e}+ \kappa_{c} )\kappa_{hs} - \sigma_{T_1+j}\right)\left\|\s_{T_1+j}\right\|^3,
		\end{eqnarray*}}
		\noindent where the last inequality is due to Assumption \ref{Assumption-Objective-Gradient-Hessian}. Then it follows that
		\begin{equation*}
		-\frac{\s_{T_1+j}^\top\nabla f(\y_l+\s_{T_1+j})}{\left\| \s_{T_1+j}\right\|^3} \geq \sigma_{T_1+j} - \frac{L_h}{2} - \kappa_\theta - (\kappa_{e}+ \kappa_{c} )\kappa_{hs}.
		\end{equation*}
		Therefore, we have
		\begin{equation*}
		\sigma_{T_1+j} \geq \frac{L_h}{2}+\kappa_\theta+(\kappa_{e}+ \kappa_{c} )\kappa_{hs}+\eta \quad \Longrightarrow -\frac{\s_{T_1+j}^\top\nabla f(\y_l+\s_{T_1+j})}{\left\| \s_{T_1+j}\right\|^3}\geq\eta ,
		\end{equation*}
		which further implies that
		\begin{equation*}
		\sigma_{T_1+j+1}\leq \sigma_{T_1+j}\leq \gamma_2 \cdot \sigma_{T_1+j-1}\leq \gamma_2 \left(\frac{L_h}{2}+\kappa_\theta+(\kappa_{e}+ \kappa_{c} )\kappa_{hs}+\eta\right),\; \forall \; j \in \SCal.
		\end{equation*}
		Therefore, the above quantity can be bounded by $\bar{\sigma}_2=\max\left\{\bar{\sigma}_1,\frac{\gamma_2 L_h}{2}+\gamma_2\kappa_\theta+\gamma_2(\kappa_{e}+ \kappa_{c} )\kappa_{hs}+\gamma_2\eta\right\}$, where $\bar{\sigma}_1$ is responsible for an upper bound of $\sigma_{T_1}$. In addition, it follows from the construction of Algorithm \ref{Algorithm: AARCQ} that $\sigma_{\min}\leq\sigma_{T_1+j}$ for all iterations, and $\gamma_1\sigma_{T_1+j}\leq\sigma_{T_1+j+1}$ for all unsuccessful iterations. Therefore, we have
		\begin{equation*}
		\frac{\bar{\sigma}_2}{\sigma_{\min}} \geq \frac{\sigma_{T_1+T_2}}{\sigma_{T_1}} = \prod_{j\in\SCal} \frac{\sigma_{T_1+j+1}}{\sigma_{T_1+j}} \cdot \prod_{j\notin\SCal} \frac{\sigma_{T_1+j+1}}{\sigma_{T_1+j}} \geq \gamma_1^{T_2-|\SCal|}\left(\frac{\sigma_{\min}}{\bar{\sigma}_2}\right)^{|\SCal|},
		\end{equation*}
		hence
		\begin{equation*}
		|\SCal| \le T_2\leq |\SCal|+\frac{\left(|\SCal|+1\right)}{\log\gamma_1} \log\left(\frac{\bar{\sigma}_2}{\sigma_{\min}}\right) \leq \left(1+ \frac{2}{\log\gamma_1} \log\left(\frac{\bar{\sigma}_2}{\sigma_{\min}}\right)\right)|\SCal|.
		\end{equation*}
	\end{proof}
	Before estimating an upper bound for $T_3$, i.e., the total number of times of successfully updating $\varsigma>0$, we need to extend Lemma \ref{Lemma:AARC-T3-P2} in Algorithm \ref{Algorithm: AARC} to the following lemma.
	\begin{lemma}\label{Lemma:AARCQ-T3-P}
		For each iteration $j$ in the subroutine \textsf{AAS}, if it is successful, we have
		\begin{equation*}
		(1-\kappa_\theta)\left\| \nabla f(\x_{j+1})\right\| \leq \left(\frac{ L_h}{2}+\bar{\sigma}_2 + (\kappa_{e}+ \kappa_{c} )\kappa_{hs} + \kappa_\theta L_g\right)\left\|\s_j\right\|^2,
		\end{equation*}
		where $\kappa_\theta\in(0,1)$ is used in Condition~\ref{Cond:Approx_Subprob}.
	\end{lemma}
	\begin{proof}
		We denote $j$-th iteration is the $l$-th successful iteration, and note $\nabla_{\s} m(\y_l,\s_j,\sigma_j)=\nabla f(\y_l) + H(\y_l)\s_j + \sigma_j\|\s_j\|\cdot\s_j$. Then we have	
		\begin{eqnarray*}
						& & \left\| \nabla f(\x_{j+1}) \right\| \\
			 & \leq & \left\|  \nabla f(\y_l+\s_j) - \nabla_{\s} m(\y_l, \s_j, \sigma_j)\right\| + \left\|\nabla_{\s} m(\y_l, \s_j, \sigma_j)\right\| \\
			& \leq & \left\| \nabla f(\y_l+\s_j) - \nabla_{\s} m(\y_l, \s_j, \sigma_j)\right\| + \kappa_\theta\cdot\min\left(1, \left\|\s_j\right\|\right)\cdot \left\| \nabla f(\y_l) \right\| \\
			& \leq & \left\|\int_{0}^1\left( \nabla^2 f(\y_l+\tau \s_j)- \nabla^2 f(\y_l)\right) \s_j d\tau\right\| + \left\| \nabla^2 f(\y_l) - H(\y_l)\right\|\left\| \s_j\right\| + \sigma_j\left\| \s_j \right\|^2 + \kappa_\theta\cdot\min\left(1, \left\|\s_j\right\|\right)\cdot \left\| \nabla f(\y_l)\right\| \\
			& \leq & \frac{L_h}{2} \left\|\s_j\right\|^2 + (\kappa_{e}+ \kappa_{c} )\kappa_{hs}\left\| \s_j \right\|^2 +
			\sigma_j\left\|\s_j\right\|^2 +  \kappa_\theta \cdot \|\s_j \| \cdot \left\| \nabla f(\y_l) - \nabla f(\y_l+\s_j)\right\| + \kappa_\theta \left\| \nabla f(\x_{j+1})\right\| \\
			& \leq & \frac{L_h}{2} \left\|\s_j\right\|^2 + (\kappa_{e}+ \kappa_{c} )\kappa_{hs}\left\|\s_j\right\|^2 + \bar{\sigma}_2\left\|\s_j\right\|^2 +  \kappa_\theta L_g\left\|\s_j\right\|^2 + \kappa_\theta\left\|\nabla f(\x_{j+1})\right\|,
		\end{eqnarray*}
		where the second inequality holds true due to Condition~\ref{Cond:Approx_Subprob}, and the last two inequality follow from Assumption \ref{Assumption-Objective-Gradient-Hessian}. Rearranging the terms, the conclusion follows.
	\end{proof}
	Now we are ready to estimate the upper bound of $T_3$, i.e., the total number of count of successfully updating $\varsigma>0$.
	\begin{lemma}\label{Lemma:AARCQ-T3}
		We must have
		\begin{equation*}
		\psi_l(\z_l)\ge\frac{l(l+1)(l+2)}{6} f(\bar{\x}_l)
		\end{equation*}
		if $\varsigma_l\ge\left(\frac{ L_h+2\bar{\sigma}_2+2\kappa_\theta L_g}{1-\kappa_\theta}\right)^{3} \frac{1}{\eta^2}$, which further implies that
		\begin{equation*}
		T_3 \le \left\lceil \frac{1}{\log\left(\gamma_3\right)}\log \left[\left(\frac{ L_h+2\bar{\sigma}_2+2(\kappa_{e}+ \kappa_{c} )\kappa_{hs}+2\kappa_\theta L_g}{1-\kappa_\theta}\right)^{3} \frac{1}{\eta^2\varsigma_1} \right] \right\rceil.
		\end{equation*}
	\end{lemma}
	\begin{proof}
		The proof is similar to that of Lemma \ref{Lemma:AARC-T3} except
		replacing Lemma \ref{Lemma:AARC-T3-P2}
		with Lemma \ref{Lemma:AARCQ-T3-P}.
	\end{proof}
	
	Now we are able to prove the base case of $l =1$ for Theorem \ref{Theorem:AARCQ-Main}.
	\begin{theorem}\label{Theorem:AARCQ-Main-Prep}
		It holds that
		{\small\begin{equation*}
			f(\bar{\x}_1) \leq \psi_1(\z_1) \leq \psi_1(\z) \leq f(\z) + \frac{L_h+\bar{\sigma}_1+(\kappa_{e}+ \kappa_{c} )\kappa_{hs}}{2}\left\|\z-\x_0\right\|^3 + \frac{2\kappa_\theta (1+\kappa_\theta) L_g^2}{\sigma_{\min}}\|\x_0-\x^*\|^2+ \frac{1}{6}\varsigma_1\left\|\z-\bar{\x}_1\right\|^3 .
			\end{equation*}}
	\end{theorem}

	\begin{proof}
		By the definition of $\psi_1(\z)$ and the fact that $\bar{\x}_1=\x_{T_1}$, we have
		\begin{equation*}
		f(\bar{\x}_1) = f(\x_{T_1}) = \psi_1(\z_1).
		\end{equation*}
		Furthermore, by the criterion of successful iteration in \textsf{SAS},
		\begin{eqnarray*}
			f(\bar{\x}_1) & = & f(\x_{T_1}) \\
			& \leq & m(\x_{T_1-1}, \s_{T_1-1}, \sigma_{T_1-1}) \\
			& = & \left[ m(\x_{T_1-1}, \s_{T_1-1}, \sigma_{T_1-1}) - m(\x_{T_1-1}, \s^m_{T_1-1}, \sigma_{T_1-1}) \right] + m(\x_{T_1-1}, \s^m_{T_1-1}, \sigma_{T_1-1}),
		\end{eqnarray*}
		where $\s^m_{T_1-1}$ denotes the global minimizer of $m(\x_{T_1-1}, \s, \sigma_{T_1-1})$ over $\br^d$. Since $H(\x_{T_1-1}) \succeq 0 $ due to \eqref{Hessian-Convex}, $m(\x_{T_1-1},\s, \sigma_{T_1-1})$ is convex as well. Indeed, we have
		\begin{eqnarray*}
			\nabla_{\s}^2 m(\x_{T_1-1},\s, \sigma_{T_1-1}) & = & H(\x_{T_1-1}) + \sigma_{T_1-1}\|\s\| \cdot \BI + \sigma_{T_1-1}\frac{\s\s^\top}{\|\s\|^2} \succeq 0。
		\end{eqnarray*}
	    Therefore, we have
		\begin{eqnarray*}
			& & m(\x_{T_1-1}, \s_{T_1-1}, \sigma_{T_1-1}) - m(\x_{T_1-1}, \s^m_{T_1-1}, \sigma_{T_1-1}) \\
			& \leq & \nabla_{\s} m(\x_{T_1-1}, \s_{T_1-1}, \sigma_{T_1-1})^\top \left(\s_{T_1-1} - \s^m_{T_1-1}\right) \\
			& \leq & \left\| \nabla_{\s} m(\x_{T_1-1}, \s_{T_1-1}, \sigma_{T_1-1}) \right\| \left\|\s_{T_1-1} - \s^m_{T_1-1}  \right\| \\
			& { \eqref{Eqn:Approx_Subprob} \above 0pt  \leq }& \kappa_\theta \left\| \nabla f(\x_{T_1-1}) \right\| \left\| \s_{T_1-1} \right\| \left\| \s_{T_1-1} - \s^m_{T_1-1}  \right\|.
		\end{eqnarray*}
		To bound $\left\|\s_{T_1-1} - \s^m_{T_1-1}\right\|$, since $H(\x_{T_1-1}) \succeq 0$ we have that
		\begin{eqnarray*}
			\sigma_{\min}\left\|\s\right\|^3  \leq  \sigma_{T_1-1}\left\|\s\right\|^3 &=& \s^\top \left[ \nabla m \left(\x_{T_1-1}, \s, \sigma_{T_1-1}\right) - \nabla f(\x_{T_1-1}) - H(\x_{T_1-1}) \s \right] \\
			& \leq & \s^\top \left[ \nabla m \left(\x_{T_1-1}, \s, \sigma_{T_1-1}\right) - \nabla f(\x_{T_1-1})\right] \\
			&\leq & \left\|\s\right\| \left[ \left\| \nabla f(\x_{T_1-1})\right\| +   \left\| \nabla m\left(\x_{T_1-1}, \s, \sigma_{T_1-1}\right)\right\|  \right] \\
			&{ \eqref{Eqn:Approx_Subprob} \above 0pt  \leq }& (1+\kappa_\theta)\left\|\s\right\| \left\| \nabla f(\x_{T_1-1})\right\| 
		\end{eqnarray*}
		where $\s=\s_{T_1-1}$ or $\s=\s^m_{T_1-1}$. 	
		Thus, we conclude that
		\begin{equation*}
		\left\|\s_{T_1-1}-\s^m_{T_1-1}  \right\| \leq \left\|\s_{T_1-1}\right\| + \left\|\s^m_{T_1-1} \right\| \leq 2\sqrt{\frac{(1+\kappa_\theta)\left\|\nabla f(\x_{T_1-1})\right\|}{\sigma_{\min}}},
		\end{equation*}
		which combines with Assumption \ref{Assumption-Objective-Gradient-Hessian} yields that
		\begin{eqnarray*}
			m(\x_{T_1-1}, \s_{T_1-1}, \sigma_{T_1-1}) - m(\x_{T_1-1}, \s^m_{T_1-1}, \sigma_{T_1-1}) & \leq & \frac{2\kappa_\theta(1+\kappa_\theta)}{\sigma_{\min}} \left\|\nabla f(\x_{T_1-1})\right\|^2 \\
			& = &  \frac{2\kappa_\theta(1+\kappa_\theta)}{\sigma_{\min}} \left\|\nabla f(\x_{T_1-1}) - \nabla f(\x^*)\right\|^2 \\
			& \leq & \frac{2\kappa_\theta (1+\kappa_\theta) L_g^2}{\sigma_{\min}}\left\| \x_{T_1-1} - \x_*\right\|^2 \\
			& = & \frac{2\kappa_\theta (1+\kappa_\theta) L_g^2}{\sigma_{\min}}\left\| \x_0 - \x_*\right\|^2.
		\end{eqnarray*}
		On the other hand, we have
		\begin{eqnarray*}
			& & m(\x_{T_1-1}, \s^m_{T_1-1}, \sigma_{T_1-1}) \\
			& = & f(\x_{T_1-1}) + (\s^m_{T_1-1})^\top \nabla f(\x_{T_1-1}) + \frac{1}{2} (\s^m_{T_1-1})^\top H(\x_{T_1-1}) \s^m_{T_1-1} + \frac{1}{3}\sigma_{T_1-1}\left\| \s^m_{T_1-1}\right\|^3 \\
			& \leq & f(\x_{T_1-1}) + (\z-\x_{T_1-1})^\top \nabla f(\x_{T_1-1}) + \frac{1}{2} (\z-\x_{T_1-1})^\top H(\x_{T_1-1})(\z-\x_{T_1-1}) + \frac{1}{3}\sigma_{T_1-1}\left\|\z-\x_{T_1-1}\right\|^3 \\
			& \leq & f(\z) + \frac{L_h}{6}  \left\|\z-\x_{T_1-1}\right\|^3  + \frac{1}{3}\sigma_{T_1-1} \left\|\z-\x_{T_1-1}\right\|^3  + \frac{1}{2}(\kappa_{e}+ \kappa_{c} )\kappa_{hs}\left\|\z-\x_{T_1-1}\right\|^3\\
			& \leq & f(\z) + \frac{L_h+\bar{\sigma}_1+(\kappa_{e}+ \kappa_{c} )\kappa_{hs}}{2}\left\|\z-\x_{T_1-1}\right\|^3 \\
			& = & f(\z) + \frac{L_h+\bar{\sigma}_1+(\kappa_{e}+ \kappa_{c} )\kappa_{hs}}{2}\left\|\z-\x_0\right\|^3,
		\end{eqnarray*}
		where the second inequality is due to \eqref{Equality-Second-Order-AARCQ} and Assumption \ref{Assumption-Objective-Gradient-Hessian}. Therefore, we conclude that
		\begin{eqnarray*}
			\psi_1(\z) & = & f(\bar{\x}_1)  + \frac{1}{6}\varsigma_1\left\| \z-\bar{\x}_1\right\|^3 \\
			&\leq & f(\z) + \frac{L_h+\bar{\sigma}_1+(\kappa_{e}+ \kappa_{c} )\kappa_{hs}}{2}\left\|\z-\x_0\right\|^3 + \frac{2\kappa_\theta (1+\kappa_\theta) L_g^2}{\sigma_{\min}}\|\x_0-\x^*\|^2+ \frac{1}{6}\varsigma_1\left\|\z-\bar{\x}_1\right\|^3 .
		\end{eqnarray*}
	\end{proof}
	
	\subsection{Proofs in Section \ref{Section:AAGD}}
	\begin{lemma}\label{Lemma:AAGD-T1}
		Letting $\bar{\sigma}_1= \max\left\{\sigma_0, \gamma_2 L_g\right\} > 0$, we have $T_1\leq 1+\frac{2}{\log\left(\gamma_1\right)}\log\left(\frac{\bar{\sigma}_1}{\sigma_{\min}}\right)$.
	\end{lemma}
	\begin{proof}
		We have
		\begin{eqnarray}
		f(\x_i + \s_i) & = & f(\x_i) + \s_i^\top \nabla f(\x_i) + \int_{0}^{1} \s_i^\top \left[\nabla f(\x_i + \tau \s_i) -  \nabla f(\x_i)\right] \ d \tau \nonumber \\
		&\le & f(\x_i) + \s_i^\top \nabla f(\x_i) + \frac{L_g}{2}  \| \s_i \|^2   \nonumber \\
		& = & m(\x_i, \s_i, \sigma_i) + \left(\frac{L_g}{2}-\frac{\sigma_i}{2}\right)\left\|\s_i\right\|^3, \label{Equality-First-Order-AAGD}
		\end{eqnarray}
		where the inequality holds true due to Assumption \ref{Assumption-Objective-Gradient-Hessian}. Therefore, we conclude that
		\begin{equation*}
		\sigma_i \geq L_g \quad \Longrightarrow \quad f(\x_i+\s_i)\leq m(\x_i, \s_i, \sigma_i),
		\end{equation*}
		which further implies that $\sigma_i<L_g$ for $i\le T_1-2$. Hence,
		\begin{equation*}
		\sigma_{T_1} \le \sigma_{T_1 -1}\le \gamma_2 \sigma_{T_1 -2} \le \gamma_2 L_g.
		\end{equation*}
		Because $\bar{\sigma}_1=\max\left\{\sigma_0, \gamma_2 L_g\right\}$, it follows from the construction of Algorithm \ref{Algorithm: AARC} that $\sigma_{\min}\leq\sigma_i$ for all iterations, and $\gamma_1\sigma_i\leq\sigma_{i+1}$ for all unsuccessful iterations. Consequently, we have
		\begin{equation*}
		\frac{\bar{\sigma}_1}{\sigma_{\min}} \geq \frac{\sigma_{T_1}}{\sigma_0} = \frac{\sigma_{T_1}}{\sigma_{T_1-1}} \cdot \prod_{j=0}^{T_1-2} \frac{\sigma_{j+1}}{\sigma_j} \geq \gamma_1^{T_1-1}\left(\frac{\sigma_{\min}}{\bar{\sigma}_1}\right),
		\end{equation*}
		and hence $T_1\leq 1+\frac{2}{\log\left(\gamma_1\right)}\log\left(\frac{\bar{\sigma}_1}{\sigma_{\min}}\right)$.
	\end{proof}
	
	\begin{lemma}\label{Lemma:AAGD-T2}
		Letting $\bar{\sigma}_2 = \max\left\{\bar{\sigma}_1, \gamma_2 L_g+\gamma_2\eta\right\}>0$, we have $T_2\leq \left(1+\frac{2}{\log(\gamma_1)}\log\left(\frac{\bar{\sigma}_2}{\sigma_{\min}}\right)\right)|\SCal|$.
	\end{lemma}
	\begin{proof}
		We have
		\begin{eqnarray*}
			\s_{T_1+j}^\top\nabla f(\y_l+\s_{T_1+j}) & = & \s_{T_1+j}^\top\left[\nabla f(\y_l + \s_{T_1+j}) - \nabla f(\y_l) \right] + \s_{T_1+j}^\top\nabla f(\y_l) \\
			& \leq & \left\| \nabla f(\y_l+\s_{T_1+j}) - \nabla f(\y_l) \right\| \left\| \s_{T_1+j}\right\| - \sigma_{T_1+j} \left\| \s_{T_1+j}\right\|^2 \\
			&  \leq  & \left(L_g - \sigma_{T_1+j}\right)\left\|\s_{T_1+j}\right\|^3,
		\end{eqnarray*}
		\noindent where the last inequality is due to Assumption \ref{Assumption-Objective-Gradient-Hessian}. Then it follows that
		\begin{equation*}
		-\frac{\s_{T_1+j}^\top\nabla f(\y_l+\s_{T_1+j})}{\left\| \s_{T_1+j}\right\|^2} \geq \sigma_{T_1+j} - L_g.
		\end{equation*}
		Therefore, we have
		\begin{equation*}
		\sigma_{T_1+j} \geq L_g+\eta \quad \Longrightarrow -\frac{\s_{T_1+j}^\top\nabla f(\y_l+\s_{T_1+j})}{\left\| \s_{T_1+j}\right\|^3}\geq\eta ,
		\end{equation*}
		which further implies that
		\begin{equation*}
		\sigma_{T_1+j+1}\leq \sigma_{T_1+j}\leq \gamma_2 \cdot \sigma_{T_1+j-1}\leq \gamma_2 \left(L_g+\eta\right),\; \forall \; j \in \SCal.
		\end{equation*}
		Therefore, the above quantity is bounded by $\bar{\sigma}_2=\max\left\{\bar{\sigma}_1,\gamma_2 L_g+\gamma_2\eta\right\}$, where $\bar{\sigma}_1$ represents an upper bound on $\sigma_{T_1}$. In addition, it follows from the construction of Algorithm \ref{Algorithm: AARCQ} that $\sigma_{\min}\leq\sigma_{T_1+j}$ for all iterations, and $\gamma_1\sigma_{T_1+j}\leq\sigma_{T_1+j+1}$ for all unsuccessful iterations. Therefore, we have
		\begin{equation*}
		\frac{\bar{\sigma}_2}{\sigma_{\min}} \geq \frac{\sigma_{T_1+T_2}}{\sigma_{T_1}} = \prod_{j\in\SCal} \frac{\sigma_{T_1+j+1}}{\sigma_{T_1+j}} \cdot \prod_{j\notin\SCal} \frac{\sigma_{T_1+j+1}}{\sigma_{T_1+j}} \geq \gamma_1^{T_2-|\SCal|}\left(\frac{\sigma_{\min}}{\bar{\sigma}_2}\right)^{|\SCal|},
		\end{equation*}
		and hence
		\begin{equation*}
		|\SCal| \le T_2\leq |\SCal|+\frac{\left(|\SCal|+1\right)}{\log\gamma_1} \log\left(\frac{\bar{\sigma}_2}{\sigma_{\min}}\right) \leq \left(1+ \frac{2}{\log\gamma_1} \log\left(\frac{\bar{\sigma}_2}{\sigma_{\min}}\right)\right)|\SCal|.
		\end{equation*}
	\end{proof}
	Before estimating the upper bound of $T_3$, i.e., the total number of the count of successfully updating $\varsigma>0$, we need to prove a few technical lemmas.
	\begin{lemma}\label{Lemma:AAGD-T3-P1}
		Let $\z_l=\argmin\limits_{\z\in\br^d} \ \psi_l(\z)$, then we have $\psi_l(\z) - \psi_l(\z_l) \geq \frac{1}{8}\varsigma_l\left\|\z-\z_l\right\|^2$.
	\end{lemma}
	\begin{proof}
		It suffices to show that
		\begin{equation*}
		\psi_l(\z) - \psi_l(\z_l) - \nabla\psi_l(\z_l)^\top (\z-\z_l) \geq \frac{1}{8}\varsigma_l\left\|\z-\z_l\right\|^2.
		\end{equation*}
		By using the fact that $\z_l=\argmin_{\z\in\br^d} \ \psi_l(\z)$ and $\nabla\psi_l(\z_l)=0$, and the strongly convexity of $\psi_l$, we obtain the desired result.
	\end{proof}
	
	\begin{lemma}\label{Lemma:General-First-Order}
		For any $\s\in\br^d$ and $\g\in\br^d$, we have
		\begin{equation*}
		\s^\top\g + \frac{1}{2}\sigma\left\|\s\right\|^2 \geq -\frac{1}{2\sigma}\left\|\g\right\|^{2}.
		\end{equation*}
	\end{lemma}
	\begin{proof}
		Denote $\s^*$ to be the minimum of $\s^\top\g + \frac{1}{2}\sigma\left\|\s\right\|^2$. Hence,
		$\g + \sigma\s^*=0$.
		Therefore, $(\s^*)^\top\g = -\sigma\left\|\s^*\right\|^2$ and $\left\|\g\right\|=\sigma\left\|\s^*\right\|$, and so
		\begin{equation*}
		(\s^*)^\top\g + \frac{1}{2}\sigma\left\|\s^*\right\|^2 = -\frac{1}{2}\sigma\left\|\s^*\right\|^2 = -\frac{1}{2\sigma}\left\|\g\right\|^{2}.
		\end{equation*}
	\end{proof}

	\begin{lemma}\label{Lemma:AAGD-T3-P2}
		For each iteration $j$ in the subroutine \textsf{AAS}, if it is a successful iteration, then we have
		\begin{equation*}
		\left\|\nabla f(\x_{j+1})\right\| \leq ( L_g+\bar{\sigma}_2)\left\|\s_j\right\|.
		\end{equation*}
	\end{lemma}
	\begin{proof}
		We denote $j$-th iteration is the $l$-th successful iteration, and note $\nabla_{\s} m(\y_l,\s_j,\sigma_j)=\nabla f(\y_l) + \sigma_j\s_j$. Then we have
		\begin{eqnarray*}
			\left\|\nabla f(\x_{j+1})\right\| & = & \left\| \nabla f(\y_l+\s_j) - \nabla_{\s} m(\y_l, \s_j, \sigma_j)\right\| \\
			& \leq & \left\| \nabla f(\y_l+\s_j)-\nabla f(\y_l)\right\| + \sigma_j\left\| \s_j \right\| \\
			& \leq &  L_g\left\| \s_j \right\| + \sigma_j\left\| \s_j\right\| \\
			& \leq & (L_g+\bar{\sigma}_2)\left\|\s_j\right\|
		\end{eqnarray*}
		where the second inequality follow from Assumption \ref{Assumption-Objective-Gradient-Hessian}. Rearranging the terms, the conclusion follows.
	\end{proof}
	
	Now we are ready to estimate the upper bound of $T_3$, i.e., the total number of the count of successfully updating $\varsigma>0$.
	\begin{lemma}\label{Lemma:AAGD-T3}
		We have
		\begin{equation}\label{GInequality-Induction}
		\psi_l(\z_l) \ge \frac{l(l+1)}{2} f(\bar{\x}_l)
		\end{equation}
		if $\varsigma_l \ge \left(2L_g+2\bar{\sigma}_2\right)^{2}\frac{1}{\eta}$,  which further implies that
		\begin{equation*}
		T_3 \le \left\lceil \frac{1}{\log\left(\gamma_3\right)}\log\left[ \left(2L_g+2\bar{\sigma}_2\right)^{2}\frac{1}{\eta \, \varsigma_1} \right] \right\rceil.
		\end{equation*}.
	\end{lemma}
	\begin{proof}
		When $l=1$, it trivially holds true that $\psi_l(\z_l)\ge\frac{l(l+1)}{2} f(\bar{\x}_l)$ since $\psi_1(\z_1)=f(\bar{\x}_1)$. As a result, it suffices to show that $\varsigma_l \ge \left(2L_g+2\bar{\sigma}_2\right)^{2}\frac{1}{\eta}$ by mathematical induction. Without loss of generality, we assume \eqref{GInequality-Induction} holds true for some $l-1 \ge 1$. Then, it follows from Lemma \ref{Lemma:AAGD-T3-P1}, the construction of $\psi_l(\z)$ and our induction that
		\begin{equation*}
		\psi_{l-1}(\z) \geq \psi_{l-1}(\z_{l-1}) + \frac{1}{8}\varsigma_{l-1}\left\|\z-\z_l\right\|^2 \geq \frac{(l-1)l}{2} f(\bar{\x}_{l-1}) + \frac{1}{8}\varsigma_{l-1}\left\|\z-\z_{l-1}\right\|^2.
		\end{equation*}
		As a result, we have
		\begin{eqnarray*}
			& & \psi_l(\z_l) \\
			& = & \min_{\z\in\br^d} \ \left\{\psi_l(\z) + l\left[ f(\bar{\x}_l)+\left(\z-\bar{\x}_l\right)^\top\nabla f(\bar{\x}_l)\right] + \frac{1}{4}\left(\varsigma_l - \varsigma_{l-1}\right)\|\z- \bar{\x}_1 \|^2 \right\} \\
			& \geq & \min_{\z\in\br^d} \ \left\{\frac{(l-1)l}{2} f(\bar{\x}_{l-1}) + \frac{1}{8}\varsigma_l\left\|\z - \z_{l-1}\right\|^2 + l\left[ f(\bar{\x}_l)+ \left(\z-\bar{\x}_l\right)^\top \nabla f(\bar{\x}_l)\right]  \right\} \\
			& \geq & \min_{\z\in\br^d} \ \left\{\frac{(l-1)l}{2} \left[ f(\bar{\x}_l) + \left(\bar{\x}_{l-1}-\bar{x}_l\right)^\top\nabla f(\bar{\x}_l)\right] + \frac{1}{8} \varsigma_l\left\|\z - \z_{l-1}\right\|^2\right. \\
			& & \quad \quad \quad \left. + l\left[ f(\bar{\x}_l)+\left(\z -\bar{\x}_l\right)^\top\nabla f(\bar{\x}_l)\right]  \right\} \\
			& = & \frac{l(l+1)}{2} f(\bar{\x}_l) + \min_{\z\in\br^d} \ \left\{ \frac{(l-1)l}{2}\left(\bar{\x}_{l-1}-\bar{\x}_l\right)^\top\nabla f(\bar{\x}_l) + \frac{1}{8}\varsigma_l\left\|\z-\z_{l-1}\right\|^2 \right.\\
			& &  \quad \quad \quad \left.+ l\left(\z-\bar{\x}_l\right)^\top\nabla f(\bar{\x}_l)\right\}.
		\end{eqnarray*}
		where the first inequality holds true because $\varsigma_l \ge \varsigma_{l-1}$. By the construction of $\y_{l-1}$, one has
		\begin{eqnarray*}
			\frac{(l-1)l}{2}\bar{\x}_{l-1} & = & \frac{l(l+1)}{2}\cdot\frac{l-1}{l+1}\bar{\x}_{l-1} \\
			& = & \frac{l(l+1)}{2}\left(\y_{l-1}-\frac{2}{l+1}\z_{l-1}\right) \\
			& = & \frac{l(l+1)}{2}\y_{l-1} - l\z_{l-1}.
		\end{eqnarray*}
		Combining the above two formulas yields
		\begin{equation*}
		\psi_l(\z_l) \geq \frac{l(l+1)}{2} f(\bar{\x}_l) + \min_{\nu\in\br^d} \ \left\{ \frac{l(l+1)}{2} \left(\y_{l-1}-\bar{\x}_l\right)^\top\nabla f(\bar{\x}_l) + \frac{1}{8}\varsigma_l\left\|\z- \z_{l-1}\right\|^2 + l\left(\z-\z_{l-1}\right)^\top\nabla f(\bar{\x}_l)\right\}.
		\end{equation*}
		Then, by the criterion of successful iteration in \textsf{AAS} and Lemma \ref{Lemma:AAGD-T3-P2}, we have
		\begin{eqnarray*}
			\left(\y_{l-1}-\bar{\x}_l\right)^\top\nabla f(\bar{\x}_l) & = & -\s_{T_1+j}^\top \nabla f(\y_{l-1}+\s_{T_1+j}) \\
			& \geq & \eta\left\| \s_{T_1+j} \right\|^2 \geq \eta \left(\frac{1}{ L_g+\bar{\sigma}_2}\right)^{2}\left\|  \nabla f(\bar{\x}_l)\right\|^{2},
		\end{eqnarray*}
		where the $l$-th successful iteration count refers to the $(j-1)$-th iteration count in \textsf{AAS}. Hence, it suffices to establish
		\begin{equation*}
		\frac{l(l+1)\eta}{2}\left(\frac{1}{ L_g+\bar{\sigma}_2}\right)^{2}\left\|  \nabla f(\bar{\x}_l)\right\|^{2} + \frac{1}{8}\varsigma_l\left\|\z-\z_{l-1}\right\|^2 + l\left(\z-\z_{l-1}\right)^\top\nabla f(\bar{\x}_l) \geq 0.
		\end{equation*}
		Using Lemma \ref{Lemma:General-First-Order} and setting $\g=l\nabla f(\bar{\x}_l)$ and $\sigma=\frac{1}{4}\varsigma_l$, the above is implied by
		\begin{equation*}
		\frac{l(l+1)\eta}{2}\left(\frac{1}{ L_g+\bar{\sigma}_2}\right)^{2} \geq \frac{2}{\varsigma_{l}} l^{2}.
		\end{equation*}
		Therefore, the conclusion follows if
		$\varsigma_l \ge\left(2L_g+2\bar{\sigma}_2\right)^{2}\frac{1}{\eta}$.
	\end{proof}
	
	Finally we are in a position to prove the base case of $l =1$ for Theorem \ref{Theorem:AAGD-Main}.
	\begin{theorem}\label{Theorem:AAGD-Main-Prep}
		It holds that
		\begin{equation*}
		f(\bar{\x}_1) \leq \psi_1(\z_1) \leq \psi_1(\z) \leq f(\z) + \frac{L_g + \bar{\sigma}_1} {2}\left\|\z-\x_0\right\|^2 + \frac{1}{4}\varsigma_1\left\| \z-\bar{\x}_1\right\|^2 .
		\end{equation*}
	\end{theorem}
	
	\begin{proof}
		By the definition of $\psi_1(\z)$ and the fact that $\bar{\x}_1=\x_{T_1}$, we have
		\begin{equation*}
		f(\bar{\x}_1) = f(\x_{T_1}) =\psi_1(\z_1).
		\end{equation*}
		Furthermore, by the criterion of successful iteration in \textsf{SAS},
		\begin{eqnarray*}
			f(\bar{\x}_1) & = & f(\x_{T_1}) \\
			& \leq & m(\x_{T_1-1}, \s_{T_1-1}, \sigma_{T_1-1}) \\
			& = & f(\x_{T_1-1}) + \s_{T_1-1}^\top \nabla f(\x_{T_1-1}) + \frac{\sigma_{T_1-1}}{2}\left\|\s_{T_1-1}\right\|^2 \\
			& \leq & f(\x_{T_1-1}) + (\z-\x_{T_1-1})^\top \nabla f(\x_{T_1-1}) + \frac{1}{2}\sigma_{T_1-1}\left\| \z-\x_{T_1-1}\right\|^2 \\
			& \leq & f(\z) + \frac{L_g}{2}\left\|\z-\x_{T_1-1}\right\|^2 + \frac{1}{2}\sigma_{T_1-1} \left\|\z-\x_{T_1-1}\right\|^2 \\
			& \leq & f(\z) + \frac{L_g + \bar{\sigma}_1}{2}\left\|\z-\x_{T_1-1}\right\|^3 \\
			& = & f(\z) + \frac{L_g + \bar{\sigma}_1}{2}\left\| \z-\x_0\right\|^2,
		\end{eqnarray*}
		where the third inequality is due to Assumption \ref{Assumption-Objective-Gradient-Hessian}. Therefore, we conclude that
		\begin{eqnarray*}
			\psi_1(\z) & = & f(\bar{\x}_1)  + \frac{1}{4}\varsigma_1\left\| \z-\bar{\x}_1\right\|^2 \\
			&\leq & f(\z) + \frac{L_g+\bar{\sigma}_1}{2}\left\|\z-\x_0\right\|^3 + \frac{1}{4}\varsigma_1\left\|\z-\bar{\x}_1\right\|^2 .
		\end{eqnarray*}
	\end{proof}

\end{document}